\documentclass[11pt]{article}
\usepackage{mystyle}

\title{Equivariant Bifurcation from Relative Equilibria via Isomorphic Vector Fields}
\author{Stef Klajbor-Goderich}
\begin{document}

\maketitle

\begin{abstract}
We present a framework for studying the dynamics of equivariant vector fields near relative equilibria.
To overcome the lack of linearization at a relative equilibrium or the possible non-smoothness of the orbit space, we categorify the space of equivariant vector fields.
A category where the objects are equivariant vector fields was first introduced by Hepworth in the context of smooth stacks \cite{H09}.
Central to our approach is the ensuing notion of isomorphic equivariant vector fields.
The idea is that considering equivariant vector fields, and their corresponding dynamics, up to isomorphism is a way to take into account the symmetries of the group action without passing to the orbit space.
In particular, the category of equivariant vector fields near a relative equilibrium is equivalent to the category of equivariant vector fields on the slice representation of the relative equilibrium.
In this paper we apply this to bifurcations to and from relative equilibria, and to the genericity of conditions for equivariant bifurcation from relative equilibria.
\end{abstract}

\setcounter{tocdepth}{2}
\tableofcontents

%%%%%%%%%%%%%%%%%%%%%%%%%%%%%%%%
\section{Introduction}\label{Introduction}
We present a framework for studying the dynamics of equivariant vector fields near relative equilibria.
In particular, we study bifurcations to and from relative equilibria, and the generic conditions for equivariant bifurcation from relative equilibria.
Previously, the author used this framework to study stability of relative equilibria \cite{K18}.

The usual approaches used for vector fields near strict equilibria run into difficulties in the case of relative equilibria.
For example, it doesn't make sense to directly linearize a vector field at a relative equilibrium.
Since relative equilibria descend to equilibria of the flow on the orbit space, one could try to linearize on the orbit space \cite{Koe97,Ch02}.
The obstacle is that orbit spaces of group actions are generally not smooth.
A brute force approach is to embed the orbit space in some Euclidean space $\bbr^n$.
However, this can be difficult in practice.
Instead, we prefer to think of the orbit space as a smooth stack (see, for example, \cite{L10}).
For our purposes, this amounts to categorifying the space of equivariant vector fields.

A category where the objects are equivariant vector fields was first introduced by Hepworth in the context of smooth stacks \cite{H09}.
Given an action of a group $G$ on a manifold $M$, we categorify the space $\ffX(M)^G$ of $G$-equivariant vector fields by introducing an action on $\ffX(M)^G$ of the vector space:
\begin{equation*}
C^\infty\left(M,\ffg\right)^G:=
\left\{
\psi:M\to \ffg
\mid
\psi(g\cdot m) = \Ad(g)\psi(m),\,\,
g\in G,\, m\in M
\right\},
\end{equation*}
where $\ffg$ is the Lie algebra of $G$ (see (\ref{Eq:Action}) and Definition \ref{Def:GpoidVFs}).
Central to our approach is the ensuing notion of isomorphic equivariant vector fields (Definition \ref{Def:IsoVF}).
Isomorphic vector fields lead to equivalent continuous flows on the orbit space.
The idea is that considering equivariant vector fields, and their corresponding dynamics, {\it up to isomorphism} is a way to take into account the symmetries of the group action without passing to the orbit space.

b
In section \ref{ch:category}, we show that the category of equivariant vector fields in an invariant neighborhood of a point is equivalent to the category of equivariant vector fields on the canonical slice representation near the point (Theorem \ref{Thm:EquivalenceVFs}).
Consequently, given an equivariant vector field $X\in\ffX(M)^G$ with a relative equilibrium at a point $m\in M$, we obtain a decomposition:
\begin{equation}\label{Eq:DecompositionIntro}
X=Y^X+\partial(\psi^X),
\end{equation}
where $Y^X$ has an equilibrium at $m$ and is transverse to the group orbits near $m$, and $\partial(\psi^X)$ is an equivariant vector field built out of an isomorphism of equivariant vector fields natural in $X$ (Theorem \ref{Thm:Decomposition}).
This decomposition makes many of the methods used for equilibria available for the case of relative equilibria.
Similar decompositions to (\ref{Eq:DecompositionIntro}) have been used before, notably using invariant Riemannian metrics \cite{K90,FSSW96} and lifting the vector field to a skew-product decomposition \cite{FSSW96}.
However, we believe the decomposition (\ref{Eq:DecompositionIntro}) in terms of isomorphisms of equivariant vector fields is natural and conceptually simple.
Additionally it also helps address the effect of the choices involved (Proposition \ref{Prop:ProjectionChoice} and Proposition \ref{Prop:SliceChoice}).
In particular, the choices involved lead to transversal vector fields that are isomorphic.

In preparation for the results on bifurcation, we discuss the motion of relative equilibria on a $K$-manifold $M$, where $K$ is compact, from the point of view of the category of equivariant vector fields.
It is well-known that, the motion of a relative equilibrium of an equivariant vector field on $M$ is equivalent to linear motion on a torus \cite{F80,K90}.
In fact, there is a bound on the number of independent frequencies of the motion; that is, on the dimension of the torus containing the motion \cite{F80,K90}.
We show how isomorphisms of equivariant vector fields are related to these frequencies, which will be useful in section \ref{ch2}.
Furthermore, this bound is attained generically, but one can hope to modify the equivariant vector field to reduce, or otherwise adjust, the number of independent frequencies of the relative equilibrium's motion to obtain nongeneric motions.

Given an equivariant vector field with a relative equilibrium on a compact $K$-manifold, we provide conditions for constucting an isomorphic vector field that has any desired number of independent frequencies at the relative equilibrium (Proposition \ref{Prop:FrequencyStabilizationConditions}).
Since the resulting vector field is isomorphic to the given one, it determines the same flow on the orbit space, and hence the same dynamics modulo the symmetries (Theorem \ref{Lemma:LermanFlows}).
In particular, we show that this is always possible for actions of tori (Theorem \ref{Thm:jStabilization}).
The results in this section may also be of interest in the control of equivariant dynamical systems.

We then proceed to apply the category of equivariant vector fields point of view to bifurcations to and from relative equilibria in section \ref{ch2}.
It is well known that, due to the presence of group symmetries, one expects bifurcations to relative equilibria in place of bifurcations to equilibria.
We prove a test for bifurcations to relative equilibria on representations (Theorem \ref{Thm:BifToRelEq}).
On the one hand, this test is conceptually simple: it is essentially a generalization of the Equivariant Branching Lemma \cite{C81,V82} considered up to isomorphism of equivariant vector fields.
On the other hand, it can be quite general: it can predict bifurcations to steady-state, periodic, or quasi-periodic motion on tori.
The motion on the bifurcating branches depends on an isomorphism of equivariant vector fields used in the test.
In particular, isomorphisms of equivariant vector fields are central to reducing and reconstructing the dynamics of the bifurcation in our test.

We then extend this test to bifurcations {\it from} relative equilibria for proper actions of (potentially noncompact) Lie groups (Theorem \ref{Thm:BifToRelEqProper}).
We reduce to the slice representation at the relative equilibrium.
In fact, we prove something stronger: that bifurcations to relative equilibria are in correspondence with bifurcations to relative equilibria on the slice representation (Theorem \ref{Thm:BifBranchEquivalence}), and how isomorphisms of equivariant vector fields help relate the velocities of the corresponding relative equilibria.
Compared to similar slice reductions that can be found in the literature \cite{K90, FSSW96} we reconstruct the dynamics from the slice reduction using isomorphisms of equivariant vector fields.
This gives finely-tuned control of the reconstruction, and a clear way to track the effect of choices.

It is common in equivariant bifurcation to talk about {\it generic conditions} for equivariant vector fields or for paths of equivariant vector fields.
For example, a standard bifurcation condition is the eigenvalue crossing condition: that the eigenvalues of the linearization of a path of equivariant vector fields crosses the imaginary axis with nonzero velocity (see, for example, the Equivariant Branching Lemma \cite{C81,V82} and Theorem \ref{Thm:BifToRelEq}).
One way to formalize the notion of a condition being {\it generic} in equivariant dynamics is to endow the space of equivariant vector fields with a topology and consider whether there are {\it open and dense} subsets satisfying the desired condition.

In section \ref{ch:residual}, we consider open and dense subsets of equivariant vector fields in light of the category of equivariant vector fields.
The first main result of this section is that open and dense subsets of equivariant vector fields are ``preserved'' by isomorphisms of equivariant vector fields (Theorem \ref{Thm:MainResidual1}).
That is, the set of all equivariant vector fields isomorphic to those in an open and dense subset of equivariant vector fields is also open and dense in the space of equivariant vector fields.
The second main result of this section is that the equivalence in Theorem \ref{Thm:EquivalenceVFs} ``preserves'' open and dense subsets of equivariant vector fields up to isomorphism (Theorem \ref{Thm:MainResidual2}).
That is, in particular, the reduction to the slice representation via equivariant projection and the reconstruction of the dynamics via equivariant extension from this slice preserve open and dense subsets of equivariant vector fields \textit{up to isomorphism}.

Importantly, the equivariant projection and extension in decomposition (\ref{Eq:DecompositionIntro}) don't need to strictly preserve open and dense subsets.
In particular, the equivariant extension of an open and dense subset of vector fields on the slice representation doesn't need to be open and dense.
For this, note the equivariantly extended vector fields are all vertical in the corresponding associated bundle over the group orbit of the relative equilibrium, which is not an open and dense condition (see the discussion preceding (\ref{Eq:Decomposition})).
The result that such collections are preserved \textit{up to isomorphism} addresses this issue.

These theorems also apply to paths of equivariant vector fields, and so they also apply in the context of equivariant bifurcation.
That is, generic conditions for equivariant bifurcation problems from relative equilibria, like those in Theorem \ref{Thm:BifToRelEqProper}, correspond to generic conditions for equivariant bifurcation problems from equilibria.

In order to prove Theorems \ref{Thm:MainResidual1} and \ref{Thm:MainResidual2}, we endow the spaces of equivariant vector fields and of paths of equivariant vector fields with the Whitney $C^\infty$ topologies (section \ref{whitney}).
While these spaces are vector spaces, they are {\it not} topological vector spaces when endowed with these topologies.
The problem is the scalar multiplication fails to be continuous.
Nevertheless, the addition is continuous and hence they are {\it topological abelian groups}.
If we don't consider the topology, the category $\bbx(M)^G$ of equivariant vector fields on a $G$-manifold is a $2$-vector space (Remark \ref{Rem:2VectorSpaceStr}).
That is, it is a category internal to the category of vector spaces: its space of objects and morphisms are vector spaces, and all the structure maps are linear.
When endowed with a topology, the category $\bbx(M)^G$ becomes a topological abelian $2$-group.
That is, it becomes a category internal to the category of topological abelian groups: its space of objects and morphisms are topological abelian groups, and all the structure maps are continuous group homomorphisms.
This topological abelian $2$-group structure proves to be the necessary key for proving the main results of this section, so we discuss it in section \ref{top2groups}.

To summarize, section \ref{ch:category} defines a category of equivariant vector fields and uses this to decompose equivariant vector fields near relative equilibria.
In section \ref{motionsection}, we discuss the motion of relative equilibria from the point of view of the category of equivariant vector fields.
In section \ref{ch2}, we use isomorphisms of equivariant vector fields to study bifurcations to and from relative equilibria.
Finally, in section \ref{ch:residual}, we study open and dense collections of equivariant vector fields and paths of equivariant vector fields.

\subsection{Notation}\label{Notation}
In this brief section we give an overview of the conventions and notation that will be used throughout this paper.

\begin{enumerate}
\item We will denote the {\it tangent map} of a smooth $f:M\to N$ between smooth manifolds $M$ and $N$ by $Tf:TM\to TN$.

\item The vector space of {\it smooth vector fields} on a manifold $M$ will be denoted by $\ffX(M)$.

\item Given a diffeomorphism $f:M\to N$ between two manifolds, we will denote the corresponding {\it pushforward of vector fields along $f$} by $f_*:\ffX(M)\to\ffX(N)$ and the {\it pullback of vector fields along $f$} by $f^*:\ffX(N)\to \ffX(M)$.

\item We will denote {\it Lie groups} with uppercase Latin letters, their {\it Lie algebras} with the corresponding lowercase fraktur letter, and the {\it duals of these Lie algebras} by adding a star superscript.
For example, a Lie group $G$ will have Lie algebra $\ffg$ and dual Lie algebra $\ffg^*$.

\item The {\it adjoint representation} of a Lie group on its Lie algebra will be denoted by $\Ad$, while its {\it coadjoint representation} on the dual of the Lie algebra will be denoted by $\CoAd$.

\item The {\it action} of a group $G$ on a space $M$ will be denoted by $g\cdot m$ for all $g\in G$ and all $m\in M$.
Given $g\in G$, we denote the corresponding translation map by:
\begin{equation*}
g_M:M\to M, \qquad g_M(m):=g\cdot m.
\end{equation*}

\item Given an action of a group $G$ on a space $M$, the {\it stabilizer subgroup} of a point $m\in M$ will be denoted by the same letter as the group but with the point as a subscript (e.g. $G_m$).
If $G$ is a Lie group and $M$ is a smooth manifold, the {\it Lie algebra of the stabilizer} will also carry the point as a subscript (e.g. $\ffg_m$).

\item \label{Not:TangentAction} If $G$ is a Lie group, $M$ is a smooth manifold, and $g\in G$, then the {\it translation map} $g_M:M\to M$ is a diffeomorphism.
Furthermore, there is a canonical action of $G$ on the tangent bundle $TM$, given by:
\begin{equation*}
g\cdot X := Tg_M(X), \qquad g\in G,\, X\in TM.
\end{equation*}

\item \label{Not:EquivariantInvariant} Given actions of a Lie group $G$ on manifolds $M$ and $N$, an {\it equivariant map } $f:M\to N$ is a smooth map such that:
\begin{equation*}
f(g\cdot m) = g\cdot f(m), \qquad g\in G,\, m\in M.
\end{equation*}
An {\it invariant map} $f:M\to N$ is a smooth map such that:
\begin{equation*}
f(g\cdot m) = f(m), \qquad g\in G,\, m\in M.
\end{equation*}
Equivalently, an invariant map is an equivariant map where the action of $G$ on $N$ is the trivial action.
We denote the {\it space of equivariant maps} $M\to N$ by $C^\infty(M,N)^G$.
If to avoid ambiguity, we need to distinguish between equivariant and invariant maps, we will denote the space of invariant maps by $C^\infty(M,N)^{G-\text{inv}}$ and reserve $C^\infty(M,N)^G$ for equivariant maps in such a case.
Otherwise, we will explicitly say which one we mean or context will make it clear.
%If $V$ and $W$ are representations of a Lie group $K$, then the space of equivariant polynomial maps $p:V\to W$ is denoted by $P(V,W)^K$.

\item \label{Not:VerticalBundle}Given a smooth fiber bundle $\pi:P\to B$, the corresponding {\it vertical bundle} is the bundle over the manifold $P$ with total space $\calv P:=\ker\d\pi$.
The projection $\calv P\to P$ is the restriction of the tangent bundle projection $TP\to P$, and hence the vertical bundle is a subbundle of the tangent bundle.

\item \label{Not:AssociatedBundle} We will also make use of associated bundles.
Given a Lie group $K$, a manifold  $P$ with a free and proper right action of $K$, and a manifold $F$ with a proper left action of $K$, the {\it associated bundle} is the bundle with total space the quotient $P\times^K F:=(P\times F)/K$ of the action:
\begin{equation*}
k\cdot (p,f):= (p\cdot k^{-1},k\cdot f), \qquad k\in K, \, (p,f)\in P\times F.
\end{equation*}
The base of the associated bundle is the space $P/K$ and the typical fiber is $F$.
Thus, if $F$ is a vector space then the associated bundle $P\times^KF$ is a vector bundle over $G/K$.
We will denote the elements of $P\times^KF$ by $[p,f]$.
If the manifold $F$ is a product of the form $M\times N$, we will denote the elements of $P\times^KF$ by $[p,m,n]$ instead of $[p,(m,n)]$.
Furthermore, suppose $K$ is a compact Lie subgroup of a Lie group $G$.
The $P:=G$ is a right $K$-principal bundle over $G/K$ with the action by right-multiplication.
Then there is an action of $G$ on the associated bundle $G\times^KF$ given by:
\begin{equation*}
g'\cdot [g,f]:=[g'g,f], \qquad g'\in G,\, [g,f]\in G\times^KF.
\end{equation*}

\item \label{Not:Categories} Given a category $\calc$, we will denote the {\it collection of objects} by $\calc_0$ and the {\it collection of morphisms} by $\calc_1$.
A {\it morphism} $f$ between two objects $X$ and $Y$ in $\calc_0$ will be denoted by $f:X\to Y$.
Given a {\it functor} $F:\calc\to\cald$ between two categories $\calc$ and $\cald$, we will denote the corresponding map on objects by $F_0:\calc_0\to\cald_0$ and the corresponding map on morphisms by $F_1:\calc_1\to\cald_1$.
We will denote a natural isomorphism between two functors $F:\calc\to\cald$ and $G:\calc\to\cald$ by $h:F\cong G$ or alternatively by its corresponding map $h:\calc_0\to\cald_1$ assigning each object $X$ in $\calc_0$ to the corresponding morphism $h(X):F_0(X)\to G_0(X)$.
By an {\it equivalence of categories } between two categories $\calc$ and $\cald$ we will mean a pair of functors $F:\calc\to\cald$ and $G:\cald\to\calc$ together with natural isomorphisms $\epsilon:GF \cong 1_{\calc}$ and $\mu:FG \cong 1_{\cald}$.
\end{enumerate}

\subsection{Acknowledgements}\label{Acknowledgements}
The author would like to thank their PhD advisor, Eugene Lerman, who provided valuable guidance on this work while the author was completing their PhD dissertation.

%%%%%%%%%%%%%%%%%%%%%%%%%%%%%%%%
\section{Categories and isomorphisms of equivariant vector fields}\label{ch:category}
A category where the objects are equivariant vector fields was first introduced by Hepworth in the context of smooth stacks \cite{H09}.
In this paper, we explore applications of this category to equivariant dynamics.
The main idea is that considering equivariant vector fields, and their corresponding dynamics, {\it up to isomorphism} is a way to take into account the symmetries of the group action.
Besides introducing the category of equivariant vector fields (Definition \ref{Def:GpoidVFs}), in this section we show that the category of equivariant vector fields in an invariant neighborhood of a point is equivalent to the category of equivariant vector fields on the canonical slice representation near the point (Theorem \ref{Thm:EquivalenceVFs}).
Consequently, we obtain a decomposition of equivariant vector fields near relative equilibria into a component transverse to group orbits near the relative equilibrium and a component tangent to them (Theorem \ref{Thm:Decomposition}).

\subsection{Categories of equivariant vector fields}\label{categoryVFs}
In this section we define the category of equivariant vector fields and provide several examples.
First, recall:

\begin{definition}
A {\it $G$-manifold} $M$ is a smooth manifold $M$ with a smooth action of a Lie group $G$.
A {\it proper $G$-manifold} is one where the action is proper.
\end{definition}

\begin{definition}\label{Def:EquivVF}
An {\it equivariant vector field} on a $G$-manifold $M$ is a smooth vector field $X:M\to TM$ such that:
\begin{equation*}
X(g\cdot m)=g\cdot X(m),
\end{equation*}
for all $g\in G$ and $m\in M$, where the action on the right is as in Notation \ref{Notation}.\ref{Not:TangentAction}.
\end{definition}

\begin{notation}
We will denote the vector space of equivariant vector fields on a $G$-manifold $M$  by $\ffX(M)^G$.
\end{notation}

Morphisms between equivariant vector fields will be built out of the following class of maps:

\begin{definition}\label{Def:InfGaugeTransf}
An {\it infinitesimal gauge transformation} on a $G$-manifold $M$ is an equivariant smooth map $\psi:M\to\ffg$, where $\ffg$ is the Lie algebra of $G$.
That is, 
\begin{equation*}
\psi(g\cdot m)= \Ad(g)\psi(m),
\end{equation*}
for all $g\in G$ and $m\in M$, where $\Ad$ is the adjoint representation.
\end{definition}

\begin{notation}
We will denote the space of infinitesimal gauge transformations by $C^\infty(M,\ffg)^G$.
\end{notation}

\begin{remark}\label{Rem:InfGaugeTransf}
If the action of $G$ on $M$ is free and proper, then the orbit space $M/G$ is a~mani\-fold and the orbit space map $M\to M/G$ is a principal $G$-bundle. In this case, the space of infinitesimal gauge transformations $C^\infty(M,\ffg)^G$ is isomorphic to the space of smooth sections of the associated bundle $M\times^G\ffg\to M/G$ (see Notation \ref{Notation}.\ref{Not:AssociatedBundle}, with the action of $G$ on $\ffg$ being the $\Ad$ representation).
\end{remark}

An infinitesimal gauge transformation $\psi:M\to\ffg$ on a $G$-manifold $M$ induces an equivariant vector field $\partial(\psi):M\to TM$.
It is given by:
\begin{equation}\label{Eq:InducedVF}
\partial(\psi)(m):=\frac{\d}{\d \tau}\Big|_0 \exp(\tau\psi(m))\cdot m,
\end{equation}
for any $m\in M$.
We verify this is indeed equivariant:

\begin{lemma}\label{Lemma:InvPsiM}
Let $M$ be a $G$-manifold and let $\psi\colon M\to\ffg$ be an infinitesimal gauge transformation on $M$.
The induced vector field $\partial(\psi)$ defined by (\ref{Eq:InducedVF}) is an equivariant vector field with respect to the action of $G$.
\end{lemma}

\begin{proof}
This is a consequence of the naturality of the exponential. Let $g\in G$ and $m\in M$, then
\begin{align*}
\partial(\psi)(g\cdot m)
&=\frac{\d}{\d t}\Big|_0 \exp(t\psi(g\cdot m))\cdot g\cdot m \\
&=\frac{\d}{\d t}\Big|_0 \exp(t\Ad(g)\psi(m))\cdot g\cdot m \qquad \text{ by the equivariance of }\psi \\
&=\frac{\d}{\d t}\Big|_0 \left(g\exp(t\psi(m))g^{-1}\right)\cdot g\cdot m \qquad \text{ by the naturality of }\exp \\
&=g\cdot \partial(\psi)(m).
\end{align*}
Hence, $\partial(\psi)$ is an equivariant vector field.
\end{proof}

Given a $G$-manifold $M$, the map $\partial:C^\infty(M,\ffg)^G\to\ffX(M)^G$ assigning to each infinitesimal gauge transformation $\psi$ the vector field $\partial(\psi)$ as in (\ref{Eq:InducedVF}) is a linear map.
Consequently, the abelian group $C^\infty(M,\ffg)^G$ acts on the space $\ffX(M)^G$.
The action is given by:
\begin{equation}\label{Eq:Action}
\psi\cdot X := X + \partial(\psi),
\end{equation}
where $\psi$ is an infinitesimal gauge transformation, $X$ is an equivariant vector field, and the addition is the pointwise addition of vector fields.

\begin{remark}\label{Rem:ActionGroupoids}
Recall that the action of a group $H$ on a space $\caly$ defines an action groupoid $H\times\caly\toto\caly$ (see, for example, \cite[Example~5.1~(5)]{MM03}).
The objects of the action groupoid in (\ref{Eq:Action}) are the equivariant vector fields, while morphisms are pairs $(\psi,X)$ consisting of a gauge transformation $\psi$ and an equivariant vector field $X$.
The source map is the projection onto the second factor, the target map is the action map, and the composition corresponds to addition of the first factors, i.e. addition of infinitesimal gauge transformations.
\end{remark}

\begin{definition}\label{Def:GpoidVFs}
The {\it category } $\bbx(M)^G$ {\it of equivariant vector fields} on a $G$-manifold $M$ is the action groupoid (see Remark \ref{Rem:ActionGroupoids}) of the action of the space of infinitesimal gauge transformations $C^\infty(M,\ffg)^G$ on the space of equivariant vector fields $\ffX(M)^G$.
\end{definition}

We isolate what it means for two equivariant vector fields to be isomorhic in this category:

\begin{definition}\label{Def:IsoVF}
Two equivariant vector fields $X$ and $Y$ on a $G$-manifold $M$ are {\it isomorphic} if they are isomorphic as objects of the groupoid of equivariant vector fields.
That is, they are isomorphic if there exists an infinitesimal gauge transformation $\psi:M\to \ffg$ such that $Y=X+\partial(\psi)$.
\end{definition}

Recall that the flow of an equivariant vector field on a $G$-manifold $M$ descends to give a continuous flow on the orbit space $M/G$.
The following result has as corollary that isomorphic vector fields descend to the same continuous flow on the orbit space:

\begin{theorem}[Lerman]\label{Lemma:LermanFlows}
Let $X$ and $Y$ be two isomorphic equivariant vector fields on a $G$-manifold $M$ with flow $\phi^X$ and $\phi^Y$ respectively.
Let $\calo\subseteq\bbr\times M$ denote the domain of the flow $\phi^X$.
Then $\calo$ is also the domain of $\phi^Y$.
Furthermore, there exists a smooth map $F:\calo\to G$ such that $F(0,n)$ is the identity of $G$ for all $m\in M$, and:
\begin{equation*}
\phi^Y(\tau,m)=F(\tau,m)\cdot\phi^X(\tau,m)
\end{equation*}
for all $(\tau,m)\in\calo$.
In particular, for each $m\in M$, the curve $F(-,m)$ on $G$ is the unique solution to the ODE initial value problem:
\begin{equation*}
\begin{cases}
\frac{\partial F}{\partial\tau} (\tau,m) = TL_{F(\tau,m)}\left(\psi(\phi^X(\tau,m))\right) \\
F(0,m)=1_G
\end{cases}
\end{equation*}
where $1_G$ is the identity in $G$, and $TL_{F(\tau,m)}$ is the tangent map at the identity $1_G$ of the left translation $L_{F(\tau,m)}:G\to G$ defined by $g\mapsto F(\tau,m)g$.
\end{theorem}

\begin{proof}
See \cite[Theorem~1.6]{L15}.
\end{proof}

\begin{remark}\label{Rem:Origins}
In \cite{H09} Hepworth defined vector fields on stacks.
According to his definition, vector fields on a stack $\caly$ are objects of a category $\Vect(\caly)$.
In the case where the stack $\caly$ is the stack quotient $[M/G]$ of a $G$-manifold $M$, where $G$ is a compact Lie group, the category $\Vect(\caly)$ is equivalent to the corresponding groupoid of equivariant vector fields of Definition \ref{Def:GpoidVFs} \cite[Proposition 6.1]{H09}.
This groupoid is further explored in \cite{L15} and \cite{K18}.
\end{remark}

\begin{remark}\label{Rem:2VectorSpaceStr}
The category $\bbx(M)^G$ of equivariant vector fields for an action $G\times M \to M$ is in fact a $2$-vector space in the sense of Baez and Crans \cite{BC04}.
That is, it is a small category internal to the category of vector spaces and linear maps.
This means that the space of objects and the space of morphisms are vector spaces, and all the structure maps are linear.
For the purposes of this paper we only care about the additive structure.
That is, we only care that the category of equivariant vector fields is an abelian $2$-group.
\end{remark}

We conclude this subsection by introducing some examples of categories of equivariant vector fields.
In this subsection we concentrate on representations of compact Lie groups.
In subsection \ref{decomposition}, we use results in \cite{L15} and this paper, to give some similar examples for more general proper actions.

\begin{example}\label{Ex:Circle}
Let the circle $\bbs^1$ act on the complex plane $\bbc\cong\bbr^2$ (viewed as a real manifold) via rotations in the standard way:
\begin{equation*}
e^{i\theta}\cdot z := e^{i\theta} z, \qquad e^{i\theta}\in\bbs^1,\, z\in\bbc.
\end{equation*}
Since the group is abelian, the Adjoint action is trivial.
Thus, the infinitesimal gauge transformations are the $\bbs^1$-invariant functions on $\bbc$.
Given an infinitesimal gauge transformation $\psi:\bbc\to\bbr$, the corresponding induced vector field as in (\ref{Eq:InducedVF}) is given by:
\begin{equation*}
\partial(\psi)(z)
=\frac{\d}{\d\tau}\Big|_0 e^{\psi(z)i\tau}z
=\psi(z) i z.
\end{equation*}
It is shown in \cite[Ch.XII~Example~4.3]{GSS85}, using a nontrivial theorem of Schwarz \cite{Sch75}, that an arbitrary $\bbs^1$-invariant function $\psi:\bbc\to\bbr$ is of the form:
\begin{equation}\label{Eq:S1StdInf}
\psi(z)=\widehat\psi(|z|^2), \qquad z\in\bbc,
\end{equation}
for some smooth function $\widehat\psi:\bbr\to\bbr$.
We briefly describe how (\ref{Eq:S1StdInf}) follows from Schwarz's Theorem \cite{Sch75} as done in \cite[Ch.XII~Example~4.3]{GSS85}.
By Schwarz's Theorem \cite{Sch75}, it suffices to show that an arbitrary polynomial $\bbs^1$-invariant function $\psi:\bbc\to\bbr$ is of this form.
If $\psi$ is polynomial, then we may write:
\begin{equation*}
\psi(z)=\sum_{j,k} b_{jk} z^j \overline{z}^k, \qquad z\in\bbc,
\end{equation*}
where the sum is finite with nonnegative exponents, and $b_{jk}\in\bbc$ with $\overline{b_{jk}}=b_{kj}$ since $\psi$ is real-valued.
Invariance then implies the equation:
\begin{equation*}
\sum_{j,k} b_{jk} z^j \overline{z}^k = \sum_{j,k} b_{jk}e^{i(j-k)\theta}z^j\overline{z}^k, \qquad z\in\bbc,
\end{equation*}
which in turn implies that $b_{jk}=0$ or $j=k$.
Hence, $\psi$ is of the form:
\begin{equation*}
\psi(z)=\sum_{k}b_{kk}(z\overline{z})^k = \sum_k b_{kk} \left(|z|^2\right)^k,
\end{equation*}
which is of the desired form with:
\begin{equation*}
\widehat\psi(y):=\sum_k b_{kk}y^k, \qquad y\in\bbr.
\end{equation*}

Similarly, it is shown in \cite[Ch.XII~Example~5.4]{GSS85}, using Po\'enaru's Theorem \cite{Poe76}, that the $\bbs^1$-equivariant vector fields on $\bbc$ are of the form:
\begin{equation}\label{Eq:S1StdVF}
X(z)=f\left(|z|^2\right)z+g\left(|z|^2\right)iz, \qquad z \in \bbc,
\end{equation}
where $f:\bbr\to\bbr$ and $g:\bbr\to\bbr$ are smooth functions.
As described in \cite[Ch.XII~Example~5.4]{GSS85}, by Schwarz's Theorem \cite{Sch75} and Po\'enaru's Theorem \cite{Poe76}, it suffices to show that an arbitrary $\bbs^1$-equivariant vector field $X$ is of the form (\ref{Eq:S1StdVF}).
If $X$ is a polynomial vector field, then we may write:
\begin{equation*}
X(z)=\sum_{j,k} b_{jk} z^j \overline{z}^k,\qquad z\in\bbc,
\end{equation*}
where the sum is finite with nonnegative exponents and $b_{jk}\in\bbc$.
Equivariance then implies the equality:
\begin{equation*}
\sum_{j,k} b_{jk} z^j \overline{z}^k = \sum_{j,k} b_{jk} e^{i(j-k-1)\theta} z^j \overline{z}^k,
\end{equation*}
which in turn implies that $b_{jk}=0$ or $j=k+1$.
Hence, we have that:
\begin{equation*}
X(z)=\sum_k b_{k+1,k} (z\overline{z})^k z=\sum_{k} b_{k+1,k} \left(|z|^2\right)^k z = f\left(|z|^2\right)z + g\left(|z|^2\right)iz,
\end{equation*}
where:
\begin{equation*}
f(y):=\sum_k \text{Re}(b_{k+1,k})y^k,
\qquad
g(y):=\sum_k \text{Im}(b_{k+1,k})y^k,
\qquad
y\in\bbr.
\end{equation*}
Hence, all equivariant vector fields are as in (\ref{Eq:S1StdVF}).
From the point of view of the category $\bbx(\bbc)^{\bbs^1}$ of $\bbs^1$-equivariant vector fields on $\bbc$, the form (\ref{Eq:S1StdVF}) means that every equivariant vector field $X$ is isomorphic to its radial component:
\begin{equation*}
\widehat X(z):=f\left(|z|^2\right), \qquad z \in \bbc,
\end{equation*}
via the infinitesimal gauge transformation defined by $\psi(z):=g\left(|z|^2\right)$ for every $z\in\bbc$.
\end{example}

\begin{remark}
It should be noted that it is very diffficult in general to compute standard forms for infinitesimal gauge transformations and equivariant vector fields using Schwarz's Theorem \cite{Sch75} and Po\'enaru's Theorem \cite{Poe76}; as done, for example, to obtain (\ref{Eq:S1StdInf}) and (\ref{Eq:S1StdVF}).
While helpful, such expressions are not integral to the methods described in this paper.
\end{remark}

\begin{example}\label{Ex:O(2)C}
Recall that the orthogonal group $O(2)$ is generated by the circle $\bbs^1\cong SO(2)$ and a reflection $\kappa$.
Consider the representation of the orthogonal group $O(2)$ on $\bbr^2\cong \bbc$ generated by:
\begin{equation*}
  R_\theta\cdot z := e^{i\theta}z, \qquad \kappa\cdot z := \overline z, \qquad e^{i\theta}\in\bbs^1,\, z\in\bbc.
\end{equation*}
As in Example \ref{Ex:Circle}, the infinitesimal gauge transformations must be $\bbs^1$-invariant, but they must also satisfy equivariance with respect to the action of the reflection $\kappa$.
That is, an infinitesimal gauge transformation $\psi$ must satisfy:
\begin{equation*}
\psi(\overline{z})=\psi(\kappa\cdot z)=\Ad(\kappa)\psi(z)=-\psi(z),
\qquad z\in\bbc.
\end{equation*}
On the other hand, the $\bbs^1$-invariance implies that $\psi(\overline{z})=\psi(z)$ for all $z\in\bbc$.
Therefore, the map $\psi$ must satisfy $\psi(z)=-\psi(z)$, meaning $\psi(z)=0$ for all $z\in\bbc$.
Thus, there is only one infinitesimal gauge transformation $\mathbf{0}:\bbc\to\bbr$, mapping every point to the $0$ vector.

It is shown in \cite[Ch.XII~Example~5.5]{GSS85}, that the $O(2)$-equivariant vector fields on $\bbc$ are of the form:
\begin{equation}\label{Eq:O2StdVF}
X(z)=f\left(|z|^2\right)z,\qquad z\in\bbc,
\end{equation}
where $f:\bbr\to\bbr$ is a smooth function.
To see why, note that an $O(2)$-equivariant vector field $X$ is, in particular, $\bbs^1$-equivariant.
Thus, it is of the form (\ref{Eq:S1StdVF}).
On the other hand, equivariance with respect to the $\kappa$ reflection implies that:
\begin{equation*}
f\left(|z|^2\right) \overline{z} + g\left(|z|^2\right) i\overline{z}
=
f\left(|z|^2\right) \overline{z} - g\left(|z|^2\right) i\overline{z},
\qquad
z\in\bbc.
\end{equation*}
This means that $g(|z|^2)=0$ for all $z\in\bbc$, so $X$ is as in (\ref{Eq:O2StdVF}).
Thus, the category $\bbx(\bbc)^{O(2)}$ of $O(2)$-equivariant vector fields on $\bbc$ is a discrete category.
Put another way, every $O(2)$-equivariant vector field on $\bbc$ is a radial vector field that is only isomorphic to itself.
\end{example}

\begin{example}\label{Ex:T2C2}
Example \ref{Ex:Circle} can be generalized to the standard action by rotations of an $n$-dimensional torus $\bbt^n$ on the $n$-fold product $\bbc^n\cong \bbr^{2n}$.
We illustrate with the $\bbt^2$ case.
That is, consider the representation of the torus $\bbt^2:=\bbs^1\times\bbs^1$ on $\bbr^4\cong \bbc^2$ given by:
\begin{equation*}
\left(e^{i\theta},e^{i\varphi}\right)\cdot (z_1,z_2)
:=\left(e^{i\theta}z_1,e^{i\varphi}z_2\right),\qquad
\left(e^{i\theta},e^{i\varphi}\right)\in\bbt^2,\, (z_1,z_2)\in\bbc^2.
\end{equation*}
Since $\bbt^2$ is abelian, the Adjoint action of $\bbt^2$ on $\bbr^2$ is trivial.
Hence, the infinitesimal gauge transformations are the $\bbt^2$-invariant, or doubly-periodic, functions on $\bbc^2$.
Given an infinitesimal gauge transformation $\psi=(\psi_1,\psi_2):\bbc^2\to\bbr^2$, the corresponding induced vector field as in (\ref{Eq:InducedVF}) is giveny by:
\begin{equation*}
\partial(\psi)(z_1,z_2)=
\frac{\d}{\d\tau}\Big|_0 
\left(\begin{array}{c}
e^{\psi_1(z_1,z_2)i\tau} z_1 \\
e^{\psi_2(z_1,z_2)i\tau} z_2
\end{array}\right)
=
\left(\begin{array}{c}
\psi_1\left(z_1,z_2\right)iz_1\\
\psi_2\left(z_1,z_2\right)iz_2\\
\end{array}\right),
\end{equation*}
for $(z_1,z_1)\in\bbc^2$.
We claim that an arbitrary infinitesimal gauge transformation $\psi:\bbc^2\to\bbr^2$ is of the form:
\begin{equation}\label{Eq:T2StdInf}
\psi(z_1,z_2)=\Big(\psi_1\left(|z_1|^2,|z_2|^2\right),\psi_2\left(|z_1|^2,|z_2|^2\right)\Big),
\qquad (z_1,z_2)\in\bbc^2,
\end{equation}
for some smooth functions $\psi_1:\bbr^2\to\bbr$ and $\psi_2:\bbr^2\to\bbr$.
To see this, as in Example \ref{Ex:Circle}, by Schwarz's Theorem \cite{Sch75}, it suffices to show that an arbitray polynomial $\bbt^2$-invariant map $\psi:\bbc^2\to\bbr^2$ is of this form.
Let $\psi_1:\bbc^2\to\bbr$ be the first component of $\psi$.
If $\psi$ is a polynomial, we may write:
\begin{equation*}
\psi_1(z_1,z_2)=\sum_{j,k,l,m} a_{jklm} z_1^j\overline{z_1}^kz_2^l\overline{z_2}^m, \qquad (z_1,z_2)\in\bbc^2,
\end{equation*}
where the sum is finite with nonnegative exponents, and $a_{jklm}\in\bbc$ are such that $\psi_1$ is real-valued.
Invariance then implies the equation:
\begin{equation}\label{Eq:T2Invariance}
\sum_{j,k,l,m} a_{jklm} z_1^j\overline{z_1}^kz_2^l\overline{z_2}^m
=
\sum_{j,k,l,m} a_{jklm} e^{i(j-k)\theta}e^{i(l-m)\varphi}z_1^j\overline{z_1}^kz_2^l\overline{z_2}^m,
\end{equation}
for all $(z_1,z_2)\in\bbc^2,\,(e^{i\theta},e^{i\varphi})\in\bbt^2$.
Equation (\ref{Eq:T2Invariance}) implies that $a_{jklm}=0$ or both $j=k$ and $l=m$.
This means that $\psi_1$ is of the form:
\begin{equation*}
\psi_1(z_1,z_2)
=\sum_{k,m} a_{kkmm} \left(z_1\overline{z_1}\right)^k\left(z_2\overline{z_2}\right)^m 
=\sum_{k,m} a_{kkmm} \left(|z_1|^2\right)^k\left(|z_2|^2\right)^m,
\end{equation*}
for all $(z_1,z_2)\in\bbc^2$, which is of the desired form with:
\begin{equation*}
\widehat\psi_1(x,y):=\sum_{k,m} a_{kkmm}x^ky^m, \qquad (x,y)\in\bbr^2.
\end{equation*}
The argument is completely analogous for the second component of $\psi$, which proves that infinitesimal gauge transformations are of the form (\ref{Eq:T2StdInf}).

Similarly, we claim that the equivariant vector fields are of the form:
\begin{equation}\label{Eq:T2StdVF}
X(z_1,z_2)=
\left(\begin{array}{c}
f_1\left(|z_1|^2,|z_2|^2\right)z_1 + g_1\left(|z_1|^2,|z_2|^2\right) iz_1\\
f_2\left(|z_1|^2,|z_2|^2\right)z_2 + g_2\left(|z_1|^2,|z_2|^2\right) iz_2
\end{array}\right),
\end{equation}
for $(z_1,z_1)\in\bbc^2$, where $f_i:\bbr^2\to\bbr$ and $g_i:\bbr^2\to\bbr$ are smooth functions.
As with Example \ref{Ex:Circle}, by Schwartz's Theorem \cite{Sch75} and Po\'enaru's theorem \cite{Poe76}, it suffices to prove that an arbitrary $\bbt^2$-equivariant polynomial vector field $X$ is of the form (\ref{Eq:T2StdVF}).
If $X$ is a polynomial vector field, denote the first component by $X_1:\bbc^2\to\bbc$ and write:
\begin{equation*}
X_1(z_1,z_2)=\sum_{j,k,l,m} b_{jklm} z_1^j\overline{z_1}^kz_2^l\overline{z_2}^m, \qquad (z_1,z_2)\in\bbc^2,
\end{equation*}
where the sum is finite with nonnegative exponents and $b_{jklm}\in\bbc$.
Equivariance then implies the following equation:
\begin{equation}\label{Eq:T2Equivariance}
e^{i\theta}\sum_{j,k,l,m} b_{jklm} z_1^j\overline{z_1}^kz_2^l\overline{z_2}^m
=\sum_{j,k,l,m} b_{jklm} e^{i(j-k)\theta}e^{i(l-m)\varphi}z_1^j\overline{z_1}^kz_2^l\overline{z_2}^m
\end{equation}
for all $(z_1,z_2)\in \bbc^2$ and $(e^{i\theta},e^{i\varphi})\in\bbt^2$.
Equation (\ref{Eq:T2Equivariance}) in turn implies that $b_{jklm}=0$ or both $j=k+1$ and $l=m$.
Thus, $X_1$ is of the form:
\begin{align*}
X_1(z_1,z_2)&=\sum_{k,m} b_{k+1,k,m,m} \left(z_1\overline{z_1}\right)^k\left(z_2\overline{z_2}\right)^m z_1\\
&=\sum_{k,m} b_{k+1,k,m,m} \left(|z_1|^2\right)^k\left(|z_2|^2\right)^m z_1\\
&=f\left(|z_1|^2,|z_2|^2\right)z_1 + g_1\left(|z_1|^2,|z_2|^2\right)iz_1,
\end{align*}
for $(z_1,z_2)\in\bbc^2$, where:
\begin{equation*}
f_1(x,y):=\sum_{k,m} \text{Re}(b_{k+1,k,m,m})x^ky^m,
\qquad
g_1(x,y):=\sum_{k,m} \text{Im}(b_{k+1,k,m,m})x^ky^m,
\end{equation*}
for $(x,y)\in\bbr^2$.
An analogous argument yields a similar result for the second component, meaning that $X$ is as in (\ref{Eq:T2StdVF}).
From the point of view of the category $\bbx(\bbc^2)^{\bbt^2}$ of $\bbt^2$-equivariant vector fields on $\bbc^2$, (\ref{Eq:T2StdVF}) implies that every $\bbt^2$-equivariant vector field is isomorphic to a vector field of the form:
\begin{equation*}
\widehat X(z_1,z_2)=
\left(\begin{array}{c}
f_1\left(|z_1|^2,|z_2|^2\right)z_1\\
f_2\left(|z_1|^2,|z_2|^2\right)z_2\\
\end{array}\right),
\qquad (z_1,z_1)\in\bbc^2,
\end{equation*}
via the infinitesimal gauge transformation defined by:
\begin{equation}
\psi(z_1,z_2)=
\left(\begin{array}{c}
g_1\left(|z_1|^2,|z_2|^2\right)\\
g_2\left(|z_1|^2,|z_2|^2\right)\\
\end{array}\right),
\qquad (z_1,z_1)\in\bbc^2.
\end{equation}
\end{example}

\begin{example}\label{Ex:O(2)C2}
Consider the representation of the orthogonal group $O(2)$ on $\bbr^4\cong\bbc^2$ generated by:
\begin{equation*}
R_\theta\cdot (z_1,z_2) := \left(e^{i\theta} z_1,e^{-i\theta} z_2\right), \qquad
\kappa\cdot (z_1,z_2) := (z_2,z_1), \qquad e^{i\theta}\in\bbs^1.
\end{equation*}
If $\psi:\bbc^2\to\bbr$ is an infinitesimal gauge transformation, then equivariance with respect to the Adjoint action means that it must satisfy:
\begin{equation*}
\psi(e^{i\theta}z_1,e^{-i\theta}z_2) = \psi(z_1,z_2),
\qquad
\psi(z_2,z_1)=-\psi(z_1,z_2),
\end{equation*}
for all $(z_1,z_2)\in\bbc^2$ and $e^{i\theta}\in\bbs^1$.
The vector field induced by $\psi$ is of the form:
\begin{equation*}
\partial(\psi)(z_1,z_2)=
\frac{\d}{\d\tau}\Big|_0
\left(\begin{array}{c}
e^{\psi(z_1,z_2)i\tau}z_1\\
e^{-\psi(z_1,z_2)i\tau}z_2\\
\end{array}\right)
=\left(\begin{array}{c}
\psi\left(z_1,z_2\right)iz_1\\
\psi\left(z_2,z_1\right)iz_2
\end{array}\right),
\end{equation*}
for $(z_1,z_2)\in\bbc^2$ and where we used $\psi(z_2,z_1)=-\psi(z_1,z_2)$.
It is shown in \cite[\S~6.2.4]{CL00}, that the equivariant vector fields are of the form:
\begin{equation*}
X(z_1,z_2)=
\left(\begin{array}{c}
f\left(z_1,z_2\right)z_1 + g\left(z_1,z_2\right) iz_1 + h(z_1,z_2) \overline{z_2} + k(z_1,z_2) i\overline{z_2}\\
f\left(z_2,z_1\right)z_2 + g\left(z_2,z_1\right) iz_2 + h(z_2,z_1) \overline{z_1} + k(z_2,z_1) i\overline{z_1}
\end{array}\right),
\end{equation*}
for $(z_1,z_2)\in\bbc^2$, where $f:\bbc^2\to\bbr$, $g:\bbc^2\to \bbr$, $h:\bbc^2\to\bbr$, and $k:\bbc^2\to\bbr$ are smooth $\bbs^1$-invariant functions.
Note that, unlike in Example \ref{Ex:Circle}, not every $O(2)$-equivariant vector field is isomorphic to a radial one in the category $\bbx(\bbc^2)^{O(2)}$ of $O(2)$-equivariant vector fields in $\bbc^2$.
\end{example}

\subsection{Relative equilibria and ismorphsims of equivariant vector fields}
A relative equilibrium of an equivariant vector field is a point where the vector field is tangent to the group orbit.
They are the natural analogue of equilibria in the presence of symmetries and are thus central to much of equivariant dynamics.
In this subsection we review some definitions and standard facts concerning relative equilibria.
Furthermore, we prove that isomorphisms of equivariant vector fields preserve relative equilibria, but modify their velocities, hence the motion of the relative equilibrium may change (Lemma \ref{Lemma:IsoShareRel}).

First, recall:

\begin{definition}\label{Def:RelEq}
Let $M$ be a $G$-manifold and let $X$ be an equivariant vector field on $M$.
A point $m\in M$ is a {\it relative equilibrium} of $X$ if the vector $X(m)$ is tangent to the group orbit $G\cdot m$ at $m$.
That is, $m$ is a relative equilibrium of $X$ if $X(m)\in T_m(G\cdot m)$.
\end{definition}

\begin{definition}\label{Def:EvMap}
Let $M$ be a $G$-manifold and let $m\in M$ be a point.
The {\it evaluation map} is the map:
\begin{equation*}
\ev_m:G\to M, \qquad \ev_m(g):=g\cdot m.
\end{equation*}
\end{definition}

\begin{definition}\label{Def:Velocity}
Let $M$ be a $G$-manifold, let $X$ be an equivariant vector field on $M$, and let $m$ be a relative equilibrium of $X$.
A {\it velocity} of $m$ with respect to $X$ is a Lie algebra vector $\xi\in\ffg$ such that $X(m)=T\ev_m(\xi)$.
\end{definition}

\begin{remark}\label{Rem:VelocityFacts}
Let $m$ be a relative equilibrium of an equivariant vector field $X$ on a $G$-manifold $M$.
Velocities for relative equilibria always exist since:
\begin{equation}\label{Eq:TangentEvaluation}
T_m(G\cdot m)=T\ev_m(\ffg).
\end{equation}
Furthermore, such velocities are unique modulo the Lie algebra $\ffg_m$ of the isotropy group $G_m$ of $m$, which is the kernel of the map $T\ev_m$.
That is, if $\xi$ and $\widetilde\xi$ are two velocities of $m$, then there exists a vector $\eta\in\ffg_m$ such that $\widetilde\xi = \xi + \eta$.
\end{remark}

We recall the following two standard facts about velocities of relative equilibria:

\begin{lemma}
Let $M$ be a $G$-manifold, and let $X$ be an equivariant vector field on $M$ with a relative equilibrium at a point $m\in M$.
Let $\xi\in\ffg$ be a velocity of the relative equilibrium $m$ of $X$, then:
\begin{enumerate}
\item For any $g\in G$, the vector $\Ad(g)(\xi)$ is a velocity of the relative equilibrium $g\cdot m$ of $X$
\item The integral curve of $X$ starting at $m$ is equal to:
\begin{equation*}
\gamma:\bbr\to M, \qquad \gamma(t):=\exp(\tau\xi)\cdot m.
\end{equation*}
\end{enumerate}
\end{lemma}

\begin{proof}
The first follows from the following computation:
\begin{align*}
X(g\cdot m)&=Tg_M X(m) \qquad \text{by the equivariance of }X\\
&=Tg_M T\ev_m (\xi) \qquad \text{since }\xi\text{ is a velocity of }m\\
&=\frac{\d}{\d\tau}\Big|_0 g\exp(\tau\xi)\cdot m \\
&=\frac{\d}{\d\tau}\Big|_0 \left(g\exp(\tau\xi)g^{-1}\right)\cdot (g\cdot m) \\
&=\frac{\d}{\d\tau}\Big|_0 \exp(\tau \Ad(g)(\xi))\cdot (g\cdot m) \qquad \text{ by the naturality of }\exp \\
&=T\ev_{g\cdot m} \left(\Ad(g)(\xi)\right).
\end{align*}
To prove the second note that $\gamma(0)=m$ and that for any $s\in\bbr$ we have:
\begin{align*}
&\frac{\d}{\d\tau}\Big|_s \gamma(\tau)\\
 &= \frac{\d}{\d\tau}\Big|_s \exp(\tau\xi)\cdot m \\
&=\frac{\d}{\d u}\Big|_0 \exp\left((s+u)\xi\right) \cdot m \qquad \text{changing variables by }\tau=u+s\\
&=\frac{\d}{\d u}\Big|_0 \exp(s\xi)\exp(u\xi)\cdot m \,\, \text{ by the $1$-parameter subgroup property of }\exp\\
&=\frac{\d}{\d u}\Big|_0 \Big(\exp(s\xi)\exp(u\xi)\exp(s\xi)^{-1}\Big)\cdot (\exp(s\xi)\cdot m) \\
&=\frac{\d}{\d u}\Big|_0 \exp\Big(u \Ad\left(\exp(s\xi)\right)(\xi)\Big) \cdot (\exp(s\xi)\cdot m) \,\,\text{ by the naturality of }\exp \\
&=T\ev_{\exp(s\xi)\cdot m} \left(\Ad\left(\exp(s\xi)\right)(\xi)\right)\\
&=X\left(\exp(s\xi)\cdot m\right) \qquad \text{ by the first part of this lemma}.
\end{align*}
Hence, $\gamma$ is the integral curve of $X$ starting at $m$.
\end{proof}

Isomorphisms preserve relative equilibria but change velocities as follows:

\begin{lemma}\label{Lemma:IsoShareRel}
Let $X$ and $Y$ be isomorphic vector fields on a $G$-manifold $M$, and let $\psi:M\to\ffg$ be an infinitesimal gauge transformation such that $Y=X+\partial(\psi)$.
If $m\in M$ is a relative equilibrium of $X$ and $\xi\in\ffg$ is a velocity of $m$ for $X$, then the point $m$ is also a relative equilibrium of $Y$ and $\xi+\psi(m)\in\ffg$ is a velocity of $m$ for $Y$.
\end{lemma}

\begin{proof}
Note that:
\begin{equation*}
Y(m)
=X(m)+\partial(\psi)(m)
=T\ev_m(\xi) + T\ev_m(\psi(m))
=T\ev_m\left(\xi+\psi(m)\right)
\end{equation*}
Hence, $Y(m)\in T_m(G\cdot m)$, meaning the point $m$ is a relative equilibrium of $Y$.
Furthermore, $\xi+\psi(m)$ is a velocity of $m$ as a relative equilibrium of $Y$.
\end{proof}

\subsection{Decomposition of equivariant vector fields near relative equilibria}\label{decomposition}
Given an equivariant vector field with a relative equilibrium at a point (Definition \ref{Def:RelEq}), we show how it can be decomposed into two vector fields: one with an equilibrium at the given point and another induced by an isomorphism of equivariant vector fields as in (\ref{Eq:InducedVF}).
This decomposition is the main theorem of this section (Theorem \ref{Thm:Decomposition}).
In fact, this decomposition follows from an equivalence between the category of equivariant vector fields near the group orbit of a point and the category of equivariant vector fields on a slice through the point (Theorem \ref{Thm:EquivalenceVFs}).
Theorem \ref{Thm:EquivalenceVFs} is a local version of Theorem \ref{Thm:Decomposition}, and its proof will occupy most of this subsection.

Theorem \ref{Thm:EquivalenceVFs} was first proved in a different but equivalent way in \cite[Theorem~4.3]{L15}.
However, there are some differences between this subsection and the treatment in \cite{L15}.
There, equivariant vector fields and infinitesimal gauge transformations are seen as parts of a $2$-term chain complex of vector spaces (Remark \ref{Rem:2VectorSpaceStr}).
In this paper, we consider instead the category of equivariant vector fields (Definition \ref{Def:GpoidVFs}).
Thinking in terms of categories has benefits; e.g. isomorphisms of equivariant vector fields correspond precisely to the isomorphisms in the category of equivariant vector fields (Definition \ref{Def:IsoVF}), and the choices in Theorem \ref{Thm:EquivalenceVFs} can be framed in terms of the natural isomorphisms involved (see Proposition \ref{Prop:ProjectionChoice} and Proposition \ref{Prop:SliceChoice}).

As mentioned in Remark \ref{Rem:2VectorSpaceStr}, the category $\bbx(M)^G$ of equivariant vector fields on a $G$-manifold $M$ is in fact a $2$-vector space.
On the other hand, the strict $2$-category $\TwoVect$ of $2$-vector spaces and the strict $2$-category $\TwoTerm\Vect$ of $2$-term chain complexes of vector spaces are equivalent as strict $2$-categories \cite{BC04}.
Hence, Theorem \ref{Thm:EquivalenceVFs} is equivalent to \cite[Theorem~4.3]{L15}.
Similarly, some of the proofs in this subsection are the analogues in the $2$-category $\TwoVect$ of the proofs in \cite{L15}, which are done in the $2$-category $\TwoTerm\Vect$.
However, our treatment of the choices in the equivalence of categories in Theorem \ref{Thm:EquivalenceVFs} differs from \cite{L15}.
In particular, we express the effect of the choice in terms of a natural isomorphism (compare Proposition \ref{Prop:ProjectionChoice} with Lemma \cite[Lemma~3.17]{L15} and Proposition \ref{Prop:SliceChoice} with \cite[Lemma~3.21]{L15}).
We also provide additional details of some of the constructions used in Theorem \ref{Thm:EquivalenceVFs} (see, in particular, Remark \ref{Rem:ConnectionFromSplitting} and Theorem \ref{Thm:NaturalIsomorphism}).
Furthermore, at the end of the subsection, we provide some original examples of decompositions of equivariant vector fields using Theorem \ref{Thm:EquivalenceVFs}.

The main goal of this subsection is to prove the following:

\begin{theorem}\label{Thm:Decomposition}
Let $M$ be a proper $G$-manifold and let $X$ be an equivariant vector field on $M$ with a relative equilibrium at $m$.
Then there exists an equivariant vector field $Y^X\in\ffX(M)^G$, with $Y^X(m)=0$, and an infinitesimal gauge transformation $\psi^X\in C^\infty(M,\ffg)^G$ such that:
\begin{equation*}
X=Y^X+\partial\left(\psi^X\right).
\end{equation*}
Furthermore, $Y^X$ is transverse to the group orbits near $G\cdot m$.
\end{theorem}

\begin{remark}\label{Rem:LiteratureObservation1}
The decomposition of the vector field $X$ in Theorem \ref{Thm:Decomposition} is analogous to that of Krupa \cite[Theorem~2.1]{K90}.
Fiedler, Sandstede, Scheel, and Wulff extend this to proper actions by providing a different decomposition of the vector field $X$.
In \cite[Theorem~1.1]{FSSW96}, they lift the vector field $X$ to a $(G\times G_m)$-equivariant vector field on a product $G\times V$, where $V$ is the slice representation at $m$, and write the lift as a ``skew-product'' vector field.
We choose to work directly on the manifold $M$, use equivariant connections instead of invariant Riemannian metrics, and use infinitesimal gauge transformations.
We believe this description in terms of infinitesimal gauge transformations is natural and conceptually simple, but additionally it also helps address the effect of the choices involved (Proposition \ref{Prop:ProjectionChoice} and Proposition \ref{Prop:SliceChoice}).
In particular, the choices involved lead to transverse vector fields that are isomorphic.
\begin{comment}
Second, as shown in \cite{K18}, it allows for a conceptually simple way to address the stability of relative equilibria (section \ref{stability1} and \ref{stability2}).
Furthermore, as shown in section \ref{motionsection}, it also provides a way to control the independent frequencies in the motion of a relative equilibrium.
Finally, as we show in the other sections, it also lends itself well to the construction and description of bifurcations from and to relative equilibria, helping to describe bifurcations to steady-state, periodic, and quasi-periodic motion (section \ref{ch2}).
\end{comment}
\end{remark}

We will first work in an invariant neighborhood of the group orbit of the relative equilibrium.
Recall that given a $K$-manifold $D$, where $K$ is a compact Lie subgroup of a Lie group $G$ we can form the {\it associated bundle} $G\times^KD \to G/K$ (see Notation \ref{Notation}.\ref{Not:AssociatedBundle}).
The total space $G\times^KD$ is the smooth quotient of the action of $K$ on the product $G\times D$ defined by:
\begin{equation*}
k\cdot (g,v):= \left(gk^{-1},k\cdot v\right),
\qquad (k,g,v)\in K\times G\times D.
\end{equation*}
We denote the points of $G\times^KD$ by $[g,v]$.
This is a proper $G$-manifold with the action of $G$ given by:
\begin{equation*}
g'\cdot [g,v]:=[g'g,v],
\qquad \left(g',g,v\right)\in G\times G\times D.
\end{equation*}
Relatedly, we have the following definition:

\begin{definition}\label{Def:Slices}
Given a proper $G$-manifold $M$, let $K$ be the stabilizer of a point $m\in M$.
A {\it slice through $m$ for the action $G\times M \to M$} is a $K$-manifold $D$ and a $K$-equivariant embedding $j\colon D\to M$ such that:
\begin{enumerate}
\item The point $m$ is in the image $j(D)$.
\item The set:
\begin{equation*}
G\cdot j(D):= \left\{
g\cdot j(v) \mid g \in G,\, v\in D
\right\}
\end{equation*}
is open in $M$.
\item The map:
\begin{gather*}
G\times V \to G\cdot j(D), \qquad (g,v)\mapsto g\cdot j(v)
\end{gather*}
descends to a $G$-equivariant diffeomorphism:
\begin{gather*}
G\times^KD\to G\cdot j(D), \qquad [g,v]\mapsto g\cdot j(v),
\end{gather*}
where as before $G\times^KD:=(G\times D)/K$ and the action of $K$ on $G\times D$ is given by:
\begin{equation*}
k\cdot (g,v):= \left(gk^{-1},k\cdot v\right),
\qquad (k,g,v)\in K\times G\times D.
\end{equation*}
\end{enumerate}
\end{definition}

\begin{remark}\label{Rem:TubularNeighborhood}
It is a classic theorem of Palais \cite{Pa61} that slices exist at every point in a \textit{proper} $G$-manifold M (see also \cite[Theorem~2.3.3]{DK00}).
In fact, it is possible and convenient to take the slice $D$ through a point $m\in M$ to be an open ball around the origin of the canonical slice representation:
\begin{equation*}
V:=T_mM/T_m(G\cdot m),
\end{equation*}
which has a canonical representation of the stabilizer $K$ of the point $m$ (see, for example, \cite[Theorem~B.24]{GuK02}).
More is true.
In fact, the following diagram commutes:
\begin{equation}\label{Diag:SliceEmbeddingDiagram}
\begin{gathered}
\xy
{(-16,10)}*+{G\times^KD} = "1";
{(16,10)}*+{G\cdot D} = "2";
{(-16,-10)}*+{G/K} = "3";
{(16,-10)}*+{G\cdot m} = "4";
{\ar@{->}^{\cong} "1";"2"};
{\ar@{->}_{} "1";"3"};
{\ar@{->}_{\cong} "3";"4"};
{\ar@{->}^{} "2";"4"};
\endxy
\end{gathered}
\end{equation}
where the top map is the diffeomorphism $\phi_j$ of Definition \ref{Def:Slices}, the bottom map is given by:
\begin{equation*}
gK\mapsto g\cdot m, 
\qquad g\in G,
\end{equation*}
and the vertical maps are defined by:
\begin{equation*}
[g,d]\mapsto gK, 
\qquad [g,d]\in G\times^KD,
\end{equation*}
and by:
\begin{equation*}
g\cdot j(v) \mapsto g\cdot m,
\qquad g\cdot j(v)\in G\cdot D,
\end{equation*}
respectively.
In particular, the associated bundle $G\times^KV$, where $V$ is the canonical slice $K$-representation of $m$ is a vector bundle and proper $G$-manifold that is locally $G$-equivariantly diffeomorphic to a $G$-invariant neighborhood of $m$ in $M$.
Thus, throughout this paper we often focus on the local models $G\times^KV$ where $V$ is {\it some} finite-dimensional real representation of a compact Lie subgroup $K$ of a Lie group $G$.
\end{remark}

We now state a local version of Theorem \ref{Thm:Decomposition}.
As described in the introduction to this subsection, an equivalent result is originally due to Lerman in \cite{L15}.

\begin{theorem}\label{Thm:EquivalenceVFs}
Let $V$ be a representation of a compact Lie subgroup $K$ of a Lie group $G$.
There is an equivalence of categories:
\begin{equation*}
\bbx\left(G\times^KV\right)^G\simeq \bbx\left(V\right)^K
\end{equation*}
between the categories of equivariant vector fields on $G\times^KV$ and $V$ respectively (Definition \ref{Def:GpoidVFs}).
In particular, there exist functors $E:\bbx(V)^K\to \bbx(G\times^KV)^G$ and $P:\bbx(G\times^KV)^G\to\bbx(V)^K$ and a natural isomorphism $h: \ffX(G\times^KV)^G\to C^\infty(G\times^KV,\ffg)^G$ such that, for every equivariant vector field $X\in\ffX(G\times^KV)^G$, we have that:
\begin{equation*}
X = E_0(P_0(X)) + \partial\left(h(X)\right),
\end{equation*}
is a decomposition as in Theorem \ref{Thm:Decomposition} (see \ref{Notation}.\ref{Not:Categories} for notation).
\end{theorem}

The proof of Theorem \ref{Thm:EquivalenceVFs} will consist of constructing a pair of functors $E$ and $P$, and a natural isomorphism $h:EP\cong 1_{\bbx(G\times^KV)^G}$.
Before proving Theorem \ref{Thm:EquivalenceVFs}, we explain the idea behind it and why it is a local version of Theorem \ref{Thm:Decomposition}.

Given an equivariant vector field $Y\in\ffX(V)^K$ on the $K$-representation $V$, there is a canonical way to equivariantly extend it to an equivariant vector field $E_0(Y)\in\ffX(G\times^KV)^G$ on the associated bundle $G\times^KV$.
It turns out that the vector field $E_0(Y)$ is vertical in the associated bundle $G\times^KV\to G/K$.

Conversely, a choice of equivariant connection:
\begin{equation*}
\Phi\in\Omega^1(G\times^KV;\calv(G\times^KV))^G
\end{equation*}
on the bundle $G\times^KV\to G/K$, where $\calv(G\times^KV) \to G\times^KV$ is the vertical bundle (see \ref{Notation}.\ref{Not:VerticalBundle}), gives a map in the other direction.
That is, for any equivariant vector field $X\in \ffX(G\times^KV)^G$ there is an equivariant vector field $P_0(X)\in \ffX(V)^K$ on the $K$-representation $V$ corresponding to the restriction to the $K$-representation $V\cong\{[1,v]\mid v\in V\}$ of the vertical part of the vector field $X$.
We think of $P_0(X)$ as the projection of the vector field $X$ onto the representation $V$.

Given an equivariant vector field $Y$ on the representation $V$, equivariantly extending and then projecting returns the original vector field (that is, $P_0(E_0(Y))=Y$).
On the other hand, given an equivariant vector field $X$ on the bundle $G\times^KV$, projecting and then equivariantly extending does not return the original vector field (that is, $E_0(P_0(X))\not=X$).
It only returns the vertical part of the vector field.
Nevertheless, we can recover the vector field $X$ via the action of the abelian group $C^\infty(G\times^KV,\ffg)^G$ on the space of equivariant vector fields $\ffX(G\times^KV)^G$ given in (\ref{Eq:Action}).
In particular, there exists an infinitesimal gauge transformation $h(X)\in C^\infty(G\times^KV,\ffg)^G$ such that $X=h(X)\cdot E_0(P_0(X))$.
That is, there is a decomposition of the equivariant vector field $X$ given by:
\begin{equation}\label{Eq:Decomposition}
X=E_0(P_0(X))+\partial(h(X)),
\end{equation}
where the vector field $E_0(P_0(X))$ is transverse to the group orbit $G\cdot [1,0]$ and the vector field $\partial(h(X))$ is tangent to the group orbits.
The decomposition (\ref{Eq:Decomposition}) is the local version of the decomposition in Theorem \ref{Thm:Decomposition}.

To prove Theorem \ref{Thm:Decomposition} we show that the maps $E_0$ and $P_0$ can be extended to functors $E$ and $P$ (Theorem \ref{Thm:InclusionFunctor} and \ref{Thm:ProjectionFunctor}), prove that $PE=\text{id}$ (Theorem \ref{Thm:ExtendThenProject}), and construct a natural isomorphism $h:EP\cong 1_{\bbx(G\times^KV)^G}$ (Theorem \ref{Thm:NaturalIsomorphism}).
The decomposition in (\ref{Eq:Decomposition}) follows from this natural isomorphism.

The following lemma will be useful in proving Theorem \ref{Thm:EquivalenceVFs}:

\begin{lemma}\label{Lemma:FunctorShortcut}
Let $A_1$ and $B_1$ be two abelian groups acting, respectively, on sets $A_0$ and $B_0$, and let $A_1\times A_0\toto A_0$ and $B_1\times B_0 \toto B_0$ be the corresponding action groupoids.
If $F_0:A_0\to B_0$ is a function and $F_1:A_1\to B_1$ is a group homomorphism such that:
\begin{equation}\label{Eq:TargetCommute}
F_0(\psi\cdot x) = F_1(\psi)\cdot F_0(x), \qquad \text{for all } (\psi,x)\in A_1\times A_0,
\end{equation}
then:
\begin{equation}\label{Eq:FunctorFormula}
(F_1\times F_0,F_0): \Big(A_1\times A_0 \toto A_0 \Big) \to \Big(B_1\times B_0 \toto B_0 \Big)
\end{equation}
is a functor.
\end{lemma}

\begin{proof}
Denote the source, target, and unit maps by $s$, $t$, and $u$ respectively, letting context imply the category we are working with as is customary.
First note that $(F_1\times F_0,F_0)$ sends a morphism $(\psi,x)\in A_1\times A_0$ to a morphism $(F_1(\psi),F_0(x))\in B_1\times B_0$ since:
\begin{equation*}
s(F_1(\psi),F_0(x))=F_0(x)=F_0(s(\psi,x))
\end{equation*}
and:
\begin{equation*}
t(F_1(\psi),F_0(x))=F_1(\psi)\cdot F_0(x) = F_0(\psi\cdot x) = F_0( t(\psi,x)),
\end{equation*}
where we make use of (\ref{Eq:TargetCommute}).
Let $\psi,\varphi \in A_1$ and $x\in a_0$ be arbitrary.
Then, since $F_1$ is a group homomorphism, note that $(F_1\times F_0,F_0)$ respects units:
\begin{equation*}
u(F_0(x))=(0,F_0(x))=(F_1(0),F_0(x))=(F_1\times F_0)(0,x)=(F_1\times F_0)(u(x)).
\end{equation*}
Finally, $(F_1\times F_0,F_0)$ also respects composition:
\begin{align*}
(F_1\times F_0)\Big((\psi,\varphi\cdot x)\circ (\varphi,x)\Big)
&=(F_1\times F_0)(\psi+\varphi, x)\\
&=(F_1(\psi+\varphi),F_0(x))\\
&=(F_1(\psi)+F_1(\varphi),F_0(x))\\
&=(F_1(\psi),F_1(\varphi)\cdot F_0(x)) \circ (F_1(\varphi),F_0(x))\\
&=(F_1(\psi),F_0(\varphi\cdot x)) \circ (F_1(\varphi),F_0(x))\\
&=(F_1\times F_0)(\psi,\varphi\cdot x) \circ (F_1\times F_0)(\varphi,x).
\end{align*}
Hence, (\ref{Eq:FunctorFormula}) is a functor as claimed.
\end{proof}

We now prove that equivariant extension of equivariant vector fields on a representation is functorial:

\begin{theorem}\label{Thm:InclusionFunctor}
Let $V$ be a representation of a compact Lie subgroup $K$ of a Lie group $G$.
Let $j:V\hookrightarrow G\times^KV$ be the embedding defined by $j(v):=[1,v]$ for $v\in V$ (see Notation \ref{Notation}.\ref{Not:AssociatedBundle}).
Then the map:
\begin{equation}\label{Eq:InclusionObjects}
E_0:\ffX(V)^K\to\ffX(G\times^KV)^G,\qquad
X\mapsto E_0X,
\end{equation}
where $E_0X$ is defined by:
\begin{equation*}
(E_0X)([g,v]):=g\cdot (Tj) X(v),\qquad
[g,v]\in G\times^KV,
\end{equation*}
with the action as in (\ref{Notation}.\ref{Not:TangentAction}), and the map:
\begin{equation}\label{Eq:InclusionGauges}
E_1:C^\infty(V,\ffk)^K \to C^\infty(G\times^KV,\ffg)^G,\qquad
\psi \mapsto E_1\psi,
\end{equation}
where, letting $\iota:\ffk\hookrightarrow\ffg$ be the Lie algebra inclusion, $E_1\psi$ is defined by:
\begin{equation*}
(E_1\psi)([g,v]):=\Ad(g)\iota\left(\psi(v)\right),\qquad
[g,v]\in G\times^KV,
\end{equation*}
define a functor $E:\bbx(V)^K\to\bbx(G\times^KV)^G$.
\end{theorem}

\begin{proof}
By Lemma \ref{Lemma:FunctorShortcut}, it suffices to check that $E_1$ is a group homomorphism and $E_1$ and $E_0$ intertwine the actions as in (\ref{Eq:TargetCommute}).
That $E_1$ is a homomorphism follows immediately from the linearity of the inclusion $\iota:\ffk\hookrightarrow \ffg$ and of the Adjoint maps $\Ad(g)$ for each $g\in G$.
It is straightforward to verify that $E_0$ is also linear in this case since the action of $G$ on $TM$ is fiber-wise linear.

It remains to check condition (\ref{Eq:TargetCommute}).
We first prove that the following diagram commutes:
\begin{equation}\label{Diag:BoundaryInclusion}
\xy
{(-24,10)}*+{C^\infty(V,\ffk)^K} = "1";
{(24,10)}*+{\ffX(V)^K} = "2";
{(-24,-10)}*+{C^\infty(G\times^KV,\ffg)^G} = "3";
{(24,-10)}*+{\ffX(G\times^KV)^G} = "4";
{\ar@{->}^{\partial} "1";"2"};
{\ar@{->}_{E_1} "1";"3"};
{\ar@{->}_{\,\,\,\,\partial} "3";"4"};
{\ar@{->}^{E_0} "2";"4"};
\endxy
\end{equation}
Observe that for an infinitesimal gauge transformation $\psi:V\to\ffk$ on $V$ and a point $[g,v]\in G\times^KV$ we have:
\begin{align*}
E_0\partial(\psi)([g,v]) 
&=g\cdot (Tj)(\partial\psi)(v) \\
&=g\cdot (Tj) T\ev_v(\psi(v))\\
&=g\cdot T\ev_{[1,v]}(\psi(v))
\qquad\text{by the chain rule}\\
&=T\ev_{g,v}\Ad(g)\psi(v)
\qquad\text{since }g\cdot T\ev_{[1,v]}=T\ev_{[g,v]}\Ad(g)\\
%\qquad\text{again by the chain rule}\\
&=T\ev_{g,v}\left(E_1\psi([g,v])\right)\\
&=\partial(E_1\psi)([g,v]).
\end{align*}
Thus, diagram (\ref{Diag:BoundaryInclusion}) commutes as claimed.
We now check condition (\ref{Eq:TargetCommute}).
Let $X\in\ffX(V)^K$ and let $\psi\in C^\infty(V,\ffk)^K$, then:
\begin{align*}
E_1(\psi)\cdot E_0(X)
&=E_0(X) + \partial(E_1(\psi))\\
&=E_0(X) + E_0(\partial(\psi))
\qquad \text{ by diagram (\ref{Diag:BoundaryInclusion}) }\\
&= E_0(X + \partial(\psi))
\qquad \text{ by the linearity of }E_0\\
&=E_0(\psi\cdot X).
\end{align*}
Hence, the maps $E_1$ and $E_0$ satisfy (\ref{Eq:TargetCommute}), and more generally the hypotheses of Lemma \ref{Lemma:FunctorShortcut}.
Consequently, $E_1$ and $E_0$ give a functor $E:\bbx(V)^K\to\bbx(G\times^KV)^G$.
\end{proof}

\begin{remark}\label{Rem:ConnectionFromSplitting}
We need to recall a standard construction of equivariant connections.
Let $V$ be a finite-dimensional real representation of a compact Lie subgroup $K$ of a Lie group $G$.
Consider the principal $K$-bundle $G\to G/K$.
By equivariance, a choice of principal connection on the principal $K$-bundle $G\to G/K$ is in correspondence with a choice of $K$-equivariant splitting of the short exact sequence of vector spaces:
\begin{equation}\label{Eq:SplittingExplained}
\xy
{(-13,0)}*+{\ffk} = "1";
{(0,0)}*+{\ffg} = "2";
{(13,0)}*+{\ffg/\ffk} = "3";
{\ar@{->}^{} "1";"2"};
{\ar@{->}_{} "2";"3"};
\endxy
\end{equation}
%To see this, note that by equivariance, it suffices to consider the map $T_1G\to T_\ffk(G/K)$, which is the right hand map in (\ref{Eq:SplittingExplained}).
Thus, it is equivalent to choose a $K$-invariant complement $\ffq$ of $\ffk$ in $\ffg$, a $K$-equivariant projection $\bbp:\ffg\to\ffk$, or a principal connection $\widetilde\Phi$ of $G\to G/K$.
In particular, given a $K$-equivariant projection $\bbp:\ffg\to\ffk$, the corresponding connection $\widetilde\Phi\in\Omega^1(G;\calv(G))$ on the principal $K$-bundle $G\to G/K$ is given by:
\begin{equation}\label{Eq:PrincipalConnection}
\widetilde\Phi(g,\xi) := \left(g,\bbp(\xi)\right), \qquad
(g,\xi)\in G\times \ffg\cong TG.
\end{equation}

The total space of the tangent bundle $T(G\times^KV)$ of the associated bundle $G\times^KV$ is itself an associated bundle $TG\times^{TK}TV\to T(G/K)$ (see, for example, \cite[Theorem~10.18~(4)]{KMS93}).
Thus, we have the following commutative diagram:
\begin{equation*}
\xy
{(-16,10)}*+{TG\times^{TK}TV} = "1";
{(16,10)}*+{G\times^KV} = "2";
{(-16,-10)}*+{T(G/K)} = "3";
{(16,-10)}*+{G/K} = "4";
{\ar@{->}^{} "1";"2"};
{\ar@{->}_{} "1";"3"};
{\ar@{->}_{} "3";"4"};
{\ar@{->}^{} "2";"4"};
\endxy
\end{equation*}
where the horizontal maps are the tangent bundle projections and the vertical maps are the associated bundle projections.
There are also canonical identifications:
\begin{equation*}
T(G\times^KV) \cong TG\times^{TK} TV \cong (G\times\ffg)\times^{K\times\ffk} (V\times V),
\end{equation*}
where the associated bundle on the right hand end is the quotient:
\begin{equation*}
(G\times\ffg)\times^{K\times\ffk} (V\times V) := \Big((G\times\ffg)\times (V\times V)\Big)/(K\times\ffk)
\end{equation*}
of the action of the group $K\times\ffk$ on the space $G\times\ffg\times V\times V$ given by:
\begin{equation}\label{Eq:TangentAssociatedBundleAction}
(k,\eta)\cdot\Big(g,\xi,v,w\Big)
=\Big(gk^{-1},\Ad(k)(\xi-\eta),k\cdot v, k\cdot \eta_V(w) \Big),
\end{equation}
for $(g,\xi,v,w)\in G\times \ffg\times V\times V$ and $(k,\eta)\in K\times \ffk$, and where $\eta_V:V\to V$ is the fundamental vector field induced by $\eta$ and defined by $w\mapsto \frac{\d}{\d\tau}\big|_0 \exp(\tau\eta)\cdot w$.
That is, the action of the group $K\times\ffk$ on $G\times \ffg\times V\times V$ is given by the tangent map of the action $K\times (G\times V) \to G\times V$ giving rise to the associated bundle $G\times^KV$.

Now let $\calv(G\times^KV)\to G\times^KV$ be the vertical bundle of the associated bundle $G\times^KV \to G/K$ (see Notation \ref{Notation}.\ref{Not:AssociatedBundle} and \ref{Notation}.\ref{Not:VerticalBundle}), and let $\varpi:TG\times TV\to TG\times^{TK}TV\cong T(G\times^KV)$ be the quotient map.
A choice of $K$-equivariant splitting $\ffg=\ffk\oplus\ffq$ determines an equivariant connection $\Phi\in\Omega^1(G\times^KV;\calv(G\times^KV))^G$ on the bundle $G\times^KV\to G/K$.
The connection $\Phi$ is given by following diagram:
\begin{equation}\label{Diag:ConnectionDefined}
\xy
{(-16,10)}*+{TG\times TV} = "1";
{(16,10)}*+{TG\times TV} = "2";
{(-16,-10)}*+{T(G\times^KV)} = "3";
{(16,-10)}*+{T(G\times^KV)} = "4";
{\ar@{->}^{\widetilde\Phi\times \id} "1";"2"};
{\ar@{->}_{\varpi} "1";"3"};
{\ar@{->}_{\Phi} "3";"4"};
{\ar@{->}^{\varpi} "2";"4"};
\endxy
\end{equation}
where $\widetilde\Phi$ is as in (\ref{Eq:PrincipalConnection}).
For a verification that this is an equivariant connection see, for example, \cite[Section~11.8]{KMS93}.

We will also make use of the corresponding horizontal bundle, which we now describe.
The $K$-equivariant splitting $\ffg=\ffk\oplus\ffq$ gives the following equivariant diffeomorphisms for the tangent bundle $T(G\times^KV)$:
\begin{align}\label{Eq:AssociatedBundleIdentification}
\begin{split}
TG\times^{TK}TV
&\cong(TG\times TV)/TK\\
&\cong \Big(G\times \ffk\oplus\ffq) \times (V\times V)\Big) / (K\times \ffk)\\
&\cong \left( G \times (V \times \ffq\times V)\right) / K\\
&= G\times^{K} (V\times \ffq\times V),
\end{split}
\end{align}
where we have used the action (\ref{Eq:TangentAssociatedBundleAction}).
We will denote the elements of the tangent bundle $G\times^{K}(V\times \ffq\times V)$ by $[g,v,\xi,w]$, where the first two entries correspond to the base point $[g,v]\in G\times^KV$ of the tangent vector $[g,v,\xi,w]$ (see also Notation \ref{Notation}.\ref{Not:AssociatedBundle}).
Now observe that we have a commutative diagram:
\begin{equation}\label{Diag:TangentAssociatedBundleIdentification}
\xy
{(-22,10)}*+{TG\times^{TK} TV} = "1";
{(22,10)}*+{G\times^{K}(V\times\ffq\times V)} = "2";
{(-22,-10)}*+{T(G/K)} = "3";
{(22,-10)}*+{G\times^K\ffq} = "4";
{\ar@{->}^{\cong} "1";"2"};
{\ar@{->}_{T\pi} "1";"3"};
{\ar@{->}_{\cong} "3";"4"};
{\ar@{->}^{p} "2";"4"};
\endxy
\end{equation}
where the top horizontal map is the identification (\ref{Eq:AssociatedBundleIdentification}), the bottom map is the identification:
\begin{equation*}
T(G/K)\cong G\times^K(\ffg/\ffk)\cong G\times^K\ffq,
\end{equation*}
where $G\times^K\ffq:=(G\times\ffq)/K$ is the quotient by the action of the group $K$ on the product $G\times\ffq$ defined by:
\begin{equation*}
k\cdot (g,\xi) = \left(gk^{-1},\Ad(k)\xi\right),
\qquad k\in K, \, (g,\xi)\in G\times\ffq,
\end{equation*}
the map $T\pi$ is the tangent map of the projection $\pi:G\times^KV\to G/K$, and the map $p$ is given by:
\begin{equation*}
[g,v,\xi,w]\mapsto [g,\xi], \qquad [g,v,\xi,w]\in G\times^{K}(V\times\ffq\times V).
\end{equation*}
Using (\ref{Diag:TangentAssociatedBundleIdentification}), the vertical bundle $\calv(G\times^KV):=\ker(T\pi)$ of the associated bundle $G\times^KV\to G/K$ can be identified with the the vector bundle:
\begin{equation}\label{Eq:AssociatedBundleVerticalBundle}
\ker p = G\times^K (V\times\{0\}\times V) \cong G\times^K(V\times V)
\end{equation}
over the base $G\times^KV$.
Thus, we see that the total space of the vertical bundle $\calv(G\times^KV)$ is also an associated bundle over $G/K$(see also \cite[Theorem~10.18~(5)]{KMS93}).
Using the description (\ref{Eq:AssociatedBundleVerticalBundle}) of the vertical bundle $\calv(G\times^KV)$, the connection $\Phi$ in (\ref{Diag:ConnectionDefined}) is such that the following diagram commutes:
\begin{equation*}
\xy
{(-22,10)}*+{T(G\times^KV)} = "1";
{(22,10)}*+{\calv(G\times^KV)} = "2";
{(-22,-10)}*+{G\times^K(V\times\ffq\times V)} = "3";
{(22,-10)}*+{G\times^K(V\times V)} = "4";
{\ar@{->}^{\Phi} "1";"2"};
{\ar@{->}_{\cong} "1";"3"};
{\ar@{->}_{\widehat\Phi} "3";"4"};
{\ar@{->}^{\cong} "2";"4"};
\endxy
\end{equation*}
where the map $\widehat\Phi$ is given by:
\begin{equation*}
\widehat\Phi([g,v,\xi,w]):=[g,v,w], \qquad [g,v,\xi,w]\in G\times^K(V\times\ffq\times V).
\end{equation*}
Hence, the horizontal bundle $\calh:= \ker \Phi$ of the associated bundle $G\times^KV\to G/K$ can be identified with the vector bundle:
\begin{equation*}
\ker\widehat\Phi=G\times^K(V\times\ffq\times\{0\})\cong G\times^K(V\times \ffq),
\end{equation*}
over the base $G\times^KV$.
Thus, we see that the total space of the horizontal bundle $\calh$ is also an associated bundle over $G/K$.
\end{remark}

We now proceed with the ingredients of the proof of Theorem \ref{Thm:EquivalenceVFs} by showing that projecting vector fields on $G\times^KV$ onto the representation $V$ using an equivariant connection is also functorial:

\begin{theorem}\label{Thm:ProjectionFunctor}
Let $V$ be a representation of a compact Lie subgroup $K$ of a Lie group $G$, and let $\ffg=\ffk\oplus\ffq$ be a $K$-equivariant splitting with corresponding $K$-equivariant projection $\bbp:\ffg\to\ffk$.
Let $\Phi$ be the equivariant connection on the bundle $G\times^KV\to G/K$ determined by the splitting $\ffg=\ffk\oplus\ffq$ (Remark \ref{Rem:ConnectionFromSplitting}).
Let $j:V\hookrightarrow G\times^KV$ be the embedding defined by $j(v):=[1,v]$ for $v\in V$ (see Notation \ref{Notation}.\ref{Not:AssociatedBundle}).
Then the map:
\begin{equation*}
P_0:\ffX(G\times^KV)^G\to\ffX(V)^K, \qquad
X\mapsto P_0X,
\end{equation*}
where $P_0X$ is defined by:
\begin{equation*}
(P_0X)(v):= j^* (\Phi\circ X) (v), \qquad
v\in V,
\end{equation*}
and the map:
\begin{equation*}
P_1:C^\infty(G\times^KV,\ffg)^G\to C^\infty(V,\ffk)^K, \qquad
\psi\mapsto P_1\psi,
\end{equation*}
where $P_1\psi$ is defined by:
\begin{equation*}
(P_1\psi)(v):=\bbp(\psi(j(v))),\qquad
v\in V,
\end{equation*}
define a functor $P:\bbx(G\times^KV)^G\to\bbx(V)^K$.
\end{theorem}

\begin{proof}
By Lemma \ref{Lemma:FunctorShortcut} it suffices to check that $P_1$ is a group homomorphism and $P_1$ and $P_0$ interwine the actions as in (\ref{Eq:TargetCommute}).
That $P_1$ is a homomorphism follows immediately from the linearity of the projection $\bbp:\ffg\to\ffk$ in the statement of the Theorem. 
It is also straightforward to verify that $P_0$ is a linear map.

It remains to check condition (\ref{Eq:TargetCommute}).
First, we need to make some observations.
Let $v\in V$ be a point on the slice, and note that the evaluation map $\ev_{(1,v)}:G\to G\times V$ at the point $(1,v)\in G\times V$ of the $G$-action on the product $G\times V$ is such that:
\begin{equation}\label{Eq:Evaluation}
T\ev_{(1,v)}:\ffg\to\ffg\times V,\qquad
T\ev_{(1,v)}(\xi)=(\xi,v).
\end{equation}
Also, if $\ev_{j(v)}:G\times G\times^KV$ is the evaluation map at the point $j(v)=[1,v]\in G\times^KV$ of the $G$-action on $G\times^KV$, then the two evaluation maps make the following diagram commute:
\begin{equation}\label{Diag:EvaluationAndQuotient}
\begin{split}
\xy
(0,18)*+{\ffg}="1";
(34,18)*+{\ffg\times V}="2";
(34,0)*+{T_{j(v)}(G\times^KV)}="3";
{\ar@{->}^{T\ev_{(1,v)}} "1";"2"};
{\ar@{->}_{T\ev_{j(v)}} "1";"3"};
{\ar@{->}^{\varpi|} "2";"3"};
\endxy
\end{split}
\end{equation}
where $\varpi:TG\times TV\to T(G\times^KV)$ is the quotient map of the associated bundle $T(G\times^KV)=TG\times^{TK}TV$ (see Remark \ref{Rem:ConnectionFromSplitting}).

We also need that the following diagam commutes:
\begin{equation}\label{Diag:BoundaryProjection}
\xy
{(-24,12)}*+{C^\infty(G\times^KV,\ffg)^G} = "1";
{(24,12)}*+{\ffX(G\times^KV)^G} = "2";
{(-24,-12)}*+{C^\infty(V,\ffk)^K} = "3";
{(24,-12)}*+{\,\,\,\ffX(V)^K_0} = "4";
{\ar@{->}^{\partial} "1";"2"};
{\ar@{->}_{P_1} "1";"3"};
{\ar@{->}_{\partial} "3";"4"};
{\ar@{->}^{P_0} "2";"4"};
\endxy
\end{equation}
Given an infinitesimal gauge transformation $\psi:G\times^KV\to\ffg$ on the associated bundle $G\times^KV$, it suffices to prove that the vector field $\partial(P_1\psi)$ on the representation $V$ is the unique vector field on $V$ that is $j$-related to the vector field $\Phi\circ \partial(\psi)$, where $\Phi$ is the given connection on the associated bundle $G\times^KV\to G/K$.
For this, let $v\in V$ be a point, and let $\varpi:TG\times TV\to TG\times^{TK}TV$ be the quotient map of the tangent bundle $T\left(G\times^KV\right)$.
Then note that:
\begin{align*}
\Phi_{j(v)}\partial(\psi)(j(v))
&=\Phi_{j(v)}T\ev_{j(v)}\left(\psi j(v)\right)\\
&=\Phi_{j(v)}\varpi T\ev_{(1,v)}\left(\psi(j(v))\right)
\qquad\text{by (\ref{Diag:EvaluationAndQuotient})}\\
&=\Phi_{j(v)}\varpi\left(\psi(j(v)),v\right)
\qquad\text{by (\ref{Eq:Evaluation})}\\
&=\varpi(\bbp\times\id)\left(\psi(j(v)),v\right)
\qquad\text{by (\ref{Diag:ConnectionDefined})}\\
&=\varpi\left(\bbp\psi j(v),v\right)\\
&=\varpi T\ev_{(1,v)} \left(\bbp\psi j(v)\right)
\qquad \text{by (\ref{Eq:Evaluation})}\\
&=\varpi T\ev_{(1,v)}\left(P_1\psi(v)\right)
\qquad \text{by the definition of }P_1\\
&=T\ev_{j(v)}\left(P_1\psi(v)\right)
\qquad \text{by (\ref{Diag:EvaluationAndQuotient})}\\
&=(Tj) T\ev_v \left(P_1\psi(v)\right)
\qquad\text{ by the chain rule}\\
&=(Tj) \partial(P_1\psi)(v).
\end{align*}
Hence, the vector fields $\Phi\circ \partial(\psi)$ and $\partial(P_1\psi)$ are $j$-related, meaning that (\ref{Diag:BoundaryProjection}) commutes as desired.

We now check condition (\ref{Eq:TargetCommute}).
Let $X\in\ffX(G\times^KV)^G$ and let $\psi\in C^\infty(G\times^KV,\ffg)^G$, then:
\begin{align*}
P_1(\psi)\cdot P_0(X)
&=P_0(X) + \partial(P_1(\psi))\\
&=P_0(X) + P_0(\partial(\psi))
\qquad \text{ by diagram (\ref{Diag:BoundaryProjection}) }\\
&= P_0(X + \partial(\psi))
\qquad \text{ by the linearity of }P_0\\
&=P_0(\psi\cdot X).
\end{align*}
Hence, the maps $P_1$ and $P_0$ satisfy (\ref{Eq:TargetCommute}), and more generally the hypotheses of Lemma \ref{Lemma:FunctorShortcut}.
Consequently, $P_1$ and $P_0$ give a functor $P:\bbx(G\times^KV)^G\to\bbx(V)^K$.
\end{proof}

\begin{remark}\label{Rem:ProjectionOnSlices}
Given a proper $G$-manifold $M$, a point $m\in M$ with isotropy $K$, a $K$-invariant splitting $\ffg=\ffk\oplus\ffq$, and a slice $D$ for the action through $m\in M$ with $K$-equivariant embedding $\iota:D\hookrightarrow M$, the projection of Theorem \ref{Thm:ProjectionFunctor} also gives a projection functor:
\begin{equation*}
P:\bbx\left(G\cdot \iota(D)\right)^G \to \bbx(D)^K,
\end{equation*}
where:
\begin{equation*}
G\cdot \iota(D):=\left\{
g\cdot \iota(v) \mid g\in G, \, v\in D
\right\}.
\end{equation*}
This follows from the $G$-equivariant diffeomorphism $G\cdot \iota(D)\cong G\times^K D$ (Remark \ref{Rem:TubularNeighborhood}).
\end{remark}

We now verify part of the equivalence of Theorem \ref{Thm:EquivalenceVFs}.

\begin{theorem}\label{Thm:ExtendThenProject}
Let $V$ be a representation of a compact Lie subgroup $K$ of a Lie group $G$, let $E:\bbx(V)^K\to\bbx(G\times^KV)^G$ be the functor of Theorem \ref{Thm:InclusionFunctor}, and let $P:\bbx(G\times^KV)^G\to\bbx(V)^K$ be a choice of functor as in Theorem \ref{Thm:ProjectionFunctor}.
Then we have that $PE = 1_{\bbx(V)^K}$.
\end{theorem}

\begin{proof}
Let $j:V\hookrightarrow G\times^KV$ be the slice embedding defined by $j(v):=[1,v]$, and let $\bbp:\ffg\to\ffk$ be the equivariant projection corresponding to the choice of functor $P$ (see Theorem \ref{Thm:ProjectionFunctor}).
Note that for any path $\psi$ of infinitesimal gauge transformations on $V$ and any point $v\in V$, we have that:
\begin{equation*}
(P_1E_1\psi)(v)
= \bbp \Big((E_1\psi) j(v)\Big)
= \bbp\Big(\left(\Ad(1)\psi(v)\right)\Big)
=\bbp\Big(\left(\psi(v)\right)\Big)
=\psi(v),
\end{equation*}
where the last equality follows since $\psi(v)\in\ffk$.
Hence, the composition $E_1\circ P_1$ is the identity.

Now consider an arbitrary equivariant vector field $X$ on $V$.
The vector field $(E_0X)$ is vertical in the bundle $G\times^KV\to G/K$, so $(P_0E_0X)$ is the unique vector field that is $j$-related to the vector field $(E_0X)$ by definition of the map $P_0$ (see Theorem \ref{Thm:ProjectionFunctor}).
On the other hand, 
\begin{equation*}
(E_0X) j(v)
=1\cdot (Tj)X(v)
=(Tj)X(v),
\qquad v\in V.
\end{equation*}
Hence, the vector field $X$ is $j$-related to the vector field $(E_0X)$.
Thus, the vector fields $P_0(E_0(X))$ and $X$ are equal.
Hence, the composition $E_0\circ P_0$ is the identity, completing the proof.
\end{proof}

We now proceed to construct the natural isomorphism that gives the decomposition in (\ref{Eq:Decomposition}).
We need a technical lemma:

\begin{lemma}\label{Lemma:HorizontalTriviality}
Let $V$ be a representation of a compact Lie subgroup $K$ of a Lie group $G$.
Furthermore,
\begin{enumerate}
\item Let $\ffg=\ffk\oplus\ffq$ be a $K$-equivariant splitting of the Lie algebra $\ffg$ of the Lie group $G$, where $\ffk$ is the Lie algebra of the Lie subgroup $K$.
\item Let $\calh\to G\times^KV$ be the horizontal bundle of the connection $\Phi$ induced by the splitting $\ffg=\ffk\oplus\ffq$ on the associated bundle $G\times^KV\to G/K$ (Remark \ref{Rem:ConnectionFromSplitting}).
That is, the total space $\calh$ of the horizontal bundle is $\calh:=G\times^K(V\times\ffq)$.
\item Let $\Gamma(\calh \to G\times^KV)^G$ be the space of equivariant vector fields on the associated bundle $G\times^KV$ that are horizontal with respect to the connection $\Phi$.
\end{enumerate}
Then the map:
\begin{equation}\label{Eq:HorizontalTrivialityMap}
C^\infty(V,\ffq)^K\to \Gamma(\calh\to G\times^KV)^G,
\qquad \psi\mapsto X^\psi,
\end{equation}
where the vector field $X^\psi$ is defined by:
\begin{equation*}
X^\psi([g,v]):=T\ev_{[g,v]}\left(\Ad(g)\psi(v)\right), \qquad
[g,v]\in G\times^KV,
\end{equation*}
is a linear isomorphism.
\end{lemma}

\begin{remark}
Recall from Remark \ref{Rem:ConnectionFromSplitting} that the total space of the horizontal bundle $\calh$ in the statement of Lemma \ref{Lemma:HorizontalTriviality} is canonically a vector bundle over $G\times^KV$ but an associated bundle over $G/K$.
Since $\calh:=G\times^K(V\times \ffq)$ is an associated bundle over $G/K$, it is a standard fact that there is a bijection:
\begin{equation*}
\Gamma\Big(\calh\to G/K \Big) \cong C^\infty(G,V\times\ffq)^K,
\end{equation*}
where the space on the left-hand side is the space of sections of the associated bundle projection $\calh\to G/K$ (see, for example, \cite[Theorem~10.12]{KMS93}).
Lemma \ref{Lemma:HorizontalTriviality} is a similar result, but for the space of $G$-equivariant sections of the vector bundle projection $\calh\to G\times^KV$.
\end{remark}

\begin{proof}[Proof of Lemma \ref{Lemma:HorizontalTriviality}]
We show that the map in (\ref{Eq:HorizontalTrivialityMap}) is such that the following diagram commutes, where all the other maps are linear isomorphisms:
\begin{equation}\label{Diag:HorizontalTriviality}
\begin{gathered}
\xy
{(-20,10)}*+{C^\infty(V, \ffq)^K} = "1";
{(20,10)}*+{\Gamma(\calh)^G} = "2";
{(-20,-10)}*+{\Gamma(V\times \ffq)^K} = "3";
{(20,-10)}*+{\Gamma(\calh \,|\, V)^K} = "4";
{\ar@{->} "1";"2"};
{\ar@{->}_{1_{V}\times \,(-)} "1";"3"};
{\ar@{->}_{\eta_*} "3";"4"};
{\ar@{->}_{\epsilon} "4";"2"};
\endxy
\end{gathered}
\end{equation}
where $V\times \ffq \xrightarrow{\pr_{V}} V$ is a trivial bundle with the action of the group $K$ on the total space given by:
\begin{equation*}
k\cdot (v,\xi) := (k\cdot v,\,\Ad(k)\xi), \qquad
k\in K,\, (v,\xi) \in  V\times\ffq,
\end{equation*}
and the action of $K$ on $V$ being the given representation of $K$ on $V$.

The map $1_V\times\,(-)$ in (\ref{Diag:HorizontalTriviality}) is the map defined by $\psi \mapsto 1_{V}\times\psi$.
This map is also a linear isomorphism since its inverse is given by the pushforward of the bundle projection $\pr_{V}: V\times\ffq\to V$.

To define the map $\eta_*$, note that the restricted bundle $\calh|\, V\to V$ is trivializable:
\begin{equation*}
\calh|V
\cong K\times ^K (V\times \ffq)
\cong  V\times \ffq.
\end{equation*}
Explicitly, the isomorphism is given by:
\begin{equation*}
\eta: V\times\ffq \to\calh |\, V,\qquad
(v,\xi)\mapsto T\ev_{[1,v]}(\xi).
\end{equation*}
The map $\eta_*$ is the pushforward of this diffeomorphism $\eta$.
The linearity of the map $\eta$ can be verified directly.
The inverse of $\eta_*$ is given by the pushforward of the inverse of $\eta$.

The map $\epsilon$ in (\ref{Diag:HorizontalTriviality}) corresponds to equivariant extension and is defined by $X\mapsto \epsilon (X)$ where:
\begin{equation}\label{Eq:EquivariantExtension}
\epsilon (X)([g,v]):=g\cdot X([1,v]),
\end{equation}
where the action is as in Notation (\ref{Notation}.\ref{Not:TangentAction}).
The linearity can be verified directly.
On the other hand, the inverse of the map $\epsilon$ is the pullback by the slice embedding $j: V\hookrightarrow(G\times^KV)$ defined by $j(v):=[1,v]$.
By all of the above, the map in (\ref{Eq:HorizontalTrivialityMap}) can be factored as in (\ref{Diag:HorizontalTriviality}) into linear isomorphisms, completing the proof.
\end{proof}

\begin{theorem}\label{Thm:NaturalIsomorphism}
Let $V$ be a representation of a compact Lie subgroup $K$ of a Lie group $G$, let $E:\bbx(V)^K\to\bbx(G\times^KV)^G$ be the functor of Theorem \ref{Thm:InclusionFunctor}, and let $P:\bbx(G\times^KV)^G\to\bbx(V)^K$ be a choice of functor as in Theorem \ref{Thm:ProjectionFunctor} corresponding to an equivariant splitting $\ffg=\ffk\oplus\ffq$.
Then there exists a natural isomorphism $h:EP\cong 1_{\bbx(G\times^KV)^G}$.
That is, for any choice of functor $P$ as in Theorem \ref{Thm:ProjectionFunctor}, there is a linear map:
\begin{equation}\label{Eq:TheMaph}
h:\ffX(G\times^KV)^G\to C^\infty(G\times^KV,\ffg)^G,\qquad
X\mapsto h(X),
\end{equation}
such that:
\begin{equation}\label{Eq:DecompositionRevisited}
X=E_0(P_0(X))+\partial(h(X)),\qquad
X\in\ffX(G\times^KV)^G.
\end{equation}
Furthermore, for every equivariant vector field $X\in\ffX(G\times^KV)^G$, the restriction of the infinitesimal gauge transformation $h(X):G\times^KV\to\ffg$ to $\{[1,v]\mid v\in V\}\cong V$ takes values in $\ffq$.
\end{theorem}

\begin{proof}
Let $\ffg=\ffk\oplus\ffq$ be the $K$-equivariant splitting giving the choice functor $P$, let $\Phi\in\Omega^1(G\times^KV;\calv(G\times^KV))^G$ be the induced connection on the bundle $G\times^KV\to G/K$, and let $\calh\to G\times^KV$ be the corresponding horizontal bundle (see Remark \ref{Rem:ConnectionFromSplitting}).
We define the map $h$ in (\ref{Eq:TheMaph}) as the composition filling in the dashed arrow in the following diagram:
\begin{equation}\label{Diag:NatIso}
\begin{gathered}
\xy
{(-30,0)}*+{\ffX(G\times^KV)^G} = "1";
{(30,0)}*+{C^\infty(G\times^KV,\ffg)^G} = "2";
{(-30,15)}*+{\Gamma(\calh)^G} = "3";
{(30,15)}*+{C^\infty(V,\ffg)^K} = "4";
{(0,30)}*+{C^\infty(V,\ffq)^K} = "5";
{\ar@{-->}^{h} "1";"2"};
{\ar@{->}^{\alpha} "1";"3"};
{\ar@{->}^{\beta} "3";"5"};
{\ar@{->}^{\iota_*} "5";"4"};
{\ar@{->}^{\epsilon} "4";"2"};
\endxy
\end{gathered}
\end{equation}
where the maps in this diagram are defined as follows.

The map $\alpha$ is the map defined by $X\mapsto X-E_0(P_0(X))$.
This is well-defined since the vector field $X-E_0(P_0(X))$ is horizontal in the bundle $G\times^KV\to G/K$ because the space $\Gamma(\calh)^G$ is the kernel of the linear map $P_0$ and $P_0(E_0(Y))=Y$ for any $Y\in\ffX(V)^K$ by Theorem \ref{Thm:ExtendThenProject}.
The map $\beta$ is the inverse of the linear isomorphism in the statement of Lemma \ref{Lemma:HorizontalTriviality}.
The map $\iota_*$ is the pushforward of the canonical inclusion $\iota:\ffq\hookrightarrow \ffg$, where we have used the usual identifications.
Finally, the map $\epsilon$ is the equivariant extension map defined by $\psi\mapsto \epsilon(\psi)$ where:
\begin{equation}\label{Eq:EpsilonForIso}
\epsilon(\psi)([g,v]):=\Ad(g)\psi(v), \qquad
[g,v]\in G\times^KV.
\end{equation}

We now show that equation (\ref{Eq:DecompositionRevisited}) holds.
For this, note that for all $[g,v]\in G\times^KV$ we have that:
\begin{align*}
\partial & (\epsilon(\iota_*(\psi))) ([g,v]) \\
&=T \ev_{[g,v]}(\epsilon(\iota_*(\psi))([g,v)) \\
&=T \ev_{[g,v]}(\Ad(g)\iota_*(\psi(v)))) \\
&=T \ev_{[g,v]}(\Ad(g)(\psi(v))) \\
&= \beta^{-1}(\psi)([g,v]).
\end{align*}
Hence, the following diagram commutes:
\begin{equation*}
\xy
{(-26,10)}*+{C^\infty(V,\ffq)^K} = "1";
{(26,10)}*+{\Gamma(\calh)^G} = "2";
{(-26,-10)}*+{C^\infty(G\times^KV,\ffg)^G} = "3";
{(26,-10)}*+{\ffX\left(G\times^KV\right)^G} = "4";
{\ar@{->}^{\beta^{-1}} "1";"2"};
{\ar@{->}_{\epsilon\circ\iota_*} "1";"3"};
{\ar@{->}_{\partial} "3";"4"};
{\ar@{->}^{} "2";"4"};
\endxy
\end{equation*}
where the right vertical map is the canonical inclusion (which is well-defined since $\Gamma(\calh)^G=\ker P_0$).
Therefore, for all equivariant vector fields $X$ on the associated bundle $G\times^KV$ we have that:
\begin{equation*}
\partial(h(X))
=\partial\circ\epsilon\circ\iota_*\circ\beta\circ\alpha(X)
=\beta^{-1}\circ\beta\circ\alpha(X)
=\alpha(X)
=X-E_0(P_0(X)),
\end{equation*}
which is what we needed.
That is, equation (\ref{Eq:DecompositionRevisited}) holds, which says that for all objects $X$ of the category of equivariant vector fields $\bbx(G\times^KV)^G$ we have an isomorphism $h(X)$ with source $E_0(P_0(X))$ and target $X$.

It remains to verify that this transformation is natural.
That is, given equivariant vector fields $X$ and $Y$ on $G\times^KV$ and an infinitesimal gauge transformation $\psi:G\times^KV\to \ffg$ such that $Y=X+\partial(\psi)$, naturality of the transformation $h$ consists of verifying that the following diagram commutes:
\begin{equation}\label{Diag:FirstNaturalitySquare}
\xy
{(-26,10)}*+{E_0(P_0(X))} = "1";
{(26,10)}*+{X} = "2";
{(-26,-10)}*+{E_0(P_0(Y))} = "3";
{(26,-10)}*+{Y} = "4";
{\ar@{->}^{\Big(E_0(P_0(X)),\,h(X)\Big)} "1";"2"};
{\ar@{->}_{\Big(E_0(P_0(X),\,E_1(P_1(\psi))\Big)} "1";"3"};
{\ar@{->}_{\Big(E_0(P_0(Y)),\,h(Y)\Big)} "3";"4"};
{\ar@{->}^{(X,\psi)} "2";"4"};
\endxy
\end{equation}
where the arrows represent isomorphisms in the category $\bbx(G\times^KV)^G$ of equivariant vector fields on $G\times^KV$ (see Remark \ref{Rem:ActionGroupoids}).
Since composition of morphisms in $\bbx(G\times^KV)^G$ corresponds to addition of the infinitesimal gauge transformations, the equality of morphisms in diagram (\ref{Diag:FirstNaturalitySquare}) corresponds to the equality:
\begin{equation*}
\Big(E_0(P_0(X)),\, h(Y) + E_1(P_1(\psi))\Big)
=
\Big(E_0(P_0(X)),\, \psi+h(X)\Big).
\end{equation*}
Thus, naturality amounts to verifying the equation:
\begin{equation}\label{Eq:Naturality}
h\Big(X+\partial(\psi)\Big) + E_1(P_1(\psi)) = \psi + h(X),
\end{equation}
where $X$ is an arbitrary equivariant vector field on $G\times^KV$ and $\psi$ is an arbitrary infinitesimal gauge transformation on $G\times^KV$.

To verify equation (\ref{Eq:Naturality}), we first need to make an observation.
Let $\varphi\in C^\infty(G\times^KV,\ffq)^K$ be an infinitesimal gauge transformation on $G\times^KV$ such that $\phi(v)\in\ffq$ for all $v\in V$.
We can pull this back by the embedding $j:V\hookrightarrow G\times^KV$ to obtain an infinitesimal gauge transformation $\varphi|\,\in C^\infty(V,\ffq)^K$ on $V$ defined by:
\begin{equation}\label{Eq:RestrictedPath}
\varphi | (v) := \varphi(j(v)) \equiv \varphi([1,v]),\qquad
v\in V.
\end{equation}
Applying the map $\beta^{-1}$ of Lemma \ref{Lemma:HorizontalTriviality} returns the equivariant field induced by the unrestricted map $\varphi$.
That is, $\beta^{-1}(\varphi|)=\partial(\varphi)$, which in turn implies that $\varphi | = \beta(\partial(\varphi))$.
Now, if $\psi\in C^\infty(G\times^KV,\ffg)^G$ is an arbitrary infinitesimal gauge transformation, then the infinitesimal gauge transformation $\psi- E_1(P_1(\psi))$ is a map in $C^\infty\left(G\times^KV,\ffq\right)^K$ by the definition of $P_1$ (see the statement of Theorem \ref{Thm:ProjectionFunctor}).
Hence, we have that:
\begin{equation}\label{Eq:RestrictedPathOfInterest}
\beta\partial\Big(\left(\psi- E_1(P_1(\psi))\right)\Big) = \Big(\psi- E_1(P_1(\psi)) \Big)\Big|,
\end{equation}
where the map $\Big(\psi- E_1(P_1(\psi)) \Big)\Big|$ is a restriction as in (\ref{Eq:RestrictedPath}).

With this observation in hand, we can proceed to verify the naturality of $h$.
Let $X$ be an arbitrary equivariant vector field on $G\times^KV$ and let $\psi$ be an arbitrary infinitesimal gauge transformation on $G\times^KV$. 
Then note that:
\begin{align*}
h\Big(X&+\partial(\psi)\Big) + E_1(P_1(\psi))\\
&=h(X) + h(\partial(\psi)) + E_1(P_1(\psi)) 
\qquad\qquad \text{by the linearity of }h\\
&=h(X) + \epsilon\iota_*\beta\alpha(\partial(\psi)) + E_1(P_1(\psi)) 
\qquad\qquad \text{by (\ref{Diag:NatIso})}\\
&=h(X) + \epsilon\iota_*\beta\Big(\partial\Big(\psi+E_1(P_1(\psi))\Big)\Big) \\
&\qquad\qquad\qquad\qquad\qquad\qquad + E_1(P_1(\psi)) \qquad \text{by the definition of }\alpha\\
&=h(X) + \epsilon\iota_* \Big( \Big(\psi- E_1(P_1(\psi)) \Big)\Big|\, \Big) + E_1(P_1(\psi)) 
\qquad\text{by (\ref{Eq:RestrictedPathOfInterest})}\\
&=h(X) + \psi -E_1(P_1(\psi)) + E_1(P_1(\psi)) \\
&\qquad\qquad\qquad\qquad\qquad\qquad\qquad
\text{by the definition of }\iota_*\text{ and }\epsilon\text{ in (\ref{Eq:EquivariantExtension})}\\
&= h(X) + \psi,
\end{align*}
which confirms the naturality equation (\ref{Eq:Naturality}).
This proves that $h$ is a {\it natural} transformation.
It is actually a natural {\it isomorphism} since every morphism in the category $\bbx(G\times^KV)^G$ is invertible (recall it is an action groupoid).

Finally, for every $X\in\ffX(G\times^KV)^G$ and for any $v\in V$:
\begin{align*}
&h(X)([1,v])\\
&= \left((\epsilon\circ\iota_*\circ\beta\circ\alpha)(X)\right)([1,v])
\qquad \text{by diagram (\ref{Diag:NatIso})}\\
&=\beta(\alpha(X))(v),
\end{align*}
where the second equality follows since $\iota_*$ is an inclusion and by the definition of $\epsilon$ in (\ref{Eq:EpsilonForIso}).
Hence, by the definition of $\beta$, the restriction $h(X)|V$ of the infinitesimal gauge transformation $h(X):G\times^KV\to\ffg$ to $\{[1,v]\in G\times^KV \mid v\in V\} \cong V$ is such that:
\begin{equation*}
h(X)|V = \beta(\alpha(X)) \in C^\infty(V,\ffq)^K,
\end{equation*}
meaning that $h(X)|V$ takes values in $\ffq$ as claimed.
\end{proof}

We now use all of the results proved so far in this subsection to prove Theorem \ref{Thm:EquivalenceVFs}.

\begin{proof}[Proof of Theorem \ref{Thm:EquivalenceVFs}]
We prove the equivalence $\bbx(G\times^KV)^G\simeq \bbx(V)^K$ by exhibiting functors and natural isomorphisms as in Notation \ref{Notation}.\ref{Not:Categories}.
There is a canonical functor $E:\bbx(V)^K\to\bbx(G\times^KV)^G$ by Theorem \ref{Thm:InclusionFunctor}, and there is a choice of functor $P:\bbx(G\times^KV)^G\to\bbx(V)^K$ by Theorem \ref{Thm:ProjectionFunctor}. 
For such a pair of functors we proved in Theorem \ref{Thm:ExtendThenProject} that $PE = 1_{\bbx(V)^K}$ , and in Theorem \ref{Thm:NaturalIsomorphism} we constructed a natural isomorphism $EP\cong 1_{\bbx(G\times^KV)^G}$.
Thus, we have constructed an equivalence $\bbx(G\times^KV)^G\simeq \bbx(V)^K$.
\end{proof}

We can now prove the global decomposition in Theorem \ref{Thm:Decomposition}:

\begin{proof}[Proof of \ref{Thm:Decomposition}]
Let $V$ be the canonical slice representation $T_mM/T_m(G\cdot m)$ of the isotropy $K$ of the relative equilibrium $m$ of $X$.
By Theorem \ref{Thm:EquivalenceVFs} the canonical inclusion functor $E:\bbx(V)^K\to \bbx(G\times^KV)^G$ is part of an equivalence of categories $\bbx(V)^K\simeq \bbx(G\times^KV)^G$.
Let $P:\bbx(G\times^KV)\to \bbx(V)^K$ be a choice of projection functor as in Theorem \ref{Thm:ProjectionFunctor}.
Furthermore, let $D$ be a $K$-invariant neighborhood of the origin in $V$ such that there exist a $K$-equivariant embedding $\iota:D\hookrightarrow M$ such that (\ref{Diag:SliceEmbeddingDiagram}) holds.
It is straightforward to verify that all the maps in the equivalence $\bbx(V)^K\simeq \bbx(G\times^KV)^G$ restrict to give an equivalence $\bbx(D)^K\simeq \bbx(G\times^KD)^G$.
In particular, there exists an infinitesimal gauge transformation $h^X:G\cdot D\to\ffg$ such that:
\begin{equation}\label{Eq:FirstStep}
X|_
{G\cdot D} = E_0P_0| ( X|_{G\cdot D} ) + \partial (h^X).
\end{equation}

We now extend the infinitesimal gauge transformation $h^X$ to an infinitesimal gauge transformation $\psi^X$ on $M$.
For this, we construct a particular $K$-invariant smooth bump function.
Let $B$ be a $K$-invariant open ball around the origin $(0,0)\in V$ contained in $D$, let $\widehat B$ be a $K$-invariant closed ball in $D$ containing the ball $B$, and consider a $K$-invariant smooth bump function $\mu:D\to \bbr$.
That is, the function $\mu$ is such that:
\begin{equation}\label{Eq:KinvBump}
\begin{split}
\mu(v)=\begin{cases}
\phantom{0 \le \mu(}1 & \text{ if } v\in B \\
0 \le \mu(v) \le 1 & \text{ if }  v\in \widehat B - B \\
\phantom{0 \le \mu(}0 & \text{ if } v\in D-\widehat B
\end{cases}
\end{split}
\end{equation}
and $\mu(k\cdot v)=\mu(v)$ for all $k\in K$ and $v\in D$.
To construct such a bump function, just take any smooth bump function satisfying (\ref{Eq:KinvBump}), then average that bump function with respect to the action of $K$ to obtain the desired $K$-invariant smooth bump function, which is possible since $K$ is a compact Lie group and the sets $B$ and $\widehat B$ are $K$-invariant.

We can extend the bump function $\mu:D\to\bbr$ to a bump function $\epsilon (\mu): G\cdot D\to \bbr$ defined by:
\begin{equation*}
\epsilon(\mu)(g\cdot v):=\mu (v)
\end{equation*}
That is, $\epsilon(\mu)$ is a bump function that is $1$ on the $G$-invariant neighborhood $G\cdot B$ of $G\cdot m$ and $0$ for points in $G\cdot D$ outside $G\cdot\widehat B$.
We verify that this is well-defined.
First, note that if $g\cdot \iota(v)=g'\cdot \iota(v')\in G\cdot D$, then $g^{-1}g'\in K$ since $g^{-1}g'\cdot \iota(v')=\iota(v)\in \iota(D)$ and $\iota(D)\cong D$ is a slice.
Hence, since $v=g^{-1}g'\cdot v'$ because $\iota$ is a $K$-equivariant embedding and $g^{-1}g'\in K$, we have that:
\begin{equation*}
\mu(v)
=\mu\left(g^{-1}g'\cdot v'\right)
=\mu\left(v'\right),
\end{equation*}
where the last equality follows since $\mu$ is $K$-invariant.
Hence, $\epsilon(\mu)$ is well-defined on $G\cdot D$, and is a bump function as desired.
Furthermore, the bump function $\epsilon(\mu)$ is $G$-invariant.
Let $g\cdot \iota(v)\in G\cdot D$ and let $g'\in G$.
Then:
\begin{equation*}
\epsilon(\mu)(g' \cdot (g\cdot \iota(v)))
=\epsilon(\mu)(g'g\cdot \iota(v))
=\mu(v)
=\epsilon(\mu)(g\cdot \iota(v)).
\end{equation*}
Hence, $\epsilon(\mu):G\cdot D\to\bbr$ is a $G$-invariant bump function.

We can now define the desired infinitesimal gauge transformation $\psi^X:M\to \ffg$ by:
\begin{equation*}
\psi^X(m):=\begin{cases}
\left(\epsilon(\mu)(m)\right)h^X(m) & \text{ for }m\in G\cdot D\\
0 & \text{for }m\in M-G\cdot D.
\end{cases}
\end{equation*}
This map is well-defined since $\epsilon(\mu)$ is $0$ on $G\cdot D - G\cdot \widehat B$, meaning that the product $\epsilon(\mu)h^X$ extends by $0$ to all of $M$ as desired.
It suffices to check on $G\cdot D$ to verify that this map is equivariant.
Thus, let $g\in G$ and $ p\in G\cdot D$, then:
\begin{align*}
\psi^X(g\cdot p)
&=\left(\epsilon(\mu)(g\cdot p)\right) h^X(g\cdot p) \\
&=\left(\epsilon(\mu)(p)\right) h^X(g\cdot p) \qquad \text{since }\epsilon(\mu) \text{ is $G$-invairant}\\
&=\left(\epsilon(\mu)(p)\right) \Ad(g)\left(h^X(p)\right) \qquad\text{by the equivariance of }h^X\\
&=\Ad(g)\left(\left(\epsilon(\mu)(p)\right)\, h^X(p)\right) \qquad \text{by the linearity of }\Ad(g)\\
&=\Ad(g)\left(\psi^X(p)\right).
\end{align*}
Thus, $\psi^X$ is an infinitesimal gauge transformation on $M$.

Now note that $Y^X:=X-\partial(\psi^X)$ is an equivariant vector field on $G\times^KV$ such that for $p\in G\cdot B$ we have:
\begin{align*}
Y^X(p)
&=X(p)-\partial(\psi^X(p))\\
&=X(p)-\partial(h^X(p)) \qquad \text{ by the definition of }\psi^X\\
&=E_0(P_0(X(p))) \qquad \text{ by (\ref{Eq:FirstStep})}.
\end{align*}
Thus, the vector field $Y^X$ is transverse to the group orbit $G\cdot m$.
Furthermore, the vector $Y^X(m)=E_0(P_0(X(m)))$ is the vertical part of the vector $X(m)$ in the associated bundle $G\times^KD\to G/K$.
That is, it is tangent to the slice $\iota(D)$.
On the other hand, since $\iota(D)$ for the action at $m$, we have that:
\begin{equation*}
T_mM=T_m(G\cdot m)\oplus T_m\iota(D).
\end{equation*}
And since the point $m$ is a relative equilibrium of $X$ then $X(m)\in T_m(G\cdot m)$, meaning that the vertical part of the vector $X(m)$ is $0$.
That is, $Y^X(m)=E_0(P_0(X(m)))=0$, so $m$ is an equilibrium of the vector field $Y^X$.
Thus, we have constructed the decomposition:
\begin{equation*}
X=Y^X+\partial(\psi^X)
\end{equation*}
in the statement of the theorem.
\end{proof}

We show that different choices of projection functors lead to isomorphic vector fields (see Remark \ref{Rem:LiteratureObservation1}, and also compare with \cite[Lemma~3.17]{L15}):

\begin{proposition}\label{Prop:ProjectionChoice}
Let $V$ be a representation of a compact Lie subgroup $K$ of a Lie group $G$, let $\ffg=\ffk\oplus\ffq_1$ and $\ffg=\ffk\oplus\ffq_2$ be two $K$-equivariant splittings and let $P^1:\bbx(G\times^KV)^G\to \bbx(V)^K$ and $P^2:\bbx(G\times^KV)^G\to \bbx(V)^K$ be the two projection functors corresponding, respectively, to these splittings as in Theorem \ref{Thm:ProjectionFunctor}.
Then for every equivariant vector field $X\in\ffX(G\times^KV)^G$ we have that:
\begin{equation*}
P^1_0(X)=P^2_0(X)+\partial\left(P^1_1\left(h^2(X)\right)\right),
\end{equation*}
where $h^2:\ffX(G\times^KV)^G\to C^\infty(G\times^KV,\ffg)^G$ is the map corresponding to the natural isomorphism $E\circ P^2\cong 1_{\bbx(G\times^KV)^G}$.
In particular, the projected vector fields $P^1_0(X)$ and $P^2_0(X)$ on $V$ are isomorphic.
\end{proposition}

\begin{proof}
Let $X\in \ffX(G\times^KV)^G$ be an equivariant vector field.
By Theorem \ref{Thm:EquivalenceVFs}, there exist $h^1:\ffX(G\times^KV)^G\to C^\infty(G\times^KV,\ffg)^G$ and $h^2:\ffX(G\times^KV)^G\to C^\infty(G\times^KV,\ffg)^G$ such that:
\begin{equation*}
E_0(P^1_0(X)) + \partial(h^1(X)) = X = E_0(P^2_0(X)) + \partial(h^2(X)).
\end{equation*}
Therefore, by the linearity of $E_0$ and $\partial$, we have:
\begin{equation}\label{Eq:ChoiceEquality1}
E_0\left(P^1_0(X)-P^2_0(X)\right) = \partial\left(h^2(X)-h^1(X)\right).
\end{equation}
Now let $j:V\hookrightarrow G\times^KV$ be the embedding defined by $j(v):=[1,v]$, and observe that:
\begin{equation}\label{Eq:P1Kernel}
\ker \left(P^1_1\right)
=\left\{
\psi\in C^\infty(G\times^KV,\ffg)^G \mid \text{im}\left(\psi|j(V)\right) \subseteq \ffq_1
\right\}
\end{equation}
by the definition of $P^1_1$ (Theorem \ref{Thm:ProjectionFunctor}).
On the other hand, the image of the restriction to $j(V)$ of the map $h^1(X)\in C^\infty(G\times^KV,\ffg)^G$ is contained in $\ffq_1$ by (\ref{Diag:NatIso}).
Therefore, we have that $h^1(X)\in \ker P^1_1$ by (\ref{Eq:P1Kernel}).
Thus:
\begin{align*}
P^1_0(X)&-P^2_0(X)\\
&= P^1_0\left(E_0\left(P^1_0(X)-P^2_0(X)\right)\right) \qquad \text{since }P^1_0E_0 =\text{id} \text{ by Theorem \ref{Thm:ExtendThenProject}} \\
&= P^1_0\left( \partial\left(h^2(X)-h^1(X)\right) \right) \qquad \text{by (\ref{Eq:ChoiceEquality1})}\\
&=\partial\left ( P^1_1\left(h^2(X)-h^1(X)\right) \right) \qquad \text{by (\ref{Diag:BoundaryProjection})}\\
&=\partial\left ( P^1_1\left(h^2(X)\right)-P^1_1\left(h^1(X)\right) \right) \qquad \text{by the linearity of }P^1_1\\
&=\partial\left(P^1_1\left(h^2(X)\right)\right) \qquad\qquad \text{since } h^1(X)\in \ker P^1_1.
\end{align*}
This is what we wanted to prove.
\end{proof}

Next we show that the choice of slice leads to isomorphic vector fields.
For this it suffices to consider a slice inside an associated bundle as follows:

\begin{proposition}\label{Prop:SliceChoice}
Let $V$ be a representation of a compact Lie subgroup $K$ of a Lie group $G$, let $\ffg=\ffk\oplus\ffq$ be a $K$-invariant splitting, and let:
\begin{equation*}
P^V:\bbx(G\times^KV)^G\to \bbx(V)^K
\end{equation*}
be the equivariant projection functor with respect to the given splitting of $\ffg$ (Theorem \ref{Thm:ProjectionFunctor}).
Furthermore, suppose $D$ is a slice through $[1,0]\in G\times^KV$ for the action of $G$ on the associated bundle $G\times^KV$, let $\iota:D\hookrightarrow G\times^KV$ be the $K$-equivariant embedding of the slice, and let:
\begin{equation*}
P^D:\bbx(G\cdot\iota(D))^G\to \bbx(D)^K
\end{equation*}
be the equivariant projection functor with respect to the same splitting of $\ffg$ (Theorem \ref{Thm:ProjectionFunctor} and Remark \ref{Rem:ProjectionOnSlices}).
Then, after perhaps shrinking $D$, there exists a $K$-equivariant embedding $\phi:D\hookrightarrow V$ such that, for any equivariant vector field $X\in\ffX(G\cdot \iota(D))^G$, the vector fields $P^V_0(X)$ and $\phi_*\left(P^D_0(X)\right)$ are isomorphic on $V$.
In particular, we have:
\begin{equation}\label{Eq:SliceChoiceEquation}
P^V_0(X)=
\phi_*\Big(P^D_0(X)\Big)
+ \partial\Big( P^V_1\left(h^D(X)\right)\Big),
\end{equation}
where $h^D:\ffX(G\cdot \iota(D))^G\to C^\infty(G\cdot \iota(D),\ffg)^G$ is the map corresponding to the natural isomorphism $E^DP^D \cong 1_{\bbx(G\cdot\iota(D))^G}$ as in Theorem \ref{Thm:NaturalIsomorphism}.
\end{proposition}

\begin{proof}
There are two main steps in this proof: (1) constructing the embedding $\phi:D\to V$, and (2) constructing an isomorphism $P^V_0(X) \cong \phi_*\Big(P^D_0(X)\Big)$ such that (\ref{Eq:SliceChoiceEquation}) holds.
Our argument for the first step follows the argument in \cite[Lemma~3.11]{L15}, while the second step is different (see also the comments at the beginning of this section).
Consider the principal $K$-bundle $\pi:G\to G/K$.
Using the $K$-invariant splitting $\ffg=\ffk\oplus\ffq$, there exists a $K$-invariant neighborhood $\calo\subseteq \exp(\ffq)$ of the identity $1\in G$ and a $K$-equivariant section of $\pi:G\to G/K$ given by:
\begin{equation*}
s: \calo/K \to \calo, \qquad gK \mapsto g.
\end{equation*}
This in turn trivializes the associated bundle $\varpi:G\times^KV\to G/K$ via:
\begin{equation*}
\varphi:\calo/K \times V \xrightarrow{\cong} \varpi^{-1}(\calo/K), \qquad (gK,v)\mapsto s(gK)\cdot j(v) = \left[s(gK),v\right].
\end{equation*}
By perhaps shrinking $D$ we may assume $\varphi^{-1}(\iota(D))\subseteq \calo/K \times V$ and since $D$ is a slice at $[1,0]\in G\times^KV$ we have:
\begin{equation*}
T_{(K,0)}(\calo/K \times V) = T_{(K,0)}(\calo/K) \oplus T_{(K,0)} D.
\end{equation*}
Thus, by perhaps shrinking $D$ again, we may assume without loss of generality that the restriction to $\varphi(\iota(D))$ of the projection $\pr_2:\calo/K\times V \to V$ is a $K$-equivariant diffeomorphism onto its image in $V$.
Thus, the desired map $K$-equivariant embedding is the map:
\begin{equation*}
\phi:=\left(\pr_2|D\right)^{-1}\circ \varphi^{-1} \circ \iota : D \to V.
\end{equation*}
Note in particular that $\phi$ is given by a map $f:D\to \calo$ such that:
\begin{equation}\label{Eq:SliceDiffeoRelationship}
j(\phi(d))= f(d)_{G\times^KV}(\iota(d)), \qquad d\in D,
\end{equation}
where $f(d)_{G\times^KV}: G\times^KV\to G\times^KV$ is the diffeomorphism given by $[g,v]\mapsto f(d)\cdot [g,v]$.

We now proceed to show that the map $\phi$ is as desired.
Given an equivariant vector field $X\in \ffX(G\cdot \iota (D))^G$, our first step is to relate the two different decompositions as in Theorem \ref{Thm:EquivalenceVFs} in this context.
Let:
\begin{align*}
&E^D_0:\ffX(G\cdot \iota(D))^G\hookrightarrow \ffX(D)^K \\
&E^V_0:\ffX(G\times^KV)^G\hookrightarrow \ffX(V)^K
\end{align*}
be the equivariant inclusions (Theorem \ref{Thm:InclusionFunctor}), and let:
\begin{align*}
&h^D : \ffX(G\cdot \iota(D))^G \to C^\infty(G\cdot\iota(D),\ffg)^G\\
&h^V : \ffX(G\times^KV)^G \to C^\infty(G\times^KV,\ffg)^G
\end{align*}
be the maps corresponding, respectively, to the natural isomorphisms $E^DP^D\cong 1_{\bbx(G\cdot\iota(D))^G}$ and $E^VP^V\cong 1_{\bbx(G\times^KV)^G}$ (Theorem \ref{Thm:NaturalIsomorphism}).
By Theorem \ref{Thm:EquivalenceVFs}, we have decompositions:
\begin{equation*}
E^D_0\left(P^D_0(X)\right) + \partial\left( h^D(X)\right)
= X = 
E^V_0\left(P^V_0(X)\right) + \partial\left( h^V(X)\right).
\end{equation*}
Hence, in particular, we have that:
\begin{equation}\label{Eq:DifferenceProjections}
E^V_0\left(P^V_0(\phi_*X)\right) - E^D_0\left(P^D_0(X)\right)
=
\partial\left( h^D(X) -  h^V(\phi_*X)\right).
\end{equation}

Now we relate the equivariant extensions from each of the slices.
Given an equivariant vector field $Y\in\ffX(D)^K$, note that for all $d\in D$:
\begin{align*}
E^D_0(Y)&(j(\phi(d)))\\
&=E^D_0(Y)\Big(f(d)_{G\times^KV}\left(\iota(d)\right)\Big) \qquad \text{by (\ref{Eq:SliceDiffeoRelationship})}\\
&=T\left(f(d)_{G\times^KV}\right) E^D_0(Y) (\iota(d)) \qquad \text{by equivariance of }E^D_0 \\
&=T\left(f(d)_{G\times^KV}\right) T\iota\, Y (d) \qquad \text{ by definition of }E^D_0 \\
&=T(f(d)_{G\times^KV}\circ \iota ) Y(d) \\
&=T(j\circ \phi) Y(d) \qquad \text{ by (\ref{Eq:SliceDiffeoRelationship})}\\
&=Tj T\phi Y \phi^{-1} \phi(d) \\
&=Tj (\phi_* Y) \phi(d) \\
&=E^V_0(\phi_*Y) \left(j(\phi(d))\right) \qquad \text{by definition of }E^V_0.
\end{align*}
By equivariance this means that:
\begin{equation*}
E^D_0(Y)=E^V_0(\phi_*Y).
\end{equation*}
Hence, in particular, we have that:
\begin{equation}\label{Eq:EquivariantExtensionsRelated}
E^D_0\left(P^D_0(X)\right) 
=
E^V_0\left(\phi_*P^D_0(X)\right).
\end{equation}

Finally, putting all this together we obtain that:
\begin{align}\label{Eq:ProjectionDifferenceReady}
\begin{split}
&E^V_0 \Big( P^V_0(X) - \phi_*P^D_0(X) \Big)\\
&=E^V_0\left(P^V_0(X)\right) - E^V_0\left(\phi_*P^D_0(X)\right) \qquad \text{by the linearity of } E^V_0\\
&=E^V_0\left(P^V_0(X)\right) - E^D_0\left(P^D_0(X)\right) \qquad \text{by (\ref{Eq:EquivariantExtensionsRelated})}\\
&=\partial\left( h^D(X) -  h^V(\phi_*X)\right) \qquad \text{by \ref{Eq:DifferenceProjections}}.
\end{split}
\end{align}
Thus, applying $P^V_0$ to the resulting equality we have that:
\begin{align*}
&P^V_0(X) - \phi_*P^D_0(X)\\
&=P^V_0\Big(E^V_0 \Big( P^V_0(X) - \phi_*P^D_0(X) \Big)\Big) \qquad \text{ since } P^V_0E^V_0=\text{id} \text{ by Theorem \ref{Thm:ExtendThenProject}}\\
&=P^V_0\left(
\partial\left( h^D(X) -  h^V(\phi_*X)\right)
\right) \qquad \text{by (\ref{Eq:ProjectionDifferenceReady}) }\\
&=\partial\Big( P^V_1\left(h^D(X)\right) - P^V_1\left(h^V(\phi_*X)\right)\Big) \qquad \text{by (\ref{Diag:BoundaryProjection})}\\
&=\partial\Big( P^V_1\left(h^D(X)\right)\Big) \qquad \text{ since }h^v(\phi_*X)\in\ker P^V_1\text{ as in (\ref{Eq:P1Kernel})}.
\end{align*}
Hence, the infinitesimal gauge transformation:
\begin{equation*}
P^V_1\left(h^D(X)\right):V\to\ffk
\end{equation*}
gives an isomorphism between the vector fields $P^V_0(X)$ and $\phi_*P^D_0(X)$ on $V$ as claimed.
\end{proof}

We conclude this subsection by providing examples of the decomposition of equivariant vector fields using Theorem \ref{Thm:EquivalenceVFs}.

\begin{example}
Consider $M:=\bbr^3$ and fix a vector $\overrightarrow{w}\in \bbr^3$. 
Consider the action of $\bbr\times\bbs^1$ on $M$ given by:
\begin{equation*}
(r,\theta)\cdot v := R_\theta(v) + r \overrightarrow{w}, \qquad (\theta, r, v)\in \bbs^1\times\bbr\times M,
\end{equation*}
where $R_\theta$ is the counter-clockwise rotation by $\theta$ with respect to the axis spanned by $\overrightarrow{w}$.
The isotropy $K$ of any point $m\in M$ on the axis of rotation is the circle $\{0\}\times\bbs^1\cong\bbs^1$, and the canonical slice representation for such an action is isomorphic to $\bbr^2\cong \bbc$ with the standard action of the circle by counterclockwise rotations around the origin.
Thus, the associated bundle in this case is:
\begin{equation*}
\left(\bbr \times \bbs^1\right) \times^{\bbs^1}\bbc \cong \bbr\times \bbc.
\end{equation*}
This associated bundle is equivariantly diffeomorphic to the axis of rotation and the plane orthogonal to it passing through the given point, so we have:
\begin{equation*}
M \cong \bbr \times \bbc.
\end{equation*}
We want to describe the decomposition of equivariant vector fields on $M\cong \bbr\times\bbc$ following from Theorem \ref{Thm:EquivalenceVFs}.

Let $E:\bbx(\bbc)^{\bbs^1}\to \bbx(\bbr\times\bbc)^{\bbr\times\bbs^1}$ be the equivariant inclusion functor of Theorem \ref{Thm:InclusionFunctor}, and let $P:\bbx(\bbr\times\bbc)^{\bbr\times\bbs^1}\to \bbx(\bbc)^{\bbs^1}$ be the equivariant projection, as in Theorem \ref{Thm:ProjectionFunctor}, corresponding to the canonical $\bbs^1$-equivariant splitting $\bbr\oplus\bbr$ of the Lie algebra of $\bbr\times\bbs^1$.
Furthermore, let $h:\ffX(\bbr\times\bbc)^{\bbr\times\bbs^1}\to C^\infty(\bbr\times\bbc,\bbr\oplus\bbr)^{\bbr\times\bbs^1}$ be the corresponding natural isomorphism $h:EP\cong 1_{\bbx(\bbr\times\bbc)^{\bbr\times\bbs^1}}$ of Theorem \ref{Thm:NaturalIsomorphism}.
Then, by Theorem \ref{Thm:EquivalenceVFs}, for any equivariant vector field $X:\bbr\times\bbc\to\bbr\times\bbc$, we may write:
\begin{equation}\label{Eq:RS1StdVF}
X=E_0(P_0(X))+\partial(h(X)).
\end{equation}
We now use the decomposition (\ref{Eq:RS1StdVF}) to give a standard form for ($\bbr\times\bbs^1$)-equivariant vector fields on $M$, similar to the examples at the end of subsection \ref{categoryVFs}.

We start by describing the infinitesimal gauge transformations on $\bbr\times\bbc$.
Since $\bbr\times\bbs^1$ is abelian, the Adjoint representation of $\bbr\times\bbs^1$ on the Lie algebra is trivial.Hence, so the infinitesimal gauge transformations are the $(\bbr\times\bbs^1)$-invariant functions.
Let $\psi=(\psi_1,\psi_2):\bbr\times\bbc\to\bbr\oplus\bbr$ be an arbitrary infinitesimal gauge transformation on $\bbr\times\bbc$.
By $\bbr$-invariance, the first component $\psi_1:\bbr\times\bbc\to\bbr$ is completely determined by its restriction to $\{0\}\times\bbc$.
Hence, the restriction $\psi_1|:\bbc\to\bbr$ is an $\bbs^1$-invariant function.
Thus, as in Example \ref{Ex:Circle}, by Schwarz's Theorem \cite{Sch75}, we may write:
\begin{equation*}
\psi_1|(z)=\widehat \psi_1 \left(|z|^2\right),
\qquad z\in\bbc,
\end{equation*}
where $\widehat\psi_1:\bbr\to\bbr$ is a smooth function.
Exactly the same way, we may write the restrictioon to $\{0\}\times\bbc\cong\bbc$ of the second component $\psi_2:\bbr\times\bbc\to\bbr$ as:
\begin{equation*}
\psi_2|(z)=\widehat \psi_2 \left(|z|^2\right),
\qquad z\in\bbc,
\end{equation*}
where $\widehat\psi_2:\bbr\to\bbr$ is a smooth function.
Hence, by invariance, we have that:
\begin{equation}\label{Eq:R3InfGauge}
\psi(w,z)=
\left(
\widehat \psi_1 \left(|z|^2\right),\,
\widehat\psi_2\left(|z|^2\right)
\right)
\qquad (w,z)\in\bbr\times\bbc.
\end{equation}
The induced vector field $\partial(\psi)$ as in (\ref{Eq:InducedVF}) is given by:
\begin{equation}\label{Eq:R3InducedVF}
\partial(\psi)(w,z)=
\left(\begin{array}{c}
\widehat\psi_1\left(|z|^2\right) \\ 
\widehat\psi_2\left(|z|^2\right) iz\\
\end{array}\right),
\qquad (z,w)\in\bbc\times\bbr.
\end{equation}

Since $P_0(X)$ is an $\bbs^1$-equivariant vector field on $\bbc$, as described in Example \ref{Ex:Circle}, by Scharwz's Theorem \cite{Sch75} and Po\'enaru's Theorem \cite{Poe76}, we may write:
\begin{equation*}
P_0(X)(z)=f\left(|z|^2\right) z + g\left(|z|^2\right) i z,
\qquad z\in\bbc,
\end{equation*}
where $f:\bbr\to\bbr$ and $g:\bbr\to\bbr$ are smooth functions.
Thus, we have that:
\begin{equation}\label{Eq:RS1StdVFFirstPart}
E_0(P_0(X))(w,z)
=
\left(\begin{array}{c}
0 \\
f\left(|z|^2\right) z + g\left(|z|^2\right) iz
\end{array}\right),
\qquad (w,z)\in\bbr\times\bbc.
\end{equation}
On the other hand, by (\ref{Eq:R3InfGauge}), the map $h(X):\bbr\times\bbc \to \bbr\oplus\bbr$ is determined by its restriction to the slice $\{0\}\times\bbc\cong \bbc$.
By Theorem \ref{Thm:NaturalIsomorphism}, the restriction of $h(X)$ to $\{0\}\times\bbc\cong \bbc$ takes values in $\bbr\times\{0\}$.
Consequently, the infinitesimal gauge transformation $h(X):\bbr\times\bbc\to\bbr\oplus\bbr$ is given by:
\begin{equation*}
h(X)(w,z)=\left(\widehat h \left(|z|^2\right),\, 0\right),
\qquad (w,z)\in\bbr\times\bbc,
\end{equation*}
for some smooth function $\widehat h:\bbr\to\bbr$.
And hence, by (\ref{Eq:R3InducedVF}), the equivariant vector field $\partial(h(X))$ on $\bbr\times \bbc$ is given by:
\begin{equation}\label{Eq:RS1StdVFSecondPart}
\partial(h(X))(w,z)
=
\left(\begin{array}{c}
\widehat h\left(|z|^2\right) \\
0
\end{array}\right),
\qquad (w,z)\in\bbr\times\bbc.
\end{equation}
Using (\ref{Eq:RS1StdVFFirstPart}) and (\ref{Eq:RS1StdVFSecondPart}), the decomposition (\ref{Eq:RS1StdVF}) gives the desired standard form for the given equivariant vector field $X$ on $\bbr\times\bbc$:
\begin{equation}\label{Eq:R3Decomposition}
X(w,z)=
\underbrace{
\left(\begin{array}{c}
0 \\
f\left(|z|^2\right) z + g\left(|z|^2\right) iz
\end{array}\right)
}_{E_0(P_0(X))}
+
\underbrace{
\left(\begin{array}{c}
\widehat h\left(|z|^2\right) \\
0
\end{array}\right)
}_{\partial(h(X))}
,
\end{equation}
for $(z,w)\in\bbc\times\bbr$, where $f:\bbr\to\bbr$, $g:\bbr\to\bbr$, and $\widehat h:\bbr\to\bbr$ are smooth functions.
Using the equivariant diffeomorphism $M\cong \bbr\times\bbc$, the decomposition (\ref{Eq:R3Decomposition}) of equivariant vector fields on $M$ is a decomposition of equivariant vector fields on $M$ into a component that's tangent to the plane orthogonal to the axis of rotation and a component in the direction of the axis of rotation.
In particular, every $\bbr\times\bbs^1$-equivariant vector field on $M$ is isomorphic in the category $\bbx(\bbr\times\bbc)^{\bbr\times\bbs^1}$ to a vector field tangent to the planes orthogonal to the axis of rotation.
However, we can say more.
We can rearrange the decomposition (\ref{Eq:R3Decomposition}) as follows:
\begin{equation}\label{Eq:R3Decomposition2}
X(w,z)=
\left(\begin{array}{c}
0 \\
f\left(|z|^2\right) z
\end{array}\right)
+
\underbrace{
\left(\begin{array}{c}
\widehat h\left(|z|^2\right) \\
g\left(|z|^2\right) iz
\end{array}\right)
}_{\partial(\psi)}
,
\end{equation}
for $(w,z)\in\bbr\times\bbc$, where we are using (\ref{Eq:R3InducedVF}) for the infinitesimal gauge transformation $\psi:\bbr\times\bbc\to\bbr\oplus\bbr$ defined by:
\begin{equation*}
\psi(w,z):=\Big(  h\left(|z|^2\right) ,\,  g\left(|z|^2\right)\Big),
\qquad (w,z)\in\bbr\times\bbc.
\end{equation*} 
Hence, the equivariant vector field $X$ is also isomorphic in the category $\bbx(\bbr\times\bbc)^{\bbr\times\bbs^1}$ to the equivariant vector field defined by:
\begin{equation*}
\widehat X (w,z):=
\left(\begin{array}{c}
0 \\
f\left(|z|^2\right) z
\end{array}\right),
\qquad (w,z)\in\bbr\times\bbc.
\end{equation*}
Thus, every $(\bbr\times\bbs^1)$-equivariant vector field on $\bbr\times\bbc$ is isomorphic in the category $\bbx(\bbr\times\bbc)^{\bbr\times\bbs^1}$ to a radial vector field on a plane orthogonal to the axis of rotation.
\end{example}

\begin{example}\label{Ex:SO3}
Consider the group $G:=SO(3)$ of $3\times 3$ orthogonal matrices with determinant equal to $1$.
Let $K:=\bbs^1\cong SO(2)$ be the circle in $SO(3)$ consisting of the matrices:
\begin{equation*}
k_\theta:=\left(\begin{array}{ccc}
\cos\theta & -\sin\theta & 0 \\
\sin\theta & \cos\theta & 0 \\
0 & 0 & 1
\end{array}\right),
\qquad \theta\in\bbr.
\end{equation*}
Let $V:=\bbc$ be the standard representation of the circle $K$ given by:
\begin{equation*}
k_\theta\cdot z := e^{i\theta}z, \qquad k_\theta \in K, \, z\in\bbc.
\end{equation*}
We will decompose the $SO(3)$-equivariant vector fields on the associated bundle $SO(3)\times^{\bbs^1}\bbc$ using Theorem \ref{Thm:EquivalenceVFs}.
Thus, we are working in the category $\bbx(SO(3)\times^{\bbs^1}\bbc)^{SO(3)}$ of $SO(3)$-equivariant vector fields on the associated bundle $SO(3)\times^{\bbs^1}\bbc$.

Recall that the Lie algebra $\ffg:=\ffs\ffo(3)$ of $SO(3)$ can be identified with $\bbr^3$ with the Adjoint representation of $SO(3)$ corresponding to the standard application of the matrix $g\in SO(3)$ to the vector $w\in \bbr^3$.
Furthermore, the vectors $w\in\ffs\ffo(3)$ can be thought of as corresponding to the axes of rotation of the matrices $g\in SO(3)$, when the latter are thought of as rotation matrices.
In particular, the Lie algebra $\ffk\cong\bbr$ of $K\cong\bbs^1$ can be identified with the third copy of $\bbr$ in $\bbr^3\cong\ffs\ffo(3)$.
Thus, the canonical splitting of $\bbr^3$:
\begin{equation}\label{Eq:SO3Splitting}
\underbrace{\bbr^3}_{\ffg} = \underbrace{\bbr^2}_{\ffq}\oplus\underbrace{\bbr}_{\ffk}.
\end{equation}
is a $K$-equivariant splitting of the Lie algebra $\ffg\cong\ffs\ffo(3)$.
We will choose this splitting as the $K$-equivariant splitting of $\ffg$ needed to apply Theorem \ref{Thm:EquivalenceVFs}.
In particular, it determines an equivariant connection on the associated bundle $SO(3)\times^{\bbs^1}\bbc\to SO(3)/\bbs^1$ (see Remark \ref{Rem:ConnectionFromSplitting} and (\ref{Eq:SO3Horizontal}) below).

It is convenient to recall the description given in Remark \ref{Rem:ConnectionFromSplitting} of the tangent bundle $T(SO(3)\times^{\bbs^1}\bbc)$, and of its decomposition into the vertical bundle $\calv(SO(3)\times^{\bbs^1}\bbc)$ and the horizontal bundle $\calh(SO(3)\times^{\bbs^1}\bbc)$ of the associated bundle $SO(3)\times^{\bbs^1}\bbc \to SO(3)/\bbs^1$, the horizontal bundle being the one corresponding to the splitting (\ref{Eq:SO3Splitting}).
We start by recalling that the total space of the tangent bundle $T(SO(3)\times^{\bbs^1}\bbc)\to SO(3)\times^{\bbs^1}\bbc$ can be identified with:
\begin{align}\label{Eq:SO3Tangent}
\begin{split}
T(SO(3)\times^{\bbs^1}\bbc)
&\cong G\times^K(V\times\ffq\times V) \\
&\cong SO(3)\times^{\bbs^1}(\bbc \times \bbr^2 \times \bbc),
\end{split}
\end{align}
where for an element $[g,z,\xi,w]\in SO(3)\times^{\bbs^1}(\bbc \times \bbr^2 \times \bbc)$ the point $[g,z]\in SO(3)\times^{\bbs^1}\bbc$ is the base point and the rest is the vector part.
The total space of the vertical bundle $\calv(SO(3)\times^{\bbs^1}\bbc)\to SO(3)\times^{\bbs^1}\bbc$ of the associated bundle $SO(3)\times^{\bbs^1}\bbc\to SO(3)/\bbs^1$ can be identified with:
\begin{align}\label{Eq:SO3Vertical}
\begin{split}
\calv(SO(3)\times^{\bbs^1}\bbc)
&\cong G\times^K(V\times \{0\}\times V)\\
%&\cong G\times^K(V\times V)\\
&\cong SO(3)\times^{\bbs^1}(\bbc\times \{0\}\times \bbc),
\end{split}
\end{align}
where $\{0\}$ is the $0$-subspace in $\bbr^2\cong \ffq$, and for an element $[g,z,0,w]\in SO(3)\times^{\bbs^1}(\bbc\times\{0\}\times \bbc)$ the point $[g,z]\in SO(3)\times^{\bbs^1}\bbc$ is the base point and the rest is the vector part.
Finally, recall that the splitting (\ref{Eq:SO3Splitting}) determines a connection on the associated bundle $SO(3)\times^{\bbs^1}\bbc \to SO(3)/\bbs^1$ (Remark \ref{Rem:ConnectionFromSplitting}).
The total space of the corresponding horizontal bundle $\calh(SO(3)\times^{\bbs^1}\bbc)\to SO(3)\times^{\bbs^1}\bbc$ of the associated bundle $SO(3)\times^{\bbs^1}\bbc\to SO(3)/\bbs^1$ can be identified with:
\begin{align}\label{Eq:SO3Horizontal}
\begin{split}
\calh(SO(3)\times^{\bbs^1}\bbc)
&\cong G\times^K(V\times\ffq\times\{0\})\\
%&\cong G\times^K(V\times\ffq)\\
&\cong SO(3) \times^{\bbs^1} (\bbc\times\bbr^2\times\{0\}),
\end{split}
\end{align}
where $\{0\}$ is the $0$-subspace in $\bbc$, and for an element $[g,z,\xi,0]\in SO(3) \times^{\bbs^1} (\bbc\oplus\bbr^2\times\{0\})$ the point $[g,z]\in SO(3)\times^{\bbs^1}\bbc$ is the base point and the rest is the vector part.
Thus, we use (\ref{Eq:SO3Tangent}), (\ref{Eq:SO3Vertical}), and (\ref{Eq:SO3Horizontal}) to describe vector fields as sections of each of these bundles over $SO(3)\times^{\bbs^1}\bbc$.

Now let $E:\bbx(\bbc)^{\bbs^1}\to \bbx(SO(3)\times^{\bbs^1}\bbc)^{SO(3)}$ be the canonical inclusion functor of Theorem \ref{Thm:InclusionFunctor}, let $P:\bbx(SO(3)\times^{\bbs^1}\bbc)^{SO(3)}\to \bbx(\bbc)^{\bbs^1}$ be the equivariant projection functor corresponding to the splitting (\ref{Eq:SO3Splitting}) as in Theorem \ref{Thm:ProjectionFunctor}, and let $h:\ffX(SO(3)\times^{\bbs^1}\bbc)^{SO(3)}\to C^\infty(SO(3)\times^{\bbs^1}\bbc)^{SO(3)}$ be the corresponding natural isomorphism of Theorem \ref{Thm:NaturalIsomorphism}.
Then, by Theorem \ref{Thm:EquivalenceVFs}, we can decompose any equivariant vector field:
\begin{equation*}
X:SO(3)\times^{\bbs^1}\bbc \to SO(3)\times^{\bbs^1}(\bbc \times \bbr^2 \times \bbc),
\end{equation*}
in the form:
\begin{equation}\label{Eq:SO3Decomposition}
X = E_0(P_0(X)) + \partial(h(X)).
\end{equation}
We now describe the components in (\ref{Eq:SO3Decomposition}).

The vector field $P_0(X)$ is an $\bbs^1$-equivariant vector field on $\bbc$.
Thus, by Example \ref{Ex:Circle}, there exist smooth functions $f:\bbr\to\bbr$ and $q:\bbr\to\bbr$ such that:
\begin{equation}\label{Eq:SO3P0Part}
P_0(X)(z)=f\left(|z|^2\right)z + q\left(|z|^2\right)iz, \qquad z\in \bbc.
\end{equation}
Recall that, by definition, the vector field $E_0(P_0(X))$ is vertical in the associated bundle $SO(3)\times^{\bbs^1}\bbc\to SO(3)/\bbs^1$.
Thus, by (\ref{Eq:SO3Vertical}), the vector field $E_0(P_0(X))$ takes values in $SO(3)\times^{\bbs^1}(\bbc\times\{0\}\times \bbc)$.
Therefore, using (\ref{Eq:SO3P0Part}), we have that:
\begin{align}\label{Eq:SO3FirstPart}
\begin{split}
E_0(P_0(X))([g,z])
=\Big[g,z,0,f\left(|z|^2\right)z + q\left(|z|^2\right)iz \Big] 
%\in SO(3)\times^{\bbs^1}(\bbc\times\{0\}\times \bbc)
\end{split}
\end{align}
for all $[g,z]\in SO(3)\times^{\bbs^1}\bbc$.

Now observe that $h(X):SO(3)\times^{\bbs^1}\bbc \to \bbr^3$ is an $SO(3)$-equivariant infinitesimal gauge transformation that is completely determined by its values on the slice representation $\bbc \cong \{[I,z]\mid z\in\bbc\}$ and takes values in $\ffq\cong\bbr^2$ (Theorem \ref{Thm:NaturalIsomorphism}).
That is, it is completely determined by the $\bbs^1$-invariant restriction $h(X)|:\bbc\to \bbr^2$.
Thus, using Schwarz's Theorem \cite{Sch75} as we did for the infinitesimal gauge transformations in Example \ref{Ex:T2C2}, there exist smooth maps $\widehat h_1:\bbc\to \bbr$ and $\widehat h_2:\bbc\to\bbr$ such that:
\begin{equation*}
h(X)|(z)=\left(\widehat h_1\left(|z|^2\right),\,\widehat h_2\left(|z|^2\right) \right)\in\bbr^2\cong\ffq,
\qquad z\in \bbc.
\end{equation*}
Hence, by $SO(3)$-equivariance, we have that:
\begin{equation*}
h(X)([g,z])=g\cdot \left(\widehat h_1\left(|z|^2\right),\, \widehat h_2\left(|z|^2\right),\, 0\right) \in\bbr^3\cong\ffg,
\end{equation*}
for all $[g,z]\in SO(3)\times^{\bbs^1}\bbc$ and where the action on the right-hand side is the standard action of $SO(3)$ on $\bbr^3$ via rotations.
By the construction of $h$, the $SO(3)$-equivariant vector field $\partial(h(X))$ takes values in the horizontal bundle corresponding to the chosen splitting of $\ffg$ since it gives the horizontal part of the vector field $X$ (Theorem \ref{Thm:NaturalIsomorphism}).
Thus, using (\ref{Eq:SO3Horizontal}), the vector field $\partial(h(X))$ takes values in $SO(3)\times^{\bbs^1}(\bbc\times\bbr^2\times\{0\})$, so we have that:
\begin{align}\label{Eq:SO3SecondPart}
\begin{split}
\partial(h(X))\left([g,z]\right)
=\Big[g,z,\left(\widehat h_1\left(|z|^2\right),\,\widehat h_2\left(|z|^2\right) \right),0\Big]
%\in SO(3)\times^{\bbs^1}(\bbc\times\bbr^2),
\end{split}
\end{align}
for all $[g,z]\in SO(3)\times^{\bbs^1}\bbc$.
Therefore, using (\ref{Eq:SO3FirstPart}) and (\ref{Eq:SO3SecondPart}), we see that the decomposition (\ref{Eq:SO3Decomposition}) is:
\begin{align}\label{Eq:SO3DecompositionDeux}
\begin{split}
X([g,z])&= \underbrace{\Big[g,z,0,f\left(|z|^2\right)z + q\left(|z|^2\right)iz \Big] 
}_{E_0(P_0(X))
% \in SO(3)\times^{\bbs^1}\left(\bbc\times\left\{\overrightarrow{0}\right\}\times\bbc\right)
}\\
&\qquad\qquad\qquad + \underbrace{\Big[g,z,\left(\widehat h_1\left(|z|^2\right),\,\widehat h_2\left(|z|^2\right) \right),0\Big]}_{
\partial(h(X))
%\in SO(3)\times^{\bbs^1}(\bbc\times\bbr^2\times \{0\})
}
\end{split}
\end{align}
for all $[g,z]\in SO(3)\times^{\bbs^1}\bbc$, where the vector $X([g,z])$ is contained in the associated bundle $SO(3)\times^{\bbs^1}(\bbc\times\bbr^2\times\bbc)$, and we have used the identifications (\ref{Eq:SO3Tangent}), (\ref{Eq:SO3Vertical}), and (\ref{Eq:SO3Horizontal}).
In particular, every $SO(3)$-equivariant vector field on $SO(3)\times^{\bbs^1} \bbc$ is isomorphic in the category $\bbx(SO(3)\times^{\bbs^1}\bbc)^{SO(3)}$ to a vector field that is tangent to the slice:
\begin{equation}\label{Eq:SO3Slice}
\bbc\cong \left\{ \left[I,z\right]\in SO(3)\times^{\bbs^1}\bbc \mid z\in \bbc \right\},
\end{equation}
and that has a decomposition on the slice $\bbc$ as an $\bbs^1$-equivariant vector field as in Example \ref{Ex:Circle}.
We can say more.
Let $\psi:SO(3)\times^{\bbs^1}\bbc\to \bbr^3$ be the infinitesimal gauge transformation given by:
\begin{equation*}
\psi\left([g,z]\right):=g\cdot \left(\widehat h_1\left(|z|^2\right),\,\widehat h_2\left(|z|^2\right),\, q\left(|z|^2\right)\right),
\end{equation*}
for $[g,z]\in SO(3)\times^{\bbs^1}\bbc$, and where the action on the right-hand side is the standard action of $SO(3)$ on $\bbr^3$ via rotations.
Rewriting (\ref{Eq:SO3DecompositionDeux}) we see that we can write $X$ as:
\begin{align*}
X([g,z])
=\Big[g,z,0,f\left(|z|^2\right)z \Big]
+\underbrace{\Big[g,z,\left(\widehat h_1\left(|z|^2\right),\,\widehat h_2\left(|z|^2\right) \right),q\left(|z|^2\right)iz\Big]}_{\partial(\psi)}
\end{align*}
for all $[g,z]\in SO(3)\times^{\bbs^1}\bbc$.
Thus, every $SO(3)$-equivariant vector field $X$ on $SO(3)\times^{\bbs^1}\bbc$ is isomorphic to a vector field tangent to the slice (\ref{Eq:SO3Slice}) that is radial on the slice.
\end{example}

\subsection{Decomposition preserves relative equilibria}
In the previous subsection we proved a decomposition of equivariant vector fields near relative equilibria in the form of an equivalence of categories $\bbx(G\times^KV)^G\simeq\bbx(V)^K$ (Theorem \ref{Thm:ProjectionFunctor} and Theorem \ref{Thm:EquivalenceVFs}).
This equivalence is given by a functor $P:\bbx(G\times^KV)^G\to\bbx(V)^K$ that projects equivariant vector fields onto the slice representation $V$.
The other functor in the equivalence is given by the canonical equivariant extension functor $E:\bbx(V)^K\to\bbx(G\times^KV)^G$ (Theorem \ref{Thm:InclusionFunctor}).
It is natural, and necessary for the rest of the work in this paper, to show that both of these functors preserve relative equilibria.
We do this in this brief subsection.

For this we need the following lemma:

\begin{lemma}\label{Lemma:PBRelEq}
Let $M$ and $N$ be proper $G$-manifolds and let $f\colon M\to N$ be a $G$-equivariant diffeomorphism.
Suppose that $X$ and $Y$ are $f$-related equivariant vector fields on $M$ and $N$ respectively.
Then a point $m\in M$ is a relative equilibrium of the vector field $X$ if and only if the point $f(m)$ is a relative equilibrium of the vector field $Y$.
Thus, pullbacks and pushforwards of vector fields by equivariant diffeomorphisms preserve relative equilibria.
\end{lemma}

\begin{proof}
The verification is a straightforward computation using the equation $T f\circ X = Y\circ f$.
First, suppose $m$ is a $G$-relative equilibrium of the vector field $X$. 
Then
\begin{equation*}
Y(f(m)) =(T f)_m(X(m))\in (T f)_m(T_m(G\cdot m))=T_{f(m)}(G\cdot f(m)),
\end{equation*}
where $(T f)_m(T_m(G\cdot m))=T_{f(m)}(G\cdot f(m))$ follows by the equivariance of the diffeomorphism~$f$. 
Thus, the point $f(m)$ is a $G$-relative equilibrium of the vector field $Y$. 
The converse is completely analogous.
\end{proof}

\begin{lemma}\label{Lemma:EPreservesRelEq}
Let $V$ be a finite-dimensional real representation of a compact Lie subgroup $K$ of a Lie group $G$, let $X$ be a $K$-equivariant vector field on $V$, and let $E:\bbx(V)^K\to\bbx(G\times^KV)^G$ be the canonical equivariant extension functor of Theorem \ref{Thm:InclusionFunctor}.
If $v\in V$ is a $K$-relative equilibrium of $X$ then $[1,v]\in G\times^KV$ is a $G$-relative equilibrium of $E_0(X)\in \ffX(G\times^KV)^G$.
\end{lemma}

\begin{proof}
Let $j:V\hookrightarrow G\times^KV$ be the equivariant embedding defined by $j(v):=[1,v]$.
Note that the vector fields $X$ and $E_0(X)$ are $j$-related by the definition of $E_0$ (Theorem \ref{Thm:InclusionFunctor}).
Thus, by Lemma \ref{Lemma:PBRelEq}, we know that $[1,v]$ is a $K$-relative equilibrium of the vector field $E_0(X)$.
That is, $E_0(X)([1,v])\in T_{[1,v]}(K\cdot [1,v])$. 
Since $j$ is an embedding, the tangent space $T_{[1,v]}(K\cdot [1,v])$ is contained in the tangent space $T_{[1,v]}(G\cdot [1,v])$. 
Hence, the point $[1,v]$ is a $G$-relative equilibrium of $E_0(X)$ as claimed.
\end{proof}

\begin{lemma}\label{Lemma:PPreservesRelEq}
Let $V$ be a finite-dimensional real representation of a compact Lie subgroup $K$ of a Lie group $G$, let $X$ be a $G$-equivariant vector field on $G\times^KV$, and let $P:\bbx(G\times^KV)^G\to\bbx(V)^K$ be an equivariant projection functor corresponding to a $K$-equivariant splitting as in Theorem \ref{Thm:ProjectionFunctor}.
If $[g,v]\in G\times^KV$ is a $G$-relative equilibrium of $X$ then $v\in V$ is a $K$-relative equilibrium of $P_0(X)$.
\end{lemma}

\begin{proof}
First, as in the prooof of Lemma \ref{Lemma:EPreservesRelEq}, let $j:V\hookrightarrow G\times^KV$ be the equivariant embedding defined by $j(v):=[1,v]$.
Furthermore, note that if $[g,v]\in G\times^KV$ is a relative equilibrium of $X$, then $j(v)=[1,v]$ is also a $G$-relative of $X$ since these two points are in the same $G$-orbit in $G\times^KV$.
By Theorem \ref{Thm:EquivalenceVFs}, there exists an infinitesimal gauge transformation $\psi^X\in C^\infty(G\times^KV,\ffg)^G$ such that:
\begin{equation*}
X=E_0(P_0(X)) + \partial(\psi^X).
\end{equation*}
That is, $X$ is isomorphic to the vector field $E_0(P_0(X))$.
Since isomorphisms preserve relative equilibria by Lemma \ref{Lemma:IsoShareRel}, the point $j(v)$ is a $G$-relative equilibrium of the vector field $E_0(P_0(X))$.
On the other hand, the vector field $E_0(P_0(X))$ is vertical in the vector bundle $G\times^KV\to G/K$.
Hence, the vector $E_0(P_0(X))(j(v))$ is also tangent to the slice $j(V)\cong V$.
That is, we now have:
\begin{equation*}
E_0(P_0(X))(j(v))\in T_{j(v)}\left(G\cdot j(V)\right) \cap T_{j(v)}j(V) = T_{j(v)} (K\cdot j(v)),
\end{equation*}
meaning that the point $j(v)$ is a $K$-relative equilibrium of the vector field $E_0(P_0(X))$.
By Lemma \ref{Lemma:PBRelEq}, this implies that $v\in V$ is a $K$-relative equilibrium of $P_0(X)$ since the vector fields $E_0(P_0(X))$ and $P_0(X)$ are $j$-related by the definition of $E_0$ (Theorem \ref{Thm:InclusionFunctor}).
\end{proof}

Thus, we have shown that the functors in the equivalence $\bbx(G\times^KV)^G\simeq \bbx(V)^K$ preserve relative equilibria.

%%%%%%%%%%%%%%%%%%%%%%%%%%%%%%%%
\section{Motion of relative equilibria and isomorphisms}\label{motionsection}
In this section, we consider the motion of relative equilibria on manifolds with compact symmetry groups from the perspective of the category of equivariant vector fields.
Using these tools, we provide some brief observations about such motion.
The results in this section will be useful to characterize the behavior of bifurcating solutions in section \ref{ch2}, but they may also be of interest in the control of equivariant dynamical systems.

It is well-known that the motion of a relative equilibrium of an equivariant vector field on such a manifold is equivalent to linear motion on a torus (see \cite{F80,K90} and also Theorem \ref{Thm:FieldKrupaTorus}).
In fact, there is a bound on the number of independent frequencies of the motion; that is, on the dimension of the torus containing the motion \cite{F80,K90}.
This bound is attained generically, but one can seek to modify the equivariant vector field to reduce, or otherwise adjust, the number of independent frequencies of the relative equilibrium's motion to obtain nongeneric motions.

Given an equivariant vector field with a relative equilibrium on a $K$-manifold, with compact symmetry group $K$, we describe conditions for constructing an isomorphic vector field that has any desired number of independent frequencies at the relative equilibrium (Proposition \ref{Prop:FrequencyStabilizationConditions}).
Since the resulting vector field is isomorphic to the given one, it determines the same flow on the orbit space, and hence the same dynamics modulo the symmetries (Theorem \ref{Lemma:LermanFlows}).
In particular, we show that this is always possible for actions of tori (Theorem \ref{Thm:jStabilization}).

\begin{remark}
By a torus we mean a compact, connected, and abelian Lie group.
Recall that any compact Lie group $K$ has a maximal torus Lie subgroup, and that all maximal tori in $K$ have the same dimension.
The rank of $K$ is then the dimension of any maximal torus in $K$.
For convenience in dealing with cases where relative equiliobria are strict equilibria, we will consider the trivial group to be a torus of dimension $0$.
\end{remark}

We begin by recalling the following well-known result asserting that the motion of a relative equilibrium for a compact group action is contained in a torus:

\begin{theorem}[Field \cite{F80} and Krupa \cite{K90}]\label{Thm:FieldKrupaTorus}
Let $M$ be a $K$-manifold with $K$ compact, let $X$ be an equivariant vector field on $M$ with a relative equilibrium at a point $m\in M$ with isotropy $K_m$.
Then the integral curve $\phi_m(-)$ of $X$ starting at $m$ is equivalent to linear motion on a torus.
That is, the closure of the image of the curve:
\begin{equation*}
T_m:=\overline{\left\{
\phi_m(\tau)
\right\}}
\end{equation*}
is diffeomorphic to a torus.
Furthermore, the dimension of $T_m$ is bounded by the rank of the Lie group $N(K_m)/K_m$, where $N(K_m)$ is the normalizer of the stabilizer $K_m$.
\end{theorem}

\begin{proof}
Let $\xi\ffk$ be a velocity of the relative equilibrium, so that in particular $\phi_m(\tau)=\exp(\tau\xi)\cdot m$.
The isotropy is constant along the integral curve $\phi_m$.
That is, the stabilizer $K_{\phi_m(\tau)}$ is equal to $K_m$ for all times $\tau\in\bbr$.
Thus, for all $k\in K_m$ and all $\tau\in\bbr$, we have that:
\begin{equation*}
\exp(\tau\xi)^{-1}k \exp(\tau\xi)\in K_m.
\end{equation*}
That means that $\exp(\tau\xi)\in N(K_m)$ for all times $\tau\in\bbr$, and the vector $\xi$ is in the Lie algebra $\ffn$ of the normalizer $N(K_m)$.
Consider the linear motion on the quotient group $N(K_m)/K_m$ induced by the vector $\xi\in\ffn$:
\begin{equation*}
\gamma_{\xi+\ffk_m}:\bbr\to N(K_m)/K_m, \qquad \gamma_{\xi+\ffk_m}(\tau):=\exp(\tau\xi)K_m\equiv\exp(\tau (\xi+\ffk_m))K_m.
\end{equation*}
The set:
\begin{equation*}
T:=\overline{\left\{
\exp(\tau\xi)K_m
\right\}}
\end{equation*}
is a $1$-parameter, connected, abelian, closed, and compact Lie subgroup of $N(K_m)/K_m$.
Hence, $T$ is isomorphic to a torus.
Now recall that the group orbit $K\cdot m$ is equivariantly diffeomorphic to the homogenous space $K/K_m$ via the map:
\begin{equation*}
K\cdot m \xrightarrow{\cong} K/K_m, \qquad k\cdot m \mapsto k K_m.
\end{equation*}
Under this diffeomorphism the curve $\gamma_{\xi+\ffk_m}$ corresponds to the integral curve $\phi_m$, and the set $T_m$ corresponds to $T$.
That is, the integral curve $\phi_m$ is equivalent to the linear motion $\gamma_{\xi+\ffk_m}$ on the torus $T$ in $N(K_m)/K_m$.
The dimension of $T$ is clearly bounded above by the dimension of the maximal torus in $N(K_m)/K_m$.
Hence, so is the dimension of $T_m$.
\end{proof}

The bound on the number of independent frequencies in Theorem \ref{Thm:FieldKrupaTorus} is an equality for generic relative equilibria of equivariant vector fields.
That's because generically the closure of $1$-parameter subgroups in compact Lie groups are maximal tori, and the torus containing the motion of the relative equilibrium is equivalent to a $1$-parameter subgroup in the Lie group $N(K_m)/K_m$.
Thus, if the Lie group $N(K_m)/K_m$ has a maximal torus of dimension greater than $1$, then relative equilibria generically exhibit quasi-periodic motion.
This may be desirable in some cases, but not in others.

Thus, we consider how one can modify the number of independent frequencies of a relative equilibrium.
For example, this may be of interest in applications where one seeks to constrain the frequencies of the motion of the relative equilibrium.
It is natural to look for modifications that don't change the underlying dynamics modulo the symmetries.
That is, we look for an equivariant vector field that determines the same dynamics on the orbit space as the given one, but has a relative equilibrium with a different number of independent frequencies.
We introduce the following definition to make this idea precise:

\begin{definition}\label{Def:FrequencyStabilization}
Let $M$ be a $K$-manifold where $K$ is a compact Lie group, and let $X$ be an equivariant vector field on $M$ with a relative equilibrium $m$.
Suppose that the torus:
\begin{equation*}
T_m:=\overline{\{\phi_m(\tau)\}}
\end{equation*}
has dimension $d>0$, where $\phi_m$ is the integral curve of $X$ starting at $m$.
Then, for $0\le j \le d$, an infinitesimal gauge transformation $\psi\in C^\infty(M,\ffk)^K$ {\it stabilizes the frequencies of $X$ at $m$ to order $j$} if the torus:
\begin{equation*}
T_m^\psi:=\overline{\{\varphi_m(\tau) \}}
\end{equation*}
has dimension $j$, where $\varphi_m$ is the integral curve of $X+\partial(\psi)$ starting at $m$.
\end{definition}

It is immediate from Theorem \ref{Thm:Decomposition} that one can always get rid of all the independent frequencies and stabilize to order $0$:

\begin{proposition}\label{Prop:0thOrderFrequencyStabilization}
Let $M$ be a $K$-manifold where $K$ is a compact Lie group and let $X$ be an equivariant vector field on $M$ with a relative equilibrium at a point $m\in M$.
Then there exists $\psi\in C^\infty(M,\ffk)^K$ that stabilizes the the frequencies of $X$ at $m$ to order $0$.
That is, the infinitesimal gauge transformation $\psi$ is such that the vector field $X+\partial(\psi)$ has an equilibrium at $m$.
\end{proposition}

\begin{proof}
By Theorem \ref{Thm:Decomposition} there exists a decomposition of $X$ of the form:
\begin{equation*}
X=Y^X+\partial(\psi^X),
\end{equation*}
where $Y^X$ is an equivariant vector field with an equilibrium at $m$ and $\psi^X\in C^\infty(M,\ffk)^K$.
Thus, the infinitesimal gauge transformation $\psi:=-\psi^X\in C^\infty(M,\ffk)$ is such that:
\begin{equation*}
X+\partial(\psi)=X-\partial(\psi^X)=Y^X,
\end{equation*}
so $X + \partial(\psi)$ has an equilibrium at $m$.
\end{proof}

Due to Proposition \ref{Prop:0thOrderFrequencyStabilization}, we are interested in the cases where $0<j \le d$.
For this it helps to describe the independent frequencies of a relative equilibrium using the velocity.
To do this we recall some basic facts about tori:

\begin{definition}\label{Def:LatticeBasis}
Let $T$ be an $n$-dimensional torus with Lie algebra $\fft$, and let $\exp:\fft\to T$ be the exponential map.
A {\it lattice basis} of $\fft$ is a set of $n$ $\bbr$-linearly independent vectors $\{t_1,\dots,t_d\}$ in $\fft$ such that:
\begin{equation*}
\ker(\exp)=\text{span}_{\bbz}\{t_1,\dots,t_d\}.
\end{equation*}
That is, every vector in the integral lattice $\bbz_\bbt:=\ker(\exp)$ can be written as an integer linear combination of the vectors $t_1,\dots,t_d$.
\end{definition}

\begin{remark}
The kernel of the exponential map of a torus is a discrete subgroup and thus the Lie algebra of a torus has a lattice basis as in Definition \ref{Def:LatticeBasis} (see, for example, the proof of Theorem 11.2 in \cite{H15}).
The integral lattice $\bbz_T$ of a torus $T$ is such that $T=\fft/\bbz_T$, and the lattice basis induces an isomorphism with the standard torus $\bbt^n:=\bbr^n/\bbz^n$ that satisfies the following diagram:
\begin{equation}\label{Diag:TorusIsomorphism}
\begin{gathered}
\xy
(-36,8)*+{\fft}="1";
(-12,8)*+{\bbr^n}="2";
(-36,-8)*{\,T}="3";
(-12,-8)*+{\bbt^n}="4";
{\ar@{->}^{\cong} "1";"2"};
{\ar@{->}_{\exp} "1";"3"};
{\ar@{->}^{\exp} "2";"4"};
{\ar@{->}_{\cong} "3";"4"};
\endxy
\end{gathered}
\end{equation}
where the top map sends a vector $\xi$ to its coordinate representation with respect to the lattice basis.
\end{remark}

Lattice bases can be used to determine when a vector in the Lie algebra of a torus induces a dense curve:

\begin{lemma}\label{Lemma:QuasiPeriodicInTorus}
Let $T$ be a $d$-dimensional torus, where $d>1$, and let $\fft$ be its Lie algebra.
Let $\xi\in \fft$ be a Lie algebra vector that such that $\xi_1,\dots,\xi_d$ are its coordinates with respect to any lattice basis of $\fft$.
Then the curve:
\begin{equation*}
\gamma_\xi:\bbr\to T, \qquad \gamma_\xi(\tau):=\exp(\tau\xi),
\end{equation*}
is dense in $T$ if and only if the numbers $\xi_1,\dots,\xi_d$ are $\bbq$-linearly independent.
\end{lemma}

\begin{proof}
It is well-known that a vector $(\xi_1,\dots,\xi_d)\in\bbr^d$ is such that the curve:
\begin{equation}\label{Eq:StandardTorusCurve}
\tau \mapsto \left(\xi_1\tau,\dots,\xi_d\tau\right) +\bbz^d, \qquad \tau\in\bbr,
\end{equation}
is dense in the torus $\bbt^d:=\bbr^d/\bbz^d$ if and only if the numbers $\xi_1,\dots,\xi_d$ are $\bbq$-linearly independent.
Now let $t_1,\dots,t_d$ be a lattice basis of $\fft$, and let $\xi_1,\dots,\xi_d$ be the coordinates of $\xi$ with respect to this basis.
The curve $\gamma_\xi$ corresponds to the curve in (\ref{Eq:StandardTorusCurve}) under the corresponding isomorphism $T\cong \bbt^d$ as in (\ref{Diag:TorusIsomorphism}).
The claim follows since the isomorphism is, in particular, a homeomorphism.
\end{proof}

With this we can now prove the following condition for frequency stabilization:

\begin{proposition}\label{Prop:FrequencyStabilizationConditions}
Let $M$ be a $K$-manifold where $K$ is a compact Lie group, and let $X$ be an equivariant vector field on $M$ with a relative equilibrium $m$ with velocity $\xi\in\ffk$.
Let $T^d$ be the torus:
\begin{equation*}
T^d:=\overline{\{
\exp(\tau\xi)K_m
\}}
\end{equation*}
in the Lie group $N(K_m)/K_m$, let $t_1,\dots,t_d$ be a lattice basis of the Lie algebra $\fft^d$ of $T^d$, and let $\xi_1,\dots,\xi_d$ be the corresponding coordinates of the vector $\xi+\ffk_m \in \fft^d$.
If $\psi\in C^\infty(M,\ffk)^K$ is an infinitesimal gauge transformation such that:
\begin{equation*}
\psi(m)+\ffk_m=-\xi_{j+1}t_{j+1}- \,\dots\, -\xi_dt_d,
\end{equation*}
then $\psi$ stabilizes the frequencies of $X$ at $m$ to order $j$.
\end{proposition}

\begin{proof}
Since $\xi\in\ffk$ is a velocity of the relative equilibrium $m$ for $X$, the integral curve of $X$ starting at $m$ is given by $\phi_m(\tau)=\exp(\tau\xi)\cdot m$ for $\tau\in\bbr$.
As described in the proof of Theorem \ref{Thm:FieldKrupaTorus}, the integral curve $\phi_m$ corresponds to the curve $\tau\mapsto \exp(\tau\xi)K_m$ on the Lie group $N(K_m)/K_m$, which has velocity $\xi+\ffk_m\in\ffk/\ffk_m$.
Therefore, by Lemma \ref{Lemma:QuasiPeriodicInTorus}, the numbers $\xi_1,\dots,\xi_d$ are $\bbq$-linearly independent.

The vector fields $X$ and $X+\partial(\psi)$ are isomorphic, so by Lemma \ref{Lemma:IsoShareRel}, $m$ is also a relative equilibrium of the vector field $X+\partial(\psi)$ and has velocity $\xi+\psi(m)$.
Hence, the integral curve of $X+\partial(\psi)$ is given by:
\begin{equation*}
\phi^\psi_m:\bbr\to M, \qquad \phi^\psi_m(\tau)=\exp(\tau(\xi+\psi(m)))\cdot m.
\end{equation*}
This corresponds to the curve:
\begin{equation*}
\bbr\to N(K_m)/K_m, \qquad \tau\mapsto \exp(\tau(\xi+\psi(m)))K_m,
\end{equation*}
in the Lie group $N(K_m)/K_m$, which has velocity $\xi+\psi(m)+\ffk_m\in\ffk/\ffk_m$.
Let $T^j$ be the subtorus of $T^d$ in $N(K_m)/K_m$ corresponding to the Lie subalgebra:
\begin{equation*}
\fft^j:=\text{span}_\bbr\{t_1,\dots,t_j\}
\end{equation*}
of the Lie algebra $\fft^d$ of $T^d$.
By the choice of $\psi$, we have that:
\begin{equation*}
\xi+\psi(m)+\ffk_m = \xi_1t_1 + \dots + \xi_jt_j,
\end{equation*}
where the numbers $\xi_1,\dots,\xi_j$ are $\bbq$-linearly independent.
Hence, the curve $\tau\mapsto \exp(\tau(\xi+\psi(m)))K_m$ is dense in the subtorus $T^j$ by Lemma \ref{Lemma:QuasiPeriodicInTorus}.
Consequently, $\psi$ stabilizes the frequencies of $X$ at $m$ to order $j$ as desired.
\end{proof}

\begin{example}
Let $n>0$ and consider the toric action of the standard $\bbt^n\cong (\bbs^1)^n\subseteq \bbc^n$ on the complex space $\bbc^n \cong \bbr^{2n}$ given by:
\begin{equation*}
\left(e^{i\theta_1},\dots, e^{i\theta_n}\right) \cdot (z_1,\dots, z_n) := \left(e^{i\theta_1}z_1,\dots, e^{i\theta_n}z_n\right),
\end{equation*}
for $\left(e^{i\theta_1},\dots, e^{i\theta_n}\right)\in \bbt^n$ and $(z_1,\dots,z_n)\in\bbc^n$.
As in Example \ref{Ex:T2C2}, the equivariant vector fields are of the form:
\begin{equation*}
X(z_1,\dots,z_n):=
\left(
\begin{array}{c}
f_1\left(|z_1|^2,\dots,|z_n|^2\right) z_1\\
\vdots\\
f_n\left(|z_1|^2,\dots,|z_n|^2\right) z_n\\
\end{array}
\right)
+
\left(
\begin{array}{c}
g_1\left(|z_1|^2,\dots,|z_n|^2\right) iz_1\\
\vdots\\
g_n\left(|z_1|^2,\dots,|z_n|^2\right) iz_n
\end{array}
\right),
\end{equation*}
for $(z_1,\dots,z_n)\in\bbc^n$ and where the $f_i:\bbr^n\to\bbr$ and $g_i:\bbr^n\to\bbr$ are smooth functions.
This is already in the form of Theorem \ref{Thm:Decomposition}.
That is, the decomposition in that theorem is global and the relative equilibria correspond to zeros of the map $f:\bbr^n\to\bbr^n$ defined by $f:=(f_1,\dots,f_n)$.

The isotropy of any nonzero point $w=(w_1,\dots,w_n)\in\bbc^n$ is trivial, so the Field-Krupa bound in Theorem \ref{Thm:FieldKrupaTorus} says that if $X$ has a relative equilibrium at $w$ then its independent frequencies can be anything between $0$ and $\text{rank}(N(\bbt^n_w)/\bbt^n_w)=\text{rank}(\bbt^n)=n$.
Let $X$ be an equivariant vector field, given by $f_i$ and $g_i$ as above, with a relative equilibrium at a nonzero point $w\in\bbc^n$.
Let $\psi^X:\bbc^n\to\bbr^n$ be the infinitesimal gauge transformation defined by:
\begin{equation*}
\psi^X(z_1,\dots,z_n):=
\left(
\begin{array}{c}
g_1\left(|z_1|^2,\dots,|z_n|^2\right)\\
\vdots\\
g_n\left(|z_1|^2,\dots,|z_n|^2\right)
\end{array}
\right)
\end{equation*}
for $(z_1,\dots,z_n)\in\bbc^n$.
Then $\psi^X(w)$ is a velocity of the relative equilibrium $w\in\bbc^n$.
Suppose, for the sake of simplicity in this example, that the number of independent frequencies is $n$.
That is, the numbers:
\begin{equation*}
\frac{\xi_i}{2\pi}:=\frac{g_i\left(w\right)}{2\pi}, \qquad i=1,\dots,n,
\end{equation*}
are $\bbq$-linearly idependent.
Then for any $0\le j \le n$, the infinitesimal gauge transformation:
\begin{equation*}
\psi:\bbc^n\to\bbr^n, \qquad
\psi(z_1,\dots,z_n):=
\left(
\begin{array}{c}
0 \\
\vdots\\
0\\
-g_{j+1}\left(|z_1|^2,\dots,|z_n|^2\right)\\
\vdots\\
-g_n\left(|z_1|^2,\dots,|z_n|^2\right)
\end{array}
\right)
\end{equation*}
stabilizes the frequencies of $X$ at $w$ to order $j$.
\end{example}

This example is representative of actions of tori.
That is, as we now show, we can always stabilize the frequencies to any order within the Field-Krupa bound:

\begin{theorem}\label{Thm:jStabilization}
Let $M$ be a $K$-manifold, where $K$ is a torus, and let $X$ be an equivariant vector field with a relative equilibrium at $m$.
Then there exists an infinitesimal gauge transformation $\psi\in C^\infty(M,\ffk)^K$ that stabilizes the frequencies of $X$ at $m$ to any order $j$ up to the dimension of the closure of the integral curve of $X$ starting at $m$. 
Furthermore, $\psi$ may be chosen so that $X+\partial(\psi)$ only differs from $X$ in an arbitrarily small neighborhood of the group orbit of $G\cdot m$.
\end{theorem}

\begin{proof}
By Theorem \ref{Thm:Decomposition}, there exists a decomposition of $X$:
\begin{equation*}
X=Y^X+\partial\left(\psi^X\right),
\end{equation*}
where $Y^X\in\ffX(M)^K$ has an equilibrium at $m$ and $\psi^X\in C^\infty(M,\ffk)^K$.
Hence, in particular $\psi^X(m)\in\ffk$ is a velocity for $m$.
Pick a $K$-invariant inner product on the Lie algebra $\ffk$.
Use this to identify the quotient $\ffk/\ffk_m$ with the orthogonal complement $\ffk_m^\perp$ of $\ffk_m$ in $\ffk$.
Let $\xi\in\ffk^\perp_m$ be the component of the velocity $\psi^X(m)\in\ffk$ in $\ffk^\perp_m$.
Note, in particular, that the vector $\xi$ is also a velocity of the relative equilibrium $m$ of $X$ (Remark \ref{Rem:VelocityFacts}).

Let $T^d$ be the following torus in $K$:
\begin{equation*}
T^d:=\overline{\{
\exp(\tau\xi)\mid \tau\in\bbr
\}},
\end{equation*}
and let $d$ be its dimension.
Pick a lattice basis $t_1,\dots,t_d$ of the Lie algebra $\fft^d$ of $T^d$.
Let $\xi_1,\dots,\xi_d$ be the coordinates of the velocity $\xi$ with respect to this basis, and let $0\le j \le d$.
By \ref{Prop:FrequencyStabilizationConditions}, to construct the desired infinitesimal gauge transformation $\psi$, it suffices to find one such that:
\begin{equation*}
\psi(m):=-\xi_{j+1}t_{j+1}- \,\dots\, -\xi_dt_d,
\end{equation*}

For this, complete the basis $t_1,\dots,t_d$ to a lattice basis $t_1,\cdots,t_d,t_{d+1},\cdots,t_n$ of the Lie algebra $\ffk$ of the torus $K$.
Use this basis to define an invariant inner product $\langle\cdot,\cdot\rangle$ (that's possibly different from the one chosen initially) by:
\begin{equation*}
\langle t_i,t_l \rangle := \delta_{i,l}, 
\qquad 
\text{where }\delta_{i,l}:=\begin{cases}
0 & i\not=l \\
1 & i=l.
\end{cases}
\end{equation*}
Then define the map:
\begin{equation}\label{Eq:TheStabilizer}
\psi:M\to\ffk, \qquad \psi(p):=-\sum_{i=j+1}^{d} \left\langle \psi^X(p),t_i\right\rangle t_i.
\end{equation}
Since $K$ is abelian, the Ajdoint reprsentation of $K$ on $\ffk$ is trivial.
Thus, for all $(k,p)\in K\times M$ we have:
\begin{align*}
\psi(k\cdot p)
&=-\sum_{i=j+1}^{d} \left\langle \psi^X(k\cdot p),t_i\right\rangle t_i\\
&=-\sum_{i=j+1}^{d} \left\langle \Ad(k)\psi^X(p),t_i\right\rangle t_i
\qquad\text{ since }\psi^X\text{ is equivariant}\\
&=-\sum_{i=j+1}^{d} \left\langle \psi^X(p),t_i\right\rangle t_i
\qquad\text{ since the }\Ad\text{ representation is trivial}\\
&=\psi(p),
\end{align*}
meaning that $\psi$ is an infinitesimal gauge transformation.
Furthermore:
\begin{equation*}
\psi(m)
=-\sum_{i=j+1}^{d} \left\langle \psi^X(m),t_i\right\rangle t_i
=-\xi_{j+1}t_{j+1}-\,\dots\,-\xi_dt_d
\end{equation*}
as required.
Hence, by Proposition \ref{Prop:FrequencyStabilizationConditions}, the infinitesimal gauge transformation $\psi$ stabilizes the frequencies of $X$ at $m$ to order $j$.

To see that $\psi$ may be chosen so that $X+\partial(psi)$ is the same as $X$ in an arbitrary small $K$-invariant neighborhood of $G\cdot m$ we can multiply $\psi$ by a $K$-invariant bump function on a sufficiently small $K$-invariant neighborhood of $m$ as we now describe.
Let $B$ be a small $K$-invariant neighborhood of the orbit $G\cdot m$ and let $\widehat B$ be a closed $K$-invariant set containing $B$.
Then let $\mu_B$ be a $K$-invariant smooth bump function for $B$.
That is, $\mu_B$ is such that:
\begin{equation*}
\mu_B(p):=
\begin{cases}
1 & \text{ if } p\in B \\
0 \le \mu_B(p) \le 1 & \text{ if } p \in \widehat B - B \\
0 & \text{ if } p \in M - \widehat B
\end{cases}
\end{equation*}
and $\mu_B(k\cdot p)=\mu_B(p)$ for all $(k,p)\in K\times M$.
Such an invariant bump function may be obtained by averaging over the group action any smooth bump function over the desired neighborhoods.
Then consider the infinitesimal gauge transformation:
\begin{equation*}
\widetilde\psi:M\to\ffk, \qquad \widetilde\psi(p):=\mu_B(p)\,\psi(p),
\end{equation*}
where $\psi:M\to\ffk$ is the infinitesimal gauge transformation in (\ref{Eq:TheStabilizer}).
Since $\mu_B(m)=1$ we have:
\begin{equation*}
\widetilde\psi(m)=\psi(m)=-\xi_{j+1}t_{j+1}-\,\dots\,-\xi_dt_d.
\end{equation*}
Hence, $\widetilde\psi$ stabilizes the frequencies of $X$ at $m$ to order $j$ just like $\psi$ also equals $X$ outside of the closed set $\widehat B$.
\end{proof}

%%%%%%%%%%%%%%%%%%%%%%%%%%%%%%%%
\section{Bifurcations of equivariant vector fields 
to and from relative equbilibria}\label{ch2}
In this section we consider bifurcations of $1$-parameter families of equivariant vector fields, or equivalently bifurcations of paths of equivariant vector fields.
Due to the presence of group symmetries, one expects bifurcations to relative equilibria in place of bifurcations to equilibria.
We prove a test for bifurcations to relative equilibria on representations (Theorem \ref{Thm:BifToRelEq}).
On the one hand, this test is conceptually simple: it is essentially a generalization of the Equivariant Branching Lemma considered up to isomorphism of equivariant vector fields.
On the other hand, it can be quite general: it can predict bifurcations to steady-state, periodic, or quasi-periodic motion on tori.
The motion on the bifurcating branches depends on an isomorphism of equivariant vector fields used in the test.
In particular, isomorphisms of equivariant vector fields are central to reducing and reconstructing the dynamics of the bifurcation in our test.

We then extend this test to bifurcations {\it from} relative equilibria in the presence of noncompact symmetries.
We reduce to the slice representation at the relative equilibrium.
Compared to similar slice reductions that can be found in the literature \cite{K90, FSSW96} we reconstruct the dynamics from the slice reduction using isomorphisms of equivariant vector fields.

In section \ref{paths}, we provide a straightforward extension of the category of equivariant vector fields to a category of paths of equivariant vector fields.
We also show that the decomposition at relative equilibria in Theorem \ref{Thm:Decomposition} extends to paths of equivariants vector fields.
In section \ref{bifurcationsrep} we prove the test for bifurcations to relative equilibria on representations mentioned before.
Finally, in section \ref{bifurcationsproper} we extend this test to proper actions.

\subsection{Categories of paths of equivariant vector fields}\label{paths}
In this subsection we show that the category of equivariant vector fields from subsection \ref{categoryVFs} extends naturally to $1$-parameter families of equivariant vector fields.

We will think of $1$-parameter families of equivariant vector fields on a $G$-manifold $M$ as ``smooth'' paths in the space of equivariant vector fields $\ffX(M)^G$.
Thus, we need to discuss what it means for such a path to be ``smooth''.
We must address the same question for paths in the space of gauge transformations $C^\infty(M,\ffg)^G$.
There are several ways to do it.
For instance, we can turn the spaces $\ffX(M)^G$ and $C^\infty(M,\ffg)^G$ into Fr\'echet spaces.
However, it is enough for our purposes to use the following simpler definition:

\begin{definition}\label{Def:SmoothPath}
Let $M$ be a $G$-manifold.
A map $X:\bbr\to\ffX(M)^G$ is a {\it smooth path of equivariant vector fields} on $M$ if the associated map:
\begin{equation*}
\widehat X: \bbr\times M\to TM, \qquad
\widehat X (\lambda,m) := X(\lambda)(m),
\end{equation*}
is smooth in the usual sense.
An analogous definition gives {\it smooth paths of infinitesimal gauge transformations} $\psi:\bbr\to C^\infty(M,\ffg)^G$.
\end{definition}

From now on we will assume all such paths are smooth, so we drop the word smooth when referring to them.
A path of infinitesimal gauge transformations $\psi:\bbr\to C^\infty\left(M,\ffg\right)^G$ on a $G$-manifold $M$ induces a path of equivariant vector fields $\partial(\psi):\bbr\to \ffX(M)^G$.
It is given by:
\begin{equation}\label{Eq:InducedPathVF}
\partial(\psi)_\lambda(m):=\frac{\d}{\d \tau}\Big|_0 \exp(\tau\psi_\lambda(m))\cdot m,
\end{equation}
for any $\lambda\in\bbr$ and any $m\in M$.
The map $\partial:C^\infty\left(\bbr,C^\infty(M,\ffg)^G\right)\to C^\infty\left(\bbr,\ffX(M)^G\right)$ is linear.
The abelian group $C^\infty\left(\bbr,C^\infty(M,\ffg)^G\right)$ acts on the space $C^\infty\left(\bbr,\ffX(M)^G\right)$.
The action is given by:
\begin{equation}\label{Eq:PathAction}
\psi\cdot X := X + \partial(\psi),
\end{equation}
where $\psi$ is a path of infinitesimal gauge transformations, $X$ is a path of equivariant vector fields, and the addition is the pointwise addition.
Note that the map $\partial$ is essentially the map for vector fields $\partial:C^\infty(M,\ffg)^G\to\ffX(M)^G$ in (\ref{Eq:InducedVF}) taken parameter-wise, and the action (\ref{Eq:PathAction}) is the action of infinitesimal gauge transformations $C^\infty(M,\ffg)^G$ on equivariant vector fields $\ffX(M)^G$ in (\ref{Eq:Action}) taken parameter-wise.
Put another way, the action of paths in (\ref{Eq:PathAction}) reduces ot the action on equivariant vector fields (\ref{Eq:Action}) when the paths are all constant.
Thus, we have a category:

\begin{definition}\label{Def:GpoidFamilies}
The {\it category } $C^\infty\left(\bbr,\bbx(M)^G\right)$ {\it of paths of equivariant vector fields} on a $G$-manifold $M$ is the action groupoid (see Remark \ref{Rem:ActionGroupoids}) of the action (\ref{Eq:PathAction}) of the space of paths of infinitesimal gauge transformations $C^\infty\left(\bbr,C^\infty(M,\ffg)^G\right)$ on the space of paths of equivariant vector fields $C^\infty\left(\bbr,\ffX(M)^G\right)$.
\end{definition}

\begin{remark}
As with the category of equivariant vector fields, the category of paths of equivariant vector fields in Definition \ref{Def:GpoidFamilies} is a $2$-vector space.
That is, the space of paths of equivariant vector fields and the space of paths of infinitesimal gauge transformations are vector spaces, and all the structure maps of the category are linear maps.
\end{remark}

\begin{definition}\label{Def:IsoFams}
Two paths of equivariant vector fields $X$ and $Y$ on a $G$-manifold $M$ are {\it isomorphic} if they are isomorphic as objects of the category of paths of equivariant vector fields (Definition \ref{Def:GpoidFamilies}).
That is, they are isomorphic if there exists a path of infinitesimal gauge transformations $\psi:\bbr\to C^\infty(M,\ffg)^G$ such that $Y=X+\partial(\psi)$.
\end{definition}

\begin{notation}
In the category of paths of equivariant vector fields on a $G$-manifold $M$, an isomorphism  $X\to Y$ is given by a pair $(\psi,X)$, where $\psi$ is an infinitesimal gauge transformation such that $Y=X+\partial(\psi)$.
We will sometimes refer to the path $\psi$ as an isomorphism between $X$ and $Y$ for the sake of simplicity.
\end{notation}

We note the equivalence in Theorem \ref{Thm:EquivalenceVFs} between the categories of equivariant vector fields $\bbx(V)^K\simeq \bbx(G\times^KV)^G$ can be extended to an equivalence between the categories of paths of equivariant vector fields:

\begin{theorem}\label{Thm:EquivalencePaths}
Let $V$ be a representation of a compact Lie subgroup $K$ of a Lie group $G$.
Then there is an equivalence of categories:
\begin{equation*}
C^\infty\left(\bbr,\bbx(G\times^KV)\right) \simeq C^\infty\left(\bbr,\bbx(V)^K\right)
\end{equation*}
between the categories of paths of equivariant vector fields on $G\times^KV$ and $V$ respectively (Definition \ref{Def:GpoidFamilies}).
In particular, the functor $E:\bbx(V)^K\to\bbx(G\times^KV)^G$ of Theorem \ref{Thm:InclusionFunctor}, a functor $P:\bbx(G\times^KV)^G\to\bbx(V)^K$ as in Theorem \ref{Thm:ProjectionFunctor}, and the corresponding natural isomorphism $h:EP\cong 1_{\bbx(G\times^KV)^G}$ can all be applied parameter-wise to obtain the equivalence on paths.
\end{theorem}

\begin{proof}
The functor $E:\bbx(V)^K\to\bbx(G\times^KV)^G$ of Theorem \ref{Thm:InclusionFunctor} can be used parameter-wise to obtain a functor:
\begin{equation}\label{Eq:InclusionFunctorPaths}
E:C^\infty\left(\bbr,\bbx(V)^K\right)\to C^\infty\left(\bbr,\bbx(G\times^KV)^G\right).
\end{equation}
The functor $E$ sends a path of equivariant vector fields $X:\bbr\to\ffX(V)^K$ to the path $E_0(X):\bbr\to\ffX(G\times^KV)^G$ defined by:
\begin{equation*}
(E_0X)_\lambda([g,v]):=g\cdot (Tj) X_\lambda(v),\qquad
[g,v]\in G\times^KV,\, \lambda\in\bbr,
\end{equation*}
where $j:V\hookrightarrow G\times^KV$ is the slice embedding defined by $j(v):=[1,v]$.
Similarly, the functor $E$ sends a path of infinitesimal gauge transformations $\psi:\bbr\to C^\infty(V,\ffk)^K$ to the path $E_1(\psi):\bbr\to C^\infty(V,\ffk)^K$ defined by:
\begin{equation*}
(E_1\psi)_\lambda([g,v]):=\Ad(g)\iota\left(\psi_\lambda(v)\right),\qquad
[g,v]\in G\times^KV,\, \lambda\in\bbr,
\end{equation*}
where $\iota:\ffk\hookrightarrow\ffg$ is the Lie algebra inclusion.

Similarly, the functor $P:\bbx(G\times^KV)^G\to\bbx(V)^K$ of Theorem \ref{Thm:ProjectionFunctor} corresponding to some $K$-equivariant splitting $\ffg=\ffk\oplus\ffq$ can be taken parameter-wise to obtain a fucntor:
\begin{equation}\label{Eq:ProjectionFunctorPaths}
P:C^\infty\left(\bbr,\bbx(G\times^KV)^G\right) \to C^\infty\left(\bbr,\bbx(V)^K\right).
\end{equation}
The functor $P$ sends a path of equivariant vector fields $X:\bbr\to\ffX(G\times^KV)^G$ to the path of equivariant vector fields $P_0(X):\bbr\to\ffX(V)^K$ defined by:
\begin{equation*}
(P_0X)(v):= j^* (\Phi\circ X_\lambda) (v), \qquad
v\in V,\, \lambda\in\bbr,
\end{equation*}
where $\Phi\in \Omega^1\big(G\times^KV; \calv(G\times^KV)\big)$ is the equivariant connection on the vector bundle $G\times^KV \to G/K$ corresponding to the splitting $\ffg=\ffk\oplus\ffq$.
Similarly, the functor $P$ sends a path of infinitesimal gauge transformations $\psi:\bbr\to C^\infty(G\times^KV,\ffg)^G$ to the path of infinitesimal gauge transformations $P_1(\psi):\bbr\to C^\infty(V,\ffk)^K$ defined by:
\begin{equation*}
(P_1\psi)_\lambda(v):=\bbp(\psi_\lambda(j(v))),\qquad
v\in V,\,\lambda\in\bbr,
\end{equation*}
where $\bbp:\ffg\to\ffk$ is the $K$-equivariant projection corresponding to the splitting $\ffg=\ffk\oplus\ffq$.

The functoriality of $E$ and $P$ obtained this way can be verified using Lemma \ref{Lemma:FunctorShortcut} in a completely analogous way to the proofs of Theorem \ref{Thm:InclusionFunctor} and Theorem \ref{Thm:ProjectionFunctor}.
In fact, as can be noted by the description above, for fixed values of $\lambda\in\bbr$ the values of the functor $E$ in (\ref{Eq:InclusionFunctorPaths}) and of $P$ in (\ref{Eq:ProjectionFunctorPaths}) on equivariant vector fields and infinitesimal gauge transformations agree with those of Theorem \ref{Thm:InclusionFunctor} and Theorem \ref{Thm:ProjectionFunctor}.
Similarly, the fact that $PE=1_{C^\infty(\bbr,\bbx(V)^K)}$ holds since it is true parameter-wise (for any fixed $\lambda\in\bbr$) by Theorem \ref{Thm:ExtendThenProject}.

Finally, a completely analogous proof to that of Theorem \ref{Thm:NaturalIsomorphism} yields, for a choice of functor $P$ as in (\ref{Eq:ProjectionFunctorPaths}), a linear map:
\begin{equation}\label{Eq:TheMaphPaths}
h:C^\infty\left(\bbr,\ffX(G\times^KV)\right)^G\to C^\infty\left(\bbr,C^\infty(G\times^KV,\ffg)^G\right),\qquad
X\mapsto h(X),
\end{equation}
such that:
\begin{equation}\label{Eq:DecompositionRevisitedPaths}
X=E_0(P_0(X))+\partial(h(X)),\qquad
X\in C^\infty\left(\bbr,\ffX(G\times^KV)^G\right).
\end{equation}
That is, $h$ defines a natural isomorphism $EP\cong 1_{C^\infty(\bbr,\bbx(G\times^KV)^G)}$.
Consequently, the functor $P$ and $E$ in (\ref{Eq:InclusionFunctorPaths}) and (\ref{Eq:ProjectionFunctorPaths}), and the natural isomorphism $h: EP\cong 1$ is an equivalence of categories.
\end{proof}

\subsection{Isomorphisms and bifurcations to steady-state, periodic, and quasi-periodic motion on representations}\label{bifurcationsrep}
Let $V$ be a representation of a compact Lie group $K$.
Given a path of equivariant vector fields $X:\bbr\to\ffX(V)^K$, one often looks for nontrivial solution curves of $X$.
That is, one looks for a curve $\gamma:=(\nu,\lambda):I\to V\times\bbr$ such that $\gamma(0)=(0,0)\in V\times\bbr$ and $\nu(\delta)\in V$ is an equilibrium of the vector field $X_{\lambda(\delta)}$ for all $\delta\in I$.
An equilibrium is, in particular, a relative equilibrium.
And a relative equilibrium tends to have motion dense in a torus with dimension $\text{rank}\left(N(K_m)/K_m\right)$, which in turn is often greater than $0$ (\cite{} and Theorem \ref{Thm:FieldKrupaTorus}).
Hence, it is more natural to look for {\it relative} solution curves.
In this subsection we prove a test for finding relative solution curves consisting of relative equilibria.
The test consists of checking isomorphic vector fields for bifurcations to {\it strict} equilibria, and using the isomorphisms to reduce and reconstruct the dynamics.
This test can predict and describe bifurcations to steady-state (strict equilibrium), periodic, or quasi-periodic motion.

We introduce the following definition:

\begin{definition}\label{Def:SolutionBranch}
Let $M$ be a $G$-manifold and let $X:\bbr\to\ffX(M)^G$ be a path of equivariant vector fields on $M$ such that a point $m\in M$ is a relative equilibrium of the vector fields $X_\lambda$ for $\lambda\in\bbr$.
A {\it relative solution curve} of $X$ is a curve $\gamma=(p,\lambda):I\to M\times\bbr$ such that $\gamma(0)=(m,0)\in M\times\bbr$ and such that:
\begin{equation*}
X_{\lambda(\delta)}(p(\delta)) \in T_{p(\delta)}\left(G\cdot p(\delta)\right),
\qquad
\text{ for } \delta\in I,
\end{equation*}
where $I$ is an open interval containing $0$.
That is, $p(\delta)$ is a relative equilibrium of the vector field $X_{\lambda(\delta)}$.
A relative solution set is a {\it trivial relative solution curve} if the curve $p(\delta)=m$ for all $\delta\in I$, and {\it nontrivial} otherwise.
The {\it invariant relative solution set} generated by a relative solution curve $\gamma$ is the set:
\begin{equation*}
\Gamma(\gamma):=G\cdot \gamma(I) \subseteq V\times\bbr,
\end{equation*}
where $G$ acts on $M\times\bbr$ by $g\cdot(p,\lambda)=(g\cdot p, \lambda)$ for all $(g,p,\lambda)\in G\times M \times \bbr$.
If, for all $\delta\in I$, the $p(\delta)\in M$ are equilibria of $X_{\lambda(\delta)}$, then we will simply say the curve $\gamma$ is a {\it solution curve} and the set $\Gamma(\gamma)$ is an {\it invariant solution set}.
\end{definition}

\begin{remark}
For a representation $V$ of a Lie group $K$, the origin $0\in V$ has a $0$-dimensional orbit space consisting of the point $0$ only, since the action is by linear transformations.
Thus, if the origin is a relative equilibrium of an equivariant vector field, it is in particular an equilibrium.
\end{remark}

Given a path of equivariant vector fields $X$ on a representation $V$ of a compact Lie group $K$, various results in the literature look for solution curves of $X$ in one-dimensional fixed-point subspaces of some subgroup $\Sigma\subseteq K$:
\begin{equation*}
\text{Fix}(\Sigma)=\left\{ v\in V \mid k\cdot v = v \text{ for all }k\in \Sigma\right\}.
\end{equation*}
For example, that is part of the content of the Equivariant Branching Lemma \cite{C81,
V82}.
As a first step towards a criterion for the existence of relative solution curves, we prove the following generalization of this method:

\begin{lemma}\label{Lemma:OneDimFirstStep}
Let $V$ be a finite-dimensional representation of a compact Lie group $K$, and let $W$ be a one-dimensional subspace of $V$.
Suppose that $Y:\bbr\to\ffX(V)^K$ is a path of equivariant vector fields and such that:
\begin{enumerate}
\item $Y_\lambda(0)=0$ for all $\lambda\in\bbr$.
\item There exists a small neighborhood $B\subseteq W$ of $0$ in $W$ and a small $\epsilon>0$ such that $Y_\lambda|B$ is tangent to $W$ for all $\lambda\in (-\epsilon,\epsilon)$.
\item For $\lambda\in (-\epsilon,\epsilon)$, the linearizations $D(Y_\lambda|B)(0)=\sigma(\lambda)\text{Id}_W$ are such that $\sigma(0)=0$ and $\sigma'(0)\not=0$.
\end{enumerate}
Then $Y$ has a strict solution curve $\gamma:I\to W\times\bbr\subseteq V\times\bbr$.
\end{lemma}

\begin{proof}
It is convenient to think of $Y$ as a map $Y:V\times\bbr\to V$, using that $V$ is a finite-dimensional vector space.
Let $w\in\bbr$ denote a coordinate on $W$.
Shrinking $\epsilon>0$ if necessary, we may suppose without loss of generality that $B$ is the set of all vectors in $W$ such that $|w|<\epsilon$ and that $Y(w,\lambda)\in W$ for all $w$ and $\lambda$ such that $|w|<\epsilon$ and $|\lambda|<\epsilon$.
Since $W$ is $1$-dimensional, the linearizations $D(Y|B)_\lambda(0)$ each have a single real eigenvalue $\sigma(\lambda)$, which by assumption satisfy $\sigma(0)=0$ and $\sigma'(0)\not=0$.
Thus, using Taylor's Theorem, we have:
\begin{align*}
(Y|B)(w,\lambda)
&=\sigma(\lambda) w + \calo(w^2)\\
&=\sigma(\lambda) w + w^2 h(w,\lambda) \qquad \text{ some smooth function }h\\
&=w \Big(\sigma(\lambda) + w h(w,\lambda) \Big).
\end{align*}
Note that the zero set $\{w=0\}$ is the trivial branch of equilibria of the $Y|B$, so set $F(w,\lambda):=\sigma(\lambda) + w h(w,\lambda)$.
Thus, nontrivial solutions correspond to zero sets of $F(w,\lambda)$.
Now note that $F(0,0)=\sigma(0)+0 h(w,\lambda) = 0$ and that:
\begin{equation*}
\frac{\partial F}{\partial \lambda}(0,0)
=\Big( \sigma'(\lambda) + w \frac{\partial h}{\partial \lambda}(w,\lambda)\Big)|_{\lambda=0,w=0}
=\sigma'(0) + 0 \frac{\partial h}{\partial \lambda}(0,0)
=\sigma'(0)
\not=0.
\end{equation*}
Thus, by the Implicit Function Theorem, there exists a small interval $I$ around $0$ and a smooth map $\widehat\gamma:I\to \bbr$ with $\widehat\gamma(0)=0$ such that:
\begin{equation}\label{Eq:BranchEq}
Y(w,\widehat\gamma(w))= w F(w,\widehat\gamma(w)) = 0, \qquad w\in I.
\end{equation}
That is, we have an invariant solution curve of $Y$ given by:
\begin{equation*}
\gamma:I\to V\times\bbr, \qquad \gamma(w):=\left(w\,\overrightarrow{w_0},\widehat\gamma(w)\right),
\end{equation*}
where $\overrightarrow{w_0}$ is the basis vector of $W$ giving the coordinate $w$ used above.
\end{proof}

Even though it considers arbitrary one-dimensional subspaces, Lemma \ref{Lemma:OneDimFirstStep} is a rather restrictive test: not many equivariant vector fields are tangent to a given one-dimensional subspace.
It also only yields {\it strict} solution curves.

We can simultaneously expand the paths of equivariant vector fields that can be considered and find {\it relative} solution curves by considering paths of equivariant vector fields {\it up to isomorphism}.
By requiring that a path $X:\bbr\to\ffX(V)^K$ is only isomorphic to a path $Y:\bbr\to\ffX(V)^K$ satisfying the assumptions of Lemma \ref{Lemma:OneDimFirstStep}, we can consider more paths of equivariant vector fields.
That's because we are allowing the restrictions $X_\lambda|W$ to be tangent to the group orbits.
On the other hand, recall that isomorphic vector fields share relative equilibria \ref{Lemma:IsoShareRel}, so any {\it strict} solution curve that $Y$ has on $W$ will yield a {\it relative} solution curve of $X$.
Furthermore, the path of infinitesimal gauge transformations giving the isomorphism also gives velocities for the relative equilibria.

Putting all this together, we prove the following test, which is the main result of this subsection:

\begin{theorem}\label{Thm:BifToRelEq}
Let $V$ be a representation of a compact Lie group $K$, let $W$ be a $1$-dimensional subspace of $V$, let $X:\bbr\to\ffX(V)^K$ be a path of equivariant vector fields, and let $\psi:\bbr\to C^\infty(V,\ffk)^K$ be a path of infinitesimal gauge transformations.
Suppose that:
\begin{enumerate}
\item $X$ has a trivial branch of solutions at $0\in V$.
\item The path of equivariant vector fields $Y:=X-\partial(\psi)$ is tangent to $W$ in a neighborhood $B$ of $0$ in $W$.
\item The eigenvalues $\sigma(\lambda)\in\bbr$ of the linearizations of $Y|B$ satisfy $\sigma(0)=0$ and $\sigma'(0)\not=0$.
\end{enumerate}
Then $X$ has a relative solution curve $\gamma=(\nu,\lambda):I\to V\times\bbr$.
Furthermore, for $\delta\in I$, the relative equilibrium $\nu(\delta)\in V$ of $X_{\lambda(\delta)}$ has velocity $\psi_{\lambda(\delta)}(\nu(\delta))\in \ffk$.
\end{theorem}

\begin{proof}
Note that the path $Y:\bbr\to\ffX(V)^K$ of equivariant vector fields satisfies the assumptions of Lemma \ref{Lemma:OneDimFirstStep}.
Hence, there exists a solution curve $\gamma=(\nu,\lambda):I\to W\times\bbr$ of $Y$.
That is, for every $\delta\in I$, the point $\nu(\delta)\in W$ is an equilibrium of the vector field $Y_{\lambda(\delta)}$.
Note that the vector field $Y_{\lambda(\delta)}$ is isomorphic to the vector field $X_{\lambda(\delta)}$, in the sense of Definition \ref{Def:IsoVF}, since the infinitesimal gauge transformtion $\psi_{\lambda(\delta)}:V\to\ffk$ is such that:
\begin{equation*}
X_{\lambda(\delta)}=Y+\partial(\psi_{\lambda(\delta))}).
\end{equation*}
Thus, by Lemma \ref{Lemma:IsoShareRel}, the point $\nu(\delta)\in W$ is a relative equilibrium of the vector field $X_{\lambda(\delta)}$ with velocity $\psi_{\lambda(\delta)}(\nu(\delta))\in\ffk$.
Hence, the curve $\gamma=(\nu,\delta):I\to W\times\bbr$ is a relative solution curve as desired. 
\end{proof}

\begin{example}\label{Ex:S1}
Consider the standard representation of $\bbs^1$ on $\bbc$ by rotations given in Example \ref{Ex:Circle}.
Bifurcations to relative equilibria are well-understood in this representation.
Nevertheless, it serves as a simple illustration of the general strategy for using Theorem \ref{Thm:BifToRelEq}.
Using the discussion in Example \ref{Ex:Circle}, the paths of equivariant vector fields are of the form:
\begin{equation*}
X(z,\lambda)=f\left(|z|^2,\lambda\right) z + g\left(|z|^2,\lambda\right) i z , \qquad z\in\bbc,\, \lambda\in\bbr.
\end{equation*}
Any such path is isomorphic to its radial part:
\begin{equation*}
Y(z,\lambda)=f\left(|z|^2,\lambda\right) z , \qquad z\in\bbc,\, \lambda\in\bbr,
\end{equation*}
via the isomorphism given by the path $\psi:\bbc\times\bbr\to\bbr$ defined by:
\begin{equation*}
\psi(z,\lambda):=g\left(|z|^2,\lambda\right), \qquad z\in\bbc,\, \lambda\in\bbr.
\end{equation*}
The radial part $Y$ restricts to the real axis, so let $W$ be the real axis as a subspace of $\bbc$.
The linearizations $D(Y|W)_\lambda(0)$ have eigenvalues $f(0,\lambda)\in\bbr$.
Hence, if $\frac{\partial f}{\partial \lambda}(0,0)\not=0$, Theorem \ref{Thm:BifToRelEq} says there exists a curve:
\begin{equation*}
\gamma:I\to\bbr\times\bbr, \qquad \gamma(\delta)=(\delta, \lambda(\delta)),
\end{equation*}
such that:
\begin{equation*}
X(\delta, \lambda(\delta)) = g\left(\delta^2,\lambda(\delta)\right) i\delta, \qquad \delta\in I.
\end{equation*}
In particular, the point $\delta\in\bbc$ is a relative equilibrium of $X_{\lambda(\delta)}$ with velocity $g\left(\delta^2,\lambda(\delta)\right)\in\bbr$.
Hence, by Theorem \ref{Thm:FieldKrupaTorus} and Lemma \ref{Lemma:QuasiPeriodicInTorus}, the integral curve of $X_{\lambda(\delta)}$ starting at $\delta\in\bbc$ is:
\begin{equation*}
\begin{cases}
\text{steady-state} & \text{ if } g\left(\delta^2,\lambda(\delta)\right)=0 \\
\text{periodic with period } \frac{2\pi}{g\left(\delta^2,\lambda(\delta)\right)} & \text{ if } g\left(\delta^2,\lambda(\delta)\right)\not=0.
\end{cases}
\end{equation*}
The invariant relative solution set $\Gamma(\gamma)$ corresponding to this curve consists of circles.
That is, the circle:
\begin{equation*}
C_{\delta}:=\left\{
z\in\bbc \mid |z|^2=\delta^2
\right\}
\end{equation*} 
is a collection of relative equilibria of the vector field $X_{\lambda(\delta)}$.
The velocity of a relative equilibrium $\delta e^{i\theta}\in C_{\delta}$ is given by $\psi_{\lambda(\delta)}(\delta e^{i\theta})=g\left(\delta^2,\lambda(\delta)\right)\in\bbr$.
From this, we can recover the well-known case of bifurcations to circle limit cycles arising from a pitchfork bifurcation in the radial vector field.
\end{example}

%%%%%%%%%%%%%%%%%%%%%%%%%%%%%%%%

\begin{example}\label{Ex:TheTorusExample}
Consider the representation of the torus $\bbt^2$ on the product $\bbr^4\cong\bbc^2$ discussed in Example \ref{Ex:T2C2}.
For a nonzero point $\overrightarrow{z}=(z_1,z_2)\in\bbc^2$, the Field-Krupa bound in Theorem \ref{Thm:FieldKrupaTorus} is either $1$ (when only one coordinates is zero) or $2$ (when neither coordinates is $0$).
Hence, relative equilibria may exhibit both periodic and quasi-periodic motion from relative equilibria in this case.
We use Theorem \ref{Thm:BifToRelEq} to describe how bifurcations to each may occur, including how bifurcations to equilibria may become bifurcations to periodic or quasi-periodic motion for an isomorphic path of vector fields.

Using the discussion in Example \ref{Ex:T2C2}, the paths of equivariant vector fields are of the form:
\begin{equation}\label{Eq:PathT2C2}
X(z_1,z_2,\lambda)=
\left(\begin{array}{c}
f_1\left(|z_1|^2,|z_2|^2,\lambda\right)z_1 + g_1\left(|z_1|^2,|z_2|^2,\lambda\right) iz_1\\
f_2\left(|z_1|^2,|z_2|^2,\lambda\right)z_2 + g_2\left(|z_1|^2,|z_2|^2,\lambda\right) iz_2
\end{array}\right),
\end{equation}
where $(z_1,z_1)\in\bbc^2,\,\lambda\in\bbr$, and where $f_i:\bbr^2\times\bbr\to\bbr$ and $g_i:\bbr^2\times\bbr\to\bbr$ are smooth functions.
The functions $g_i$ define a path of infinitesimal gauge transformations.
Hence, the path $X$ is isomorphic to its radial parts given by:
\begin{equation*}
Y(z_1,z_2,\lambda)=
\left(\begin{array}{c}
f_1\left(|z_1|^2,|z_2|^2,\lambda\right)z_1\\
f_2\left(|z_1|^2,|z_2|^2,\lambda\right)z_2
\end{array}\right),
\end{equation*}
for $(z_1,z_1)\in\bbc^2,\,\lambda\in\bbr$.
We restrict $Y$ to two different one-dimensional subspaces and apply Theorem \ref{Thm:BifToRelEq} in each.

First, consider the subspace:
\begin{equation*}
W:=\left\{
(x,0)\in\bbc^2 \mid x\in\bbr
\right\}.
\end{equation*}
The path $Y$ restricts to $W$ to give the path:
\begin{equation*}
(Y|W)=
\left(\begin{array}{c}
f_1\left(x^2,0,\lambda\right)x\\
0
\end{array}\right),
\qquad x\in\bbr,\,\lambda\in\bbr.
\end{equation*}
The eigenvalues of the linearizations $D(Y|W)_\lambda(0)$ are $f_1(0,0,\lambda)\in\bbr$.
Hence, if $f_1(0,0,0)=0$ and $\frac{\partial f_1}{\partial\lambda}(0,0,0)\not=0$, by Theorem \ref{Thm:BifToRelEq} there exists a curve:
\begin{equation*}
\gamma:I\to W\times\bbr, \qquad \gamma(\delta)=\left(\delta,0,\widehat\gamma(\delta)\right),
\end{equation*}
such that $(\delta,0)$ is a relative euqilibrium of $X(-,-,\widehat\gamma(\delta))$ with velocity given by:
\begin{equation*}
\psi(\delta,0,\widehat\gamma(\delta))=
\left(\begin{array}{c}
g_1\left(\delta^2,0,\widehat\gamma(\delta)\right)\\
g_2\left(\delta^2,0,\widehat\gamma(\delta)\right)
\end{array}\right).
\end{equation*}
The stabilizer of this point is $\{1\}\times \bbs^1$, so the velocity modulo the Lie algebra of the stabilizer can be taken to be the value:
\begin{equation*}
\psi(\delta,0,\widehat\gamma(\delta)) + \{0\}\times\bbr \cong g_1\left(\delta^2,0,\widehat\gamma(\delta)\right).
\end{equation*}
Hence, by Theorem \ref{Thm:FieldKrupaTorus} and Lemma \ref{Lemma:QuasiPeriodicInTorus}, the motion of the relative equilibrium is steady-state if $g_1\left(\delta^2,0,\widehat\gamma(\delta)\right)=0$ and periodic if this is nonzero.
As in the case of Example \ref{Ex:S1}, the relative invariant solution set corresponding to this relative solution curve consists of circles of relative equilibria.
The real axis in the second copy of $\bbc^2$ yields an analogous case of bifurcations to periodic trajectories.

For potentially quasi-periodic bifurcations consider the subspace:
\begin{equation*}
W:=\left\{ (x,x)\in\bbc^2 \mid x\in\bbr \right\}.
\end{equation*}
Suppose that $f_1$ and $f_2$ in (\ref{Eq:PathT2C2}) agree on the diagonal $\Delta:=\{(x,x)\in\bbr^2\mid x\in\bbr\}$ in $\bbr^2$.
Then $Y$ restricts to $W$ to give the path:
\begin{equation*}
(Y|W)(x,x,\lambda)=
\left(\begin{array}{c}
f_1\left(x,x,\lambda\right)x\\
f_2\left(x,x,\lambda\right)x
\end{array}\right)
=
f(x,\lambda)
\left(\begin{array}{c}
x\\
x
\end{array}\right),
\qquad x\in\bbr,\,\lambda\in\bbr,
\end{equation*}
where $f(x,\lambda):=f_1(x,x,\lambda)=f_2(x,x,\lambda)$.
The eigenvalues of the linearizations $D(Y|W)_\lambda(0)$ are $f(0,\lambda)\in\bbr$.
Hence, if $f(0,0)=0$ and $\frac{\partial f}{\partial\lambda}(0,0)\not=0$, by Theorem \ref{Thm:BifToRelEq}, there exists a curve:
\begin{equation*}
\gamma:I\to W\times\bbr, \qquad \gamma(\delta)=\left(\delta,\delta,\widehat\gamma(\delta)\right),
\end{equation*}
such that $(\delta,\delta)\in\bbc^2$ is a relative equilibrium of the vector field $X(-,-,\widehat\gamma(\delta))$.
The relative invariant solution set $\Gamma(\gamma)$ corresponding to this relative solution curve consists of tori:
\begin{equation*}
T_{\delta,\delta}=\left\{
\left(\delta e^{i\theta},\delta e^{i\varphi}\right) \in \bbc^2 \mid \theta,\varphi \in\bbr
\right\}.
\end{equation*}
The velocity of a relative equilibrium $\left(\delta e^{i\theta},\delta e^{i\varphi}\right)$ of the vector field $X(-,-,\widehat\gamma(\delta))$ is given by the value of the infinitesimal gauge transformation:
\begin{equation*}
\psi\left(\delta e^{i\theta},\delta e^{i\varphi},\widehat\gamma(\delta)\right) :=
\Big(
g_1\left(\delta^2, \delta^2, \widehat\gamma(r)\right),
g_2\left(\delta^2, \delta^2, \widehat\gamma(r)\right)
\Big)\in\bbr^2.
\end{equation*}
The stabilizer of this point is trivial if $\delta\not=0$.
Hence, by Theorem \ref{Thm:FieldKrupaTorus} and Lemma \ref{Lemma:QuasiPeriodicInTorus}, the motion of this relative equilibrium depends on the numbers:
\begin{equation*}
\frac{g_1\left(\delta^2, \delta^2, \widehat\gamma(r)\right)}{2\pi},\, \frac{g_2\left(\delta^2, \delta^2, \widehat\gamma(r)\right)}{2\pi}.
\end{equation*}
If the value of the infinitesimal gauge transformation is $(0,0)\in\bbr^2$, then the motion is steady-state.
If these numbers are $\bbq$-linearly dependent, then the motion is periodic.
If these numbers are $\bbq$-linearly independent, then the motion is quasi-periodic and dense in a $2$-torus.
\end{example}

\subsection{Isomorphisms and bifurcations on proper actions}\label{bifurcationsproper}
In this subsection we consider bifurcations to relative equilibria on proper actions.
In particular, we extend the test for bifurcations to relative equilibria in subsection \ref{bifurcationsrep} to this context (Theorem \ref{Thm:BifToRelEqProper}).
We do this by reducing bifurcation problems to the slice representation using the decompositions in subsection \ref{decomposition}.
More generally we show how relative solution curves of a given path of equivariant vector fields correspond to relative solution curves on the slice representation (Theorem \ref{Thm:BifBranchEquivalence}), and how isomorphisms of equivariant vector fields relate the velocities of the corresponding relative equilibria.

Consider a proper $G$-manifold $M$, and let $X:\bbr\to\ffX(M)^G$ be a path of equivariant vector fields on $M$ with a relative equilibrium at a point $m\in M$.
The question of equivariant bifurcation from $m$ is local near the group orbit $G\cdot m$.
Therefore, throughout the subsection we fix a finite-dimensional real representation $V$ of a compact Lie subgroup $K$ of a Lie group $G$, and consider the associated bundle $G\times^KV$ (see Remark \ref{Rem:TubularNeighborhood}).

Recall that a choice of $K$-equivariant splitting $\ffg=\ffk\oplus\ffq$ gives rise to a projection of equivariant vector fields $P_0:\ffX(G\times^KV)^G\to\ffX(V)^K$ (Theorem \ref{Thm:ProjectionFunctor}).
We can apply this projection parameter-wise to a path $X:\bbr\to\ffX(G\times^KV)^G$ to obtain a path of equivariant vector fields on $V$:
\begin{equation*}
P_0(X):\bbr\to\ffX(V)^K, \qquad P_0(X)_\lambda:=P_0\left(X_\lambda\right).
\end{equation*}
Similarly, the canonical inclusion of equivariant vector fields $E_0:\ffX(V)^K\to\ffX(G\times^KV)^G$ of Theorem \ref{Thm:InclusionFunctor} and the natural isomorphism $h:\ffX(G\times^KV)^G\to C^\infty(G\times^KV,\ffg)^G$ can be applied parameter-wise.
Thus, as described in Theorem \ref{Thm:EquivalencePaths}, we have a decomposition of {\it paths} of equivariant vector fields:
\begin{equation*}
X=E_0(P_0(X))+\partial(h(X)), \qquad X\in C^\infty(\bbr,\ffX(G\times^KV)^G).
\end{equation*}
As we now show, this decomposition lets us find relative solution curves of $X$ on $G\times^KV$ by finding relative solution curves of $P_0(X)$ on $V$.

\begin{theorem}\label{Thm:BifBranchEquivalence}
Let $X:\bbr\to\ffX(G\times^KV)^G$ be a path of equivariant vector fields on $G\times^KV$ with the point $[1,0]\in G\times^KV$ as a relative equilibrium.
Let $P_0(X):\bbr\to\ffX(V)^K$ be the projected path on $V$ with respect to the chosen projection $P_0$.
Then a curve:
\begin{equation*}
\gamma:I\to V\times\bbr, \qquad \gamma(\delta):=(\nu(\delta),\,\lambda(\delta)),
\end{equation*}
is a relative solution curve of $P_0(X)$ starting at $0\in V$ if and only if the curve:
\begin{equation*}
j_*\gamma:I\to(G\times^KV)\times\bbr, \qquad j_*\gamma(\delta):=\left([1,\nu(\delta)],\,\lambda(\delta)\right),
\end{equation*}
is a relative solution curve of $X$ starting at $[1,0]\in G\times^KV$.
Furthermore, for $\delta\in I$, a vector $\xi\in\ffk$ is a velocity of the relative equilibrium $\nu(\delta)\in V$ of $P_0(X_{\lambda(\delta)})$ if and only if the vector:
\begin{equation*}
\xi+h\left(X_{\lambda(\delta)}\right)([1,\nu(\delta)])\in\ffg
\end{equation*}
is a velocity of the relative equilibrium $[1,\nu(\delta)]\in G\times^KV$ of $X_{\lambda(\delta)}$, where $h$ is the natural isomorphism $E\circ P\cong 1_{\bbx(G\times^KV)^G}$ from Theorem \ref{Thm:NaturalIsomorphism}.
\end{theorem}

\begin{proof}
Suppose first that $\gamma:I\to V\times\bbr$ is a relative solution curve of $P_0(X)$ starting at $0\in V$.
That means that, for $\delta\in I$, the point $\nu(\delta)\in V$ is a relative equilibrium of the vector field $P_0(X)_{\lambda(\delta)}$.
By Lemma \ref{Lemma:EPreservesRelEq}, the inclusion by equivariant extension map $E_0:\ffX(V)^K\to\ffX(G\times^KV)^G$ preserves relative equilibria.
Thus, the point $[1,\nu(\delta)]$ is a relative equilibrium of the vector field $E_0\left(P_0\left(X_{\lambda(\delta)}\right)\right)$.
By Theorem \ref{Thm:EquivalenceVFs} we have that:
\begin{equation}\label{Eq:TheEquivalenceHere}
X_{\lambda(\delta)} = E_0\left(P_0\left(X_{\lambda(\delta)}\right)\right) + \partial\left( h\left(X_{\lambda(\delta)}\right)\right),
\end{equation}
where $h\left(X_{\lambda(\delta)}\right)\in C^\infty\left(G\times^KV,\ffg\right)^G$.
Hence, the vector fields $X_{\lambda(\delta)}$ and $E_0\left(P_0\left(X_{\lambda(\delta)}\right)\right)$ are isomorphic.
By Lemma \ref{Lemma:IsoShareRel}, isomorphic equivariant vector fields share relative equilibria.
Thus, the point $[1,\nu(\delta)]$ is a relative equilibrium of the vector field $X_{\lambda(\delta)}$.
Hence, the curve $j_*\gamma$ as in the statement is a relative solution curve of the path $X$.

Conversely, let $j_*\gamma:I\to(G\times^KV)\times\bbr$ be a relative solution curve of $X$ as in the statement.
That means that, for $\delta\in I$, the point $[1,\nu(\delta)]\in G\times^KV$ is a relative equilibrium of the vector field $X_{\lambda(\delta)}$.
By Lemma \ref{Lemma:PPreservesRelEq}, the equivariant projection $P_0:\ffX(G\times^KV)^G\to\ffX(V)^K$ in the statement preserves relative equilibria.
Thus, the point $\nu(\delta)$ is a relative equilibrium of the vector field $P_0\left(X_{\lambda(\delta)}\right)$.
Hence, the curve $\gamma$ as in the statement is a relative solution curve of the path $P_0(X)$.

Now suppose that we have such relative solution curves and that $\xi\in\ffk$ is a velocity of the relative equilibrium $\nu(\delta)$ of $P_0\left(X_{\lambda(\delta)}\right)$, then $\xi$ is a velocity of $E_0\left(P_0\left(X_{\lambda(\delta)}\right)\right)$ since this vector field is $j$-related to $P_0\left(X_\lambda(\delta)\right)$ by the definition of .
And since $X_{\lambda(\delta)}$ and $E_0\left(P_0\left(X_{\lambda(\delta)}\right)\right)$ are isomorphic via the infinitesimal gauge transformation $h\left(X_{\lambda(\delta)}\right)$, the velocity of the relative equilibrium $[1,\nu(\delta)]$ of $X_{\lambda(\delta)}$ is:
\begin{equation*}
\xi+h\left(X_{\lambda(\delta)}\right)([1,\nu(\delta)]),
\end{equation*}
by Lemma \ref{Lemma:IsoShareRel}.

Conversely, suppose that $\xi+h\left(X_{\lambda(\delta)}\right)([1,\nu(\delta)])\in\ffg$ is a velocity of the relative equilibrium $[1,\nu(\delta)]$ of $X_{\lambda(\delta)}$.
By Lemma \ref{Lemma:IsoShareRel}, the velocity of the relative equilibrium $[1,\nu(\delta)]$ of $E_0\left(P_0\left(X_{\lambda(\delta)}\right)\right)$ is:
\begin{equation*}
\xi+h\left(X_{\lambda(\delta)}\right)([1,\nu(\delta)]) - h\left(X_{\lambda(\delta)}\right)([1,\nu(\delta)]) = \xi.
\end{equation*}
Since $E_0\left(P_0\left(X_{\lambda(\delta)}\right)\right)$ and $P_0\left(X_{\lambda(\delta)}\right)$ are $j$-related, the velocity of the relative equilibrium $\nu(\delta)$ of $X_{\lambda(\delta)}$ is also $\xi$.
\end{proof}

As mentioned in Remark \ref{Rem:LiteratureObservation1}, the decomposition of the path of equivariant vector fields $X$ used in Theorem \ref{Thm:BifBranchEquivalence} (see Theorem \ref{Thm:EquivalenceVFs} for the decomposition) is similar to a decomposition used by Krupa for compact group actions on Euclidean space \cite{K90}.
In fact, Krupa used his version to show that bifurcations of the components transversal to the group orbits leads to bifurcations to relative equilibria of the original vector fields.

One of the main benefits to the framework implicit in Theorem \ref{Thm:BifBranchEquivalence} is how it addresses the dependence on the choices involved in the reduction (see Proposition \ref{Prop:ProjectionChoice} and Proposition \ref{Prop:SliceChoice}).
For example, if we choose different invariant Riemannian metrics to define the Lie algebra splittings that give the projections or if we choose a different slice for the action through the point.
Here we show how these choices affect the velocities of the relative solution curves bifurcating from the relative equilibrium.
Together with the description of the relative solution curves and their velocities in Theorem \ref{Thm:BifBranchEquivalence}, the following results (Proposition \ref{Prop:BifurcationProjectionChoices} and Propsition \ref{Prop:BifurcationSliceChoices}) mean that one can describe the relative solution curves of the original path $X$ and their velocities regardless of the projection or slice chosen to perform the reduction.
We first address the choice of projection:

\begin{proposition}\label{Prop:BifurcationProjectionChoices}
Let $X:\bbr\to\ffX(G\times^KV)^G$ be a path of equivariant vector fields on $G\times^KV$ with the point $[1,0]\in G\times^KV$ as a relative equilibrium.
Suppose $P^1:\bbx(G\times^KV)^G\to\bbx(V)^K$ and $P^2:\bbx(G\times^KV)^G\to\bbx(V)^K$ are two projection functors as in Theorem \ref{Thm:ProjectionFunctor} (corresponding to different $K$-invariant splittings of $\ffg$).
Then a curve:
\begin{equation*}
\gamma:I\to V\times\bbr, \qquad \gamma(\delta):=(\nu(\delta),\,\lambda(\delta)),
\end{equation*}
is a relative solution curve of the path $P^1_0(X)$ starting at $0\in V$ if and only if is a relative solution curve of the path $P^2_0(X)$ starting at $0\in V$.
Furthermore, for $\delta\in I$, a vector $\xi\in\ffk$ is a velocity of the relative equilibrium $\nu(\delta)\in V$ of the vector field $P^1_0\left(X_{\lambda(\delta)}\right)$ if and only if the vector:
\begin{equation*}
\xi+P^2_1\left(h^1(X_{\lambda(\delta)})\right)(\nu(\delta))\in \ffk
\end{equation*}
is a velocity of the relative equilibrium $\nu(\delta)\in V$ of the vector field $P^2_0\left(X_{\lambda(\delta)}\right)$, where $h^1\left(X_{\lambda(\delta)}\right)\in C^\infty(G\times^KV,\ffg)^G$ is the map corresponding to the natural isomorphism $h^1:E\circ P^1 \cong 1_{\bbx(G\times^KV)^G}$ from Theorem \ref{Thm:NaturalIsomorphism}.
\end{proposition}

\begin{proof}
For every $\delta\in I$,  the vector fields $P^1_0\left(X_{\lambda(\delta)}\right)$ and $P^2_0\left(X_{\lambda(\delta)}\right)$ on $V$ are isomorphic by Proposition \ref{Prop:ProjectionChoice}.
In particular, the infinitesimal gauge transformation:
\begin{equation*}
P^2_1\left(h^1(X_{\lambda(\delta)})\right) : V \to \ffk,
\end{equation*}
where $h^1(X_{\lambda(\delta)}) \in C^\infty(G\times^KV,\ffg)^G$  corresponds to the natural isomorphism $h^1:E\circ P^1 \cong 1_{\bbx(G\times^KV)^G}$ from Theorem \ref{Thm:NaturalIsomorphism}, is such that:
\begin{equation*}
P^2_0\left(X_{\lambda(\delta)}\right) = P^1_0\left(X_{\lambda(\delta)}\right) + \partial\left( P^2_1\left(h^1(X_{\lambda(\delta)})\right) \right).
\end{equation*}
Thus, by Lemma \ref{Lemma:IsoShareRel}, the point $\nu(\delta)\in V$ is a relative equilibrium of $P^1_0\left(X_{\lambda(\delta)}\right)$ with velocity $\xi\in\ffk$ if and only if it is a relative equilibrium of $P^2_0\left(X_{\lambda(\delta)}\right)$ with velocity:
\begin{equation*}
\xi+P^2_1\left(h^1(X_{\lambda(\delta)})\right)(\nu(\delta)),
\end{equation*}
as claimed.
\end{proof}

We now address the choice of slice:

\begin{proposition}\label{Prop:BifurcationSliceChoices}
Let $X:\bbr\to\ffX(G\times^KV)^G$ be a path of equivariant vector fields on $G\times^KV$ with the point $[1,0]\in G\times^KV$ as a relative equilibrium.
Suppose $D$ is another slice for the action through $[1,0]$, with equivariant embedding $\iota:D\hookrightarrow G\times^KV$, and let $\phi:D\hookrightarrow V$ be the $K$-equivariant embedding as in Proposition \ref{Prop:SliceChoice}.
Let $P^V_0(X):\ffX(G\times^KV)^G\to\ffX(V)^K$ and $P^D_0:\ffX(G\times\iota(D))^G\to\ffX(D)^K$ be equivariant projections as in Theorem \ref{Thm:ProjectionFunctor} with respect to the same splitting on $\ffg$.
Then a curve:
\begin{equation*}
\gamma:I\to D\times\bbr, \qquad \gamma(\delta):=(d(\delta),\,\lambda(\delta)),
\end{equation*}
is a relative solution curve of $P^D_0(X)$ on $D$ if and only if the curve:
\begin{equation*}
\phi_*\gamma:I\to V\times\bbr, \qquad \phi_*\gamma(\delta):=(\phi(d(\delta)),\lambda(\delta)),
\end{equation*}
is a relative solution curve of $P^V_0(X)$ on $V$.
Furthermore, for $\delta\in I$, a vector $\xi\in\ffk$ is a velocity of the relative equilibrium $d(\delta)\in D$ of $P^D_0(X_{\lambda(\delta)})$ if and only if the vector:
\begin{equation*}
\xi
+
P^V_1\Big(h^D(X_{\lambda(\delta)})\Big)\Big(\phi(d(\delta))\Big)
\in \ffk
\end{equation*}
is a velocity of the relative equilibrium $\phi(d(\delta))\in V$ of $P^V_0(X_{\lambda(\delta)})$, where $h^D\left(X_{\lambda(\delta)}\right)\in C^\infty(G\cdot\iota(D),\ffg)^G$ is the map corresponding to the natural isomorphism $h^D:E^D\circ P^D \cong 1_{\bbx(G\cdot\iota(D))^G}$ from Theorem \ref{Thm:NaturalIsomorphism}.
\end{proposition}

\begin{proof}
Suppose that $\gamma$ is a relative solution curve of $P^D_0(X)$.
Then, for $\delta\in I$, the point $d(\delta)\in D$ is a relative equilibrium of the vector field $P^D_0\left(X_{\lambda(\delta)}\right)$ on $D$.
Suppose the velocity of this relative equilibrium is $\xi\in\ffk$.
Since $\phi:D\to V$ is a $K$-equivariant diffeomorphism onto its image, the point $\phi(d(\delta))\in V$ is a relative equilibrium of $\phi_*P^D_0\left(X_{\lambda(\delta)}\right)$ on $V$ with velocity $\xi\in\ffk$.
By Theorem \ref{Prop:SliceChoice}, we have:
\begin{equation*}
P^V_0\left(X_{\lambda(\delta)}\right)=
\phi_*\Big(P^D_0\left(X_{\lambda(\delta)}\right)\Big)
+ \partial\Big( P^V_1\left(h^D\left(X_{\lambda(\delta)}\right)\right)\Big).
\end{equation*}
Hence, the vector field $\phi_*P^D_0\left(X_{\lambda(\delta)}\right)$ on $V$ is isomorphic to the vector field $P^V_0\left(X_{\lambda(\delta)}\right)$ on $V$ via the infinitesimal gauge transformation $P^V_1\left(h^D\left(X_{\lambda(\delta)}\right)\right):V\to\ffk$.
In other words, since isomorphic vector fields share relative equilibria by Lemma \ref{Lemma:IsoShareRel}, the point $\phi(d(\delta))\in V$ is a relative equilibrium of the vector field $P^V_0\left(X_{\lambda(\delta)}\right)$ with velocity:
\begin{equation*}
\xi
+
P^V_1\Big(h^D(X_{\lambda(\delta)})\Big)\Big(\phi(d(\delta))\Big)
\in \ffk
\end{equation*}
as claimed.
The argument for the converse is completely analogous.
\end{proof}

Theorem \ref{Thm:BifBranchEquivalence} gives a way to generalize the test for bifurcation to relative equilibria in Theorem \ref{Thm:BifToRelEq} to proper actions:

\begin{theorem}\label{Thm:BifToRelEqProper}
Let $X:\bbr\to\ffX(G\times^KV)^G$ be a path of equivariant vector fields on $G\times^KV$ with the point $[1,0]\in G\times^KV$ as a relative equilibrium.
Let $P_0(X):\bbr\to\ffX(V)^K$ be the projected path on $V$ with respect to the chosen projection $P_0$.
Suppose there exists a $1$-dimensional subspace $W$ of $V$ and a path $\psi:\bbr\to C^\infty(V,\ffk)^K$ of infinitesimal gauge transformations on $V$ such that:
\begin{enumerate}
\item $P_0(X)\in C^\infty(\bbr,\ffX(V)^K)$ has a trivial branch of solutions at $0\in V$.
\item The path of equivariant vector fields $Y:=P_0(X)-\partial(\psi)$ is tangent to $W$ in a neighborhood $B$ of $0$ in $W$.
\item The eigenvalues $\sigma(\lambda)\in\bbr$ of the linearizations of $Y|B$ satisfy $\sigma(0)=0$ and $\sigma'(0)\not=0$.
\end{enumerate}
Then $X$ has a relative solution curve:
\begin{equation*}
\gamma:I\to (G\times^KV)\times\bbr, \qquad \gamma(\delta)=\left([g(\delta),\nu(\delta)],\lambda(\delta)\right),
\end{equation*}
such that the velocity of the relative equlibrium $[g(\delta),\nu(\delta)]$ of $X_{\lambda(\delta)}$ is given by:
\begin{equation*}
\xi_\delta:=
\psi_{\lambda(\delta)} (\nu(\delta))
+ h\left(X_{\lambda(\delta)}\right) \left([1,\nu(\delta)]\right) \in \ffg,
\end{equation*}
where $h(X):\bbr\to C^\infty(G\times^KV,\ffg)^G$ is the path of infinitesimal gauge transformations given by the natural isomorphism of Theorem \ref{Thm:NaturalIsomorphism}.
\end{theorem}

\begin{proof}
The path of equivariant vector fields $P_0(X):\bbr\to\ffX(V)^K$ satisfies the assumptions of Theorem \ref{Thm:BifToRelEq} with respect to the subspace $W$ of $V$ and the path $\psi:\bbr\to C^\infty(V,\ffk)^K$.
Hence, there exists a relative solution curve $\gamma=(\nu,\lambda):I\to V\times\bbr$ of $P_0(X)$ starting at $0\in V$.
Furthermore, for $\delta\in I$, the velocity of the relative equilibrium $\nu(\delta)\in V$ of $P_0(X)_{\lambda(\delta)}\in \ffX(V)^K$ is given by $\psi_{\lambda(\delta)}\left(\nu(\delta)\right)\in\ffk$.
By Theorem \ref{Thm:BifBranchEquivalence}, the curve:
\begin{equation*}
j_*\gamma:I\to (G\times^KV)\times\bbr, \qquad 
j_*\gamma(\delta)=\left([1,\nu(\delta)],\lambda(\delta)\right),
\end{equation*}
is a relative solution curve of the path $X$.
By the same theorem, for $\delta\in I$, the velocity of the relative equilibrium $[1,\nu(\delta)]\in G\times^KV$ is given by:
\begin{equation*}
\xi_\delta:=
\psi_{\lambda(\delta)} (\nu(\delta))
+ h\left(X_{\lambda(\delta)}\right) \left([1,\nu(\delta)]\right),
\end{equation*}
as claimed.
\end{proof}

\begin{comment}
\begin{example}\label{Ex:Cylinder}

\end{example}

\begin{example}\label{Ex:OneRotationOnR3}

\end{example}

\begin{example}\label{Ex:OneRotationS2}

\end{example}
\end{comment}

%%%%%%%%%%%%%%%%%%%%%%%%%%%%%%%%
\section{Generic conditions for equivariant vector fields}\label{ch:residual}
In this section, we consider open and dense subsets of equivariant vector fields.
The first main result of this section is that open and dense collections of equivariant vector fields are ``preserved'' by isomorphisms of equivariant vector fields (Theorem \ref{Thm:MainResidual1}).
That is, the set of all equivariant vector fields isomorphic to those in an open and dense subset is also open and dense in the space of equivariant vector fields.
The second main result of this section is that the equivalence in Theorem \ref{Thm:EquivalenceVFs} ``preserves'' open and dense collections of equivariant vector fields up to isomorphism (Theorem \ref{Thm:MainResidual2}).
That is, in particular, the reduction to the slice representation via equivariant projection and the reconstruction of the dynamics via equivariant extension from this slice preserve open and dense subsets of equivariant vector fields \textit{up to isomorphism}.
These theorems also apply to paths of equivariant vector fields, and so they can be helpful for equivariant bifurcation problems.

As noted in the introduction, the equivariant projection and extension in decomposition (\ref{Eq:DecompositionIntro}) don't need to strictly preserve open and dense subsets.
In particular, the equivariant extension of an open and dense subset of vector fields on the slice representation doesn't need to be open and dense.
For this, note the equivariantly extended vector fields are all vertical in the corresponding associated bundle over the group orbit of the relative equilibrium, which is not an open and dense condition (see the discussion preceding (\ref{Eq:Decomposition})).
The result that such collections are preserved \textit{up to isomorphism} addresses this issue.

We begin by endowing the spaces of equivariant vector fields and of paths of equivariant vector fields with the Whitney $C^\infty$ topologies (section \ref{whitney}).
When endowed with a topology, the category $\bbx(M)^G$ becomes a topological abelian $2$-group.
That is, it becomes a category internal to the category of topological abelian groups: its space of objects and morphisms are topological abelian groups, and all the structure maps are continuous group homomorphisms.
This topological abelian $2$-group structure proves to be the necessary key for proving the main results of this section, so we discuss it in subsection \ref{top2groups}.
We prove theorems \ref{Thm:MainResidual1} and \ref{Thm:MainResidual2} in subsection \ref{opendense}.

\subsection{Whitney topologies}\label{whitney}
In this subsection, we describe the topologies we will use, and some preliminary results we need in the rest of this section.
We point to the literature when the proofs of lemmas can be found there.
However, we could not find proof of some lemmas that we needed, so proof is provided for those results here.
Recall the definition of the Whitney topologies:

\begin{definition}\label{Def:WhitneyTop}
Let $U$ and $V$ be smooth manifolds.
\begin{itemize}
\item Given an integer $r\in \bbz_{\ge 0}$, let $J^r(U,V)$ be the space of $r$-jets of mappings from $U$ to $V$.
For a subset $O$ of $J^r(U,V)$ define the collection:
\[
\calb^r(O):=\left\{f\in C^\infty(U,V)\mid j^rf(U)\subseteq O\right\}
\]
The {\bf Whitney $C^r$-topology} on $C^\infty(U,V)$ is the topology generated by the basis:
\[
\calb^r:=
\left\{\calb^r(O)\mid O\text{ is an open subset of }J^r(U,V)\right\}
\]
We will refer to the space $C^\infty(U,V)$ equipped with the Whitney $C^r$ topology as a {\bf Whitney $C^r$ space}.
\item The {\bf Whitney $C^\infty$-topology} on $C^\infty(U,V)$ is the topology generated by the basis:
\[
\calb^\infty:=\bigcup_{r=0}^\infty\calb^r.
\]
We will refer to the space $C^\infty(U,V)$ equipped with the Whitney $C^\infty$ topology as a {\bf Whitney $C^\infty$ space}.
\end{itemize}
\end{definition}

\begin{remark}\label{Rem:WhitneyTopIntuition}
Let $U$ and $V$ be smooth manifolds.
Following Golubitsky and Guillemin \cite[p. 43]{GG73}, we can get some intuition for the Whitney $C^r$-topology on $C^\infty(U,V)$ as follows.
Pick a distance function $d$ on the space of $r$-jets $J^r(U,V)$, compatible with the topology on $J^r(U,V)$.
Let $f$ be an arbitrary smooth map in $C^\infty(U,V)$, and let $\delta:U\to \bbr_+$ be a continuous function.
Then the set:
\[
B_\delta(f):=\left\{g\in C^\infty(U,V)\mid d(j^rf(u),j^rg(u))<\delta(u) \text{ for all }u\in U\right\}
\]
is a neighborhood of $f$ in the Whitney $C^r$-topology.
One can think of $B_\delta(f)$ as those maps in $C^\infty(U,V)$ that are, together with their first $r$ partial derivatives, $\delta$-close to the map $f$ and its first $r$ partial derivatives.
In fact, the collection:
\[
\left\{B_\delta(f)\mid \delta:U\to\bbr_+ \text{ is a continuous function}\right\}
\]
forms a neighborhood basis for the map $f$ in the Whitney $C^r$-topology.
\end{remark}

\begin{lemma}\label{Lemma:WhitneyProperties}
Let $U,V,W$, and $B$ be manifolds, and let $f:V\to W$, $g:V\to B$, and $h:W\to B$ be smooth maps.
\begin{enumerate}
\item The map:
\begin{equation*}
f_*:C^\infty(U,V)\to C^\infty(U,W),
\qquad
h\mapsto fh,
\end{equation*}
is continuous with respect to the Whitney topologies.
\item The canonical bijection of sets:
\begin{equation*}
C^\infty(U,V\times W)
\cong
C^\infty(U,V)\times C^\infty(U,W)
\end{equation*}
is a homeomorphism with respect to the Whitney topologies.
\item The canonical bijection of sets:
\begin{equation*}
C^\infty\left(U,V\,\fp{g,B,h}\,W\right)
\cong
C^\infty(U,V)\,\,\fp{g_*,C^\infty(U,B),h_*}\,\,C^\infty(U,W)
\end{equation*}
is a homeomorphism with respect to the Whitney topologies.
\end{enumerate}
\end{lemma}

\begin{proof}
See \cite[Proposition~3.5]{GG73} for (1) and \cite[Proposition~3.6]{GG73} for (2).
The continuity of the maps in the bijection of (3) follows by viewing the fiber products:
\begin{equation*}
V\fp{g,B,h}W
\qquad\qquad
C^\infty(U,V)\,\,\fp{g_*,C^\infty(U,B),h_*}\,\,C^\infty(U,W)
\end{equation*}
as subspaces of the products $V\times W$ and $C^\infty(U,V)\times C^\infty(U,W)$ respectively, and then applying parts (1) and (2) and the universal property of the subspace topology.
\end{proof}

\begin{lemma}\label{Lemma:DecoupledProducts}
Let $U,V,X$ and $Y$ be smooth manifolds.
Then the map:
\begin{equation*}
C^\infty(X,V)\times C^\infty(Y,W) \to C^\infty(X\times Y, U\times V),\qquad
(f,g)\mapsto f\times g,
\end{equation*}
where the map $f\times g: X\times Y \to U\times V$ is given by $(f\times g)(x,y):=(f(x),g(y))$, is a continuous map with respect to the Whitney topologies.
\end{lemma}

\begin{proof}
The proof of this fact is completely analogous to the proof of \cite[Proposition~3.10]{GG73}.
\end{proof}

\begin{remark}
Part (1) of Lemma \ref{Lemma:WhitneyProperties} says that pushforwards by smooth maps are continous with respect to the Whitney topologies.
As discussed in the notes in \cite[p.~49]{GG73}, pullbacks by smooth maps are in general {\it not} continuous with respect to the Whitney topologies. However, the following lemmas are two special cases of interest to us where the pullback {\it is} continuous.
\end{remark}

\begin{lemma}\label{Lemma:ProperPullbacks}
Let $U,V,$ and $W$ be manifolds, and let $f:V\to U$ be a smooth proper map.
Then the pullback:
\begin{equation*}
f^*:C^\infty(U,W)\to C^\infty(V,W),\qquad
h\mapsto hf,
\end{equation*}
is continuous with respect to the Whitney topologies.
\end{lemma}

\begin{proof}
See the second note in \cite[p.~49]{GG73}.
\end{proof}

\begin{lemma}\label{Lemma:WhitneyEquivariantProperty}
Let $K$ be a compact Lie group, let $P\xrightarrow{\pi}B$ be a principal $K$-bundle, and let $N$ be a manifold with a trivial action of $K$.
Then the map:
\begin{equation*}
\pi^*:C^\infty(B,N)\to C^\infty(P,N)^K,
\qquad
f\mapsto f\pi,
\end{equation*}
is a homeomorphism.
The inverse of $\pi^*$ is the map that takes an equivariant map $f : P \to N$ to the unique map $\widetilde f: B\to N$ such that $ f=\widetilde f\pi$.
\end{lemma}

\begin{proof}
Since the group $K$ is compact, the bundle projection $\pi:P\to B$ is a proper map.
Hence, the pullback $\pi^*:C^\infty(B,N)\to C^\infty(P,N)^K$ is continuous by Lemma \ref{Lemma:ProperPullbacks}.
The remaining task is to show the continuity of the inverse map $(\pi^*)^{-1}$.
The inverse map $(\pi^*)^{-1}$ sends a $K$-invariant map $f:P\to N$ to the unique map $\widetilde f:B\to N$ such that the following diagram commutes:
\begin{equation*}
\xy
{(-10,7)}*+{P} = "1";
{(-10,-7)}*+{B} = "2";
{(10,7)}*+{N} = "3";
{\ar@{->}_{\pi} "1";"2"};
{\ar@{->}^{f} "1";"3"};
{\ar@{-->}_{\widetilde f} "2";"3"};
\endxy
\end{equation*}
It suffices to show that the basis subsets given in Definition \ref{Def:WhitneyTop}:
\begin{equation*}
\left\{
\calb^r(\calo)
\mid
r\in \bbz_{\ge 0},\,
\calo \subseteq J^r(B,N),\,
\calo \text{ open}      
\right\}
\end{equation*}
have open preimages under the map $(\pi^*)^{-1}$ in the mapping space $C^\infty(P,N)^K$.
We will show that the preimages $((\pi^*)^{-1})^{-1}\left(\calb^r(\calo)\right)
=\pi^*\left(\calb^r(\calo)\right)$ are themselves basis sets:
\begin{equation}\label{Eq:PullbackGoal}
\pi^*\left(\calb^r(\calo)\right) 
=\calb^r\left(F_r^{-1}(\calo)\right),
\end{equation}
where $F_r$ is a continuous map we define next.

Let $J^r(P,N)^K$ be the $r$-jets of $K$-invariant maps $P\to N$ and let $F_r$ be the map given by:
\begin{equation*}
F_r:J^r(P,N)^K\to J^r(B,N),\qquad
j^rf(x_0)\mapsto j^r\widetilde f (\pi(x_0)).
\end{equation*}
To verify that the map $F_r$ is continuous, pick local coordinates $U\subseteq \bbr^b$ for $B$, $V\subseteq \bbr^k$ for $K$, and $W\subseteq \bbr^n$ for $N$, and note that the $K$-invariant maps are represented by maps $f:U\times V\to W$ that are independent of the $V$ variables.
Furthermore, given a map $f:U\times V\to W$, let $(u_0,v_0)$ be a point in $U\times V$, let $T_rf(u_0,v_0)$ be the coefficients of the $r^{\text{th}}$-order Taylor polynomial at the point $(u_0,v_0)$, let $T_r^U(u_0,v_0)$ denote the coefficients of the $r^\text{th}$-order Taylor polynomial at the point $(u_0,v_0)$ consisting only of those partial derivatives with respect to the $U$-variables only, and let $T_r^C(u_0,v_0)$ correspond to the rest of the coefficients in $T_r(u_0,v_0)$.
Then note that the map $\widetilde f = (\pi^*)^{-1}(f)$ has $r^\text{th}$-order Taylor polynomial such that $T_r\widetilde f(u_0)=T_r^Uf(u_0,v_0)$.
Hence, the map $F_r$ is continuous since it is just the projection:
\begin{equation*}
j^rf(u_0,v_0)=\left(u_0,v_0,T_r^Uf(u_0,v_0),T_r^Cf(u_0,v_0)\right)\mapsto
\left(u_0,T_r^Uf(u_0,v_0)\right),
\end{equation*}
for any map $f:U\times V\to W$, independent of the $V$ variables, and any point $(u_0,v_0)\in U\times V$.

We now proceed to verify (\ref{Eq:PullbackGoal}).
Consider an arbitrary map $f\in \pi^*\left(\calb^r(\calo)\right)$.
Then $\widetilde f = (\pi^*)^{-1} f\in \calb^r(\calo)$, so the image $(j^r\widetilde f)(B)$ is contained in the open set $\calo$.
Taking the preimage of this inclusion under the map $F_r$ we obtain that:
\begin{equation*}
j^rf(P)
\subseteq
F_r^{-1}\left(F_r\left(j^rf(P)\right)\right)
\subseteq
F_r^{-1}\left((j^r\widetilde f)(B)\right)
\subseteq 
F_r^{-1}(\calo),
\end{equation*}
where we also use that $F_r(j^rf(P))=j^r\widetilde f(B)$.
Consequently the map $f$ is an element of the basis set $\calb^r(F_r^{-1}(\calo))$ as desired.

For the converse inclusion, consider an arbitrary map $f\in \calb^r(F_r^{-1}(\calo))$.
Then the image $(j^rf)(P)$ is contained in the preimage $F_r^{-1}(\calo)$.
Taking the image of this inclusion under the map $F_r$ we obtain that:
\begin{equation*}
F_r(j^rf(P))\subseteq 
F_r\left(F_r^{-1}(\calo)\right)
\subseteq 
\calo,
\end{equation*}
which in turn implies that:
\begin{equation*}
(j^r\widetilde f)(B)
=F_r(j^rf(P))
\subseteq 
\calo.
\end{equation*}
Consequently, the map $\widetilde f = (\pi^*)^{-1}(f)$ is in the basis set $\calb^r(\calo)$, meaning that the map $f$ is an element of the preimage $\pi^*\left(\calb^r(\calo)\right)$.
This proves the equality (\ref{Eq:PullbackGoal}) and hence the continuity of the inverse map $(\pi^*)^{-1}$.
Hence, the map $\pi^*$ is a homeomoprhism with respect to the Whitney $C^\infty$ topology.
\end{proof}

\begin{corollary}\label{Cor:EquivariantExtensionContinuous}
Let $G$ be a Lie group, let $K$ be a compact Lie subgroup of $G$, let $V$ be a finite-dimensional real representation of the compact Lie group $K$, and let $N$ be a $G$-manifold.
Furthermore, let $G\times^KV$ be the quotient of the action of the group $K$ on the product $G\times V$ given by:
\begin{equation}\label{Eq:KActionOnGV}
k\cdot (g,v):=\left(gk^{-1},k\cdot v\right),\qquad
k\in K,\, (g,v)\in G\times V.
\end{equation}
Then the map:
\begin{equation*}
\epsilon:C^\infty(V,N)^K\to C^\infty(G\times^KV,N)^G,\qquad
f\mapsto \epsilon(f),
\end{equation*}
where the map $\epsilon (f):G\times^KV\to N$ is defined by $\epsilon(f)([g,v]):=g\cdot f(v)$, is continuous with respect to the Whitney $C^\infty$-topologies on the spaces of equivariant maps (see Notation \ref{Notation}.\ref{Not:EquivariantInvariant}).
\end{corollary}

\begin{remark}\label{Rem:EquivariantExtension}
We call the map $\epsilon:C^\infty(V,N)^K\to C^\infty(G\times^KV,N)^G$ in the statement of Corollary \ref{Cor:EquivariantExtensionContinuous} the {\it equivariant extension} of maps from the representation $V$ to the associated bundle $G\times^KV$.
\end{remark}

\begin{proof}[Proof of Corollary \ref{Cor:EquivariantExtensionContinuous}]
We define $\epsilon$ as the composition of well-defined and continuous maps that we describe next.
First, note that the map:
\begin{equation}\label{Eq:IdentityInclusion}
\id_G\times(-):C^\infty(V,N)^K \hookrightarrow C^\infty(G,G)\times C^\infty(V,N),\qquad
f\mapsto (\id_G,f),
\end{equation}
where $\id_G$ is the identity map of $G$, is a well-defined continuous inclusion.
By Lemma \ref{Lemma:DecoupledProducts}, the map:
\begin{equation}\label{Eq:Product}
C^\infty(G,G)\times C^\infty(V,N) \to C^\infty(G\times V, G\times N),\,
(\varphi,f) \mapsto \varphi \times f,
\end{equation}
where $\varphi\times f$ is defined by $(\varphi\times f)(g,v):=(\varphi(g),f(v))$, is a well-defined continuous map.
Furthermore, let the action of the group $G$ on the manifold $N$ be given by the map $\text{ac}:G\times N\to N$.
The pushforward:
\begin{equation}\label{Eq:ActionPushforward}
\text{ac}_*:C^\infty(G\times V, G\times N) \to C^\infty(G\times V, N),\qquad
\psi \mapsto \text{ac}\circ \psi,
\end{equation}
is a continuous map by part (1) of Lemma \ref{Lemma:WhitneyProperties}.
The composition of the maps in (\ref{Eq:IdentityInclusion}), (\ref{Eq:Product}), and (\ref{Eq:ActionPushforward}) is the map:
\begin{equation}\label{Eq:eMapFirst}
e:C^\infty(V,N)^K\to C^\infty(G\times V, N), \qquad f \mapsto e(f),
\end{equation}
where:
\begin{equation*}
e(f):G\times V\to N, \qquad e(f)(g,v):= g\cdot f(v).
\end{equation*}
The map $e$ is continuous with respect to the Whitney $C^\infty$ topologies since it is the composition of continuous maps.

Note that for any $f\in C^\infty(V,N)^K$, the map $e(f):G\times V \to N$ is $K$-invariant with respect to the action of $K$ given by (\ref{Eq:KActionOnGV}) since for all $k\in K$ and all $(g,v)\in G\times V$ we have:
\begin{equation*}
e(f)(gk^{-1},k\cdot v)
=gk^{-1}\cdot f(k\cdot v)
=gk^{-1}k\cdot f(v)
=g\cdot f(v)
=e(f)(g,v),
\end{equation*}
where we use that $f$ is $K$-equivariant.
Hence, the map $e$ restricts to a continuous map:
\begin{equation*}
\overline{e}:C^\infty(V,N)^K\to C^\infty(G\times V, N)^{K-\text{inv}}, \qquad f\mapsto e(f).
\end{equation*}

On the other hand, let $\pi:G\times V\to G\times^K V$ be the quotient map of the quotient space $G\times^KV$.
Then, since $\pi:G\times V\to G\times^KV$ is a principal $K$-bundle (see Notation \ref{Notation}.\ref{Not:AssociatedBundle}), the pullback:
\begin{equation*}
\pi^*:C^\infty\left(G\times^KV, N\right) \to C^\infty(G\times V, N)^{K-\text{inv}},\qquad
f\mapsto f\pi,
\end{equation*}
is a homeomorphism by Lemma \ref{Lemma:WhitneyEquivariantProperty}.
Its inverse is the map:
\begin{equation*}
(\pi^*)^{-1}:C^\infty(G\times V, N)^{K-\text{inv}}\to C^\infty\left(G\times^KV, N\right),
\qquad h \mapsto \widehat h,
\end{equation*}
where $\widehat h$ is the unique map such that $h=\widehat h\pi$.

Now consider the composition:
\begin{equation*}
\widetilde\epsilon:C^\infty(V,N)^K\to C^\infty(G\times^KV,N), \qquad f\mapsto (\pi^*)^{-1}(\overline{e}(f)).
\end{equation*}
By the definition of $(\pi^{*})^{-1}$, for any $f\in C^\infty(V,N)^K$, the map $\widetilde\epsilon$ is such that:
\begin{equation}\label{Eq:ThePiEquation}
\overline{e}(f)=\widetilde\epsilon(f) \circ \pi.
\end{equation}
Thus, for any $(g,v)\in G\times V$, we have that:
\begin{align}\label{Eq:DefTildeEpsilon}
\begin{split}
&\widetilde\epsilon(f)([g,v])\\
&=\widetilde\epsilon(f) \circ \pi (g,v)
\qquad \text{by the definition of }\pi\\
&=\overline{e}(f)(g,v)
\qquad \text{by (\ref{Eq:ThePiEquation})}\\
&=g\cdot f(v).
\end{split}
\end{align}
Furthermore, observe that $\widetilde\epsilon(f)$ is $G$-equivariant with respect to the action of the group $G$ on the associated bundle $G\times^KV$ (see Notation \ref{Notation}.\ref{Not:AssociatedBundle}) since for all $g'\in G$ and all $[g,v]\in G\times^KV$ we have that:
\begin{align*}
&\widetilde\epsilon(f) (g'\cdot [g,v])\\
&=\widetilde\epsilon(f)([g'g,v])\\
&=g'g\cdot f(v) \qquad \text{by (\ref{Eq:DefTildeEpsilon})}\\
&=g'\cdot (g\cdot f(v)) \\
&= g' \cdot \widetilde\epsilon(f)([g,v]) \qquad \text{by (\ref{Eq:DefTildeEpsilon})}.
\end{align*}
Hence, the map:
\begin{equation*}
\epsilon:C^\infty(V,N)^K\to C^\infty(G\times^KV,N)^G, \qquad f\mapsto (\pi^*)^{-1}(\overline{e}(f)),
\end{equation*}
is the desired map in the statement of the corollary.
Note that this also shows that the map $\epsilon$ is continuous since the map $\widetilde\epsilon$ is the composition of continuous maps, the space $C^\infty(G\times^KV,N)^G$ has the subspace topology as a subspace of the space $C^\infty(G\times^KV,N)$ with the Whitney $C^\infty$ topology, and we have $\widetilde\epsilon = \iota\circ\epsilon$, where $\iota:C^\infty(G\times^KV,N)^G\hookrightarrow C^\infty(G\times^KV,N)$ is the subspace inclusion.
\end{proof}

\subsection{Topological abelian $2$-groups}\label{top2groups}
As we show in the following subsection, the category of equivariant vector fields and the category of paths of equivariant vector fields are topological abelian $2$-groups when equipped with the Whitney topologies.
It is convenient to describe and prove results for topological abelian $2$-groups in the abstract.
We begin this subsection with the definition of a topological abelian $2$-group:

\begin{definition}[topological abelian $2$-group]\label{Def:TopAb2Group}
A {\it topological abelian $2$-group} is a small category $\calc$ internal to the category $\TopAb$ of topological abelian groups.
Equivalently, it is a small category $\calc$, having a topological abelian group $\calc_1$ of morphisms and a topological abelian group $\calc_0$ of objects, such that all the structure maps of the category are continuous group homomorphisms.
\end{definition}

\begin{example}[topological action groupoids]\label{Ex:ActionGroupoid}
Let $\partial:A_1\to A_0$ be a continuous group homomorphism between two topological abelian groups $A_1$ and $A_0$.
The map $\partial$ defines a continuous action of the group $A_1$ on the group $A_0$ via the action map:
\begin{equation}\label{Eq:ActionMap}
A_1\times A_0 \to A_0,
\qquad
(\psi,a)\mapsto \partial(\psi)+a.
\end{equation}
This action gives a corresponding action groupoid $A_1\ltimes A_0$ with set of arrows given by the group $A_1\times A_0$, set of objects given by the group $A_0$, and the following structure maps:
\begin{itemize}
\item The source map is given by:
\[
s:A_1\times A_0 \to A_0,
\qquad
s(\psi,a):=a.
\]
\item The target map $t:A_1\times A_0\to A_0$ is the action map in (\ref{Eq:ActionMap}).
\item The unit map is given by:
\[
u:A_0\to A_1\times A_0,
\qquad
u(a):=(0,a).
\]
\item The composition is given by:
\[
(A_1\times A_0)\fp{s,A_0,t}(A_1\times A_0)\to A_1\times A_0,
\qquad
\Big((\varphi,b),(\psi,a)\Big)\mapsto (\varphi+\psi,a).
\]
\end{itemize}
All the structure maps are continuous group homomorphisms with the canonical topological abelian group structure on the domains and targets inherited from that of the groups $A_1$ and $A_0$.
Thus, the action groupoid $A_1\ltimes A_0$ is a topological abelian $2$-group.
\end{example}

We will also make use of the following:

\begin{definition}[topological abelian $2$-subgroup]\label{Def:2Subgroup}
Let $\calc$ be a topological abelian $2$-group.
A {\it topological abelian $2$-subgroup} of $\calc$ is a subcategory $\cald$ of $\calc$ such that $\cald$ forms a topological abelian $2$-group where the group of objects $
\cald_0$ and the group of morphisms $\cald_1$ are topological abelian subgroups of $\calc_0$ and $\calc_1$, respectively, equipped with the subspace topology.
\end{definition}

We now define the corresponding $1$-morphisms between topological abelian $2$-groups:

\begin{definition}[continuous $2$-group homomorphism]\label{Def:2GroupHom}
Let $\calc$ and $\cald$ be topological abelian $2$-groups.
A {\it $2$-group homomorphism} between $\calc$ and $\cald$ is a functor $F:\calc\to\cald$ such that the corresponding map on objects $F_0:\calc_0\to\cald_0$ and the corresponding map on morphisms $F_1:\calc_1\to\cald_1$ are continuous group homomorphisms.
\end{definition}

We also have $2$-morphisms between $2$-group homomorphisms:

\begin{definition}[continuous $2$-group natural transformation]\label{Def:2GroupNat}
Let $\calc$ and $\cald$ be topological abelian $2$-groups and let $F,G:\calc\toto\cald$ be continuous $2$-group homomorphisms.
A {\it continuous $2$-group natural transformation} from $F$ to $G$ is a natural transformation $\eta:F\Rightarrow G$ such that the corresponding map $\eta:\calc_0\to\cald_1$ is a continuous group homomorphism.
\end{definition}

Thus, we can address the issue of what it means for two topological abelian $2$-groups to be isomorphic or equivalent:

\begin{definition}[isomorphic and topologically equivalent $2$-groups]\label{Def:IsoEquiv2Groups}
Two topological abelian $2$-groups $\calc$ and $\cald$ are {\it isomorphic $2$-groups} if there exist inverse continuous $2$-group homomorphisms $F:\calc\to\cald$ and $F^{-1}:\cald\to\calc$.
Two topological abelian $2$-groups $\calc$ and $\cald$ are {\it equivalent $2$-groups} if there exist continuous $2$-group homomorphisms $F:\calc\to\cald$ and $G:\cald\to\calc$, along with continuous $2$-group natural transformations $\eta:GF\Rightarrow 1_{\calc}$ and $\varepsilon:FG\Rightarrow 1_{\cald}$.
\end{definition}

With this, we can observe that all topological abelian $2$-groups are action groupoids as in Example \ref{Ex:ActionGroupoid} up to isomorphism.
More precisely, we have:

\begin{proposition}\label{Prop:TopAb2GroupGpoid}
Let $\calc$ be a topological abelian $2$-group with source map $s$, target map $t$, and unit map $u$.
Then $\calc$ is isomorphic as a topological abelian $2$-group to the action groupoid $\ker s \ltimes \calc_0$, where $\ker s$ is the kernel of the source map acting on the space of objects $\calc_0$ by:
\begin{equation*}
\psi\cdot x := t(\psi) + x, \qquad
\psi\in\ker s,\, x\in \calc_0.
\end{equation*}
Thus, in particular, all topological abelian $2$-groups are groupoids.
\end{proposition}

\begin{proof}
We explicitly construct the inverse functors $F:\calc \to \ker s\ltimes \calc_0$ and $G:\ker s \ltimes \calc_0 \to \calc$.
At the level of objects both $F_0$ and $G_0$ are the identity maps.
At the level of morphisms we have:
\begin{equation*}
F_1:\calc_1\to\ker s \times \calc_0, \qquad
\psi \mapsto \Big( \psi - u(s(\psi)), s(\psi) \Big),
\end{equation*}
and
\begin{equation*}
G_1:\ker s \times \calc_0 \to \calc_1, \qquad
(\psi,x) \mapsto \psi + u(x).
\end{equation*}
A direct computation shows that these are continuous group homomorphisms and inverses.
\end{proof}

We will need the following lemma:

\begin{lemma}\label{Lemma:QuotientMapIsOpen}
Let $\calc$ be a topological abelian $2$-group and let $\calc_0/\calc_1$ be the quotient space, consisting of isomorphism classes, equipped with the quotient topology.
Then the quotient map $\pi:\calc_0\to\calc_0/\calc_1$ is an open map.
\end{lemma}

\begin{proof}
Let $U$ be an open subset of the space $\calc_0$.
We want to show that the image $\pi(U)$ is open, or equivalently that the set $\pi^{-1}(\pi(U))$ is open.
Note that for all objects $y\in \calc_0$ the translations:
\[
T_y:\calc_0\to \calc_0,
\qquad
x \mapsto x + y
\]
are continuous since the addition is continuous.
In fact, the translations are homeomorphisms, since the inverse of a translation $T_y$ is given by the translation $T_{-y}$.
We claim that the set $\pi^{-1}(\pi(U))$ is a union of translations of $U$, and hence open.
More precisely, we claim that:
\begin{equation}\label{Eq:Translations}
\pi^{-1}(\pi(U))
=\bigcup_{\psi\in\ker s }T_{t(\psi)}(U),
\end{equation}
meaning the set $\pi^{-1}(\pi(U))$ is open, and therefore the set $\pi(U)$ is also open.
Thus, it suffices to verify the equality (\ref{Eq:Translations}).
For this, let $y\in \pi^{-1}(\pi(U))$ then there exists $x\in U$ and $\varphi\in \calc_1$ such that $s(\varphi)=x$ and $t(\varphi)=y$.
Note that $\psi:=\varphi-u(s(\varphi)$ is in the kernel $\ker s$.
Furthermore, note that:
\begin{equation*}
T_{t(\psi)}(x)
=t(\psi) + x
=t(\varphi)-t(u(s(\varphi)))+x
=y-x+x
=y,
\end{equation*}
and so $y\in T_{t(\psi)}(U)$.
This means that $y$ is in the union of the right hand side of (\ref{Eq:Translations}).
Conversely, suppose $y$ is an arbitrary element of the union of the right hand side of (\ref{Eq:Translations}). 
That means $y=t(\psi)+x$ for some $x\in U$ and some $\psi\in \ker s$.
Consider the morphism $\varphi:=\psi+u(x)$, and note that $s(\varphi)=x$, since $\psi\in\ker s$, and $t(\varphi)=t(\psi)+x=y$.
Hence, the objects $x$ and $y$ have the same isomorphism class, meaning that $y\in\pi^{-1}(\pi(U))$ since $x\in\pi^{-1}(\pi(U))$.
Consequently, equation (\ref{Eq:Translations}) holds, so we conclude that the map $\pi$ is open as claimed.
\end{proof}

%We recall the definition of a residual subset:

%\begin{definition}\label{Def:ResidualGeneric}
%A subset $\calu$ of a topological space $\caly$ is {\it residual} if it is the countable intersection of open and dense subsets of $\caly$.
%\end{definition}

We conclude this subsection by proving that isomorphisms and equivalences of topological abelian $2$-groups ``preserve'' open and dense subsets of the space of objects.

\begin{remark}
Given a topological abelian $2$-group $\calc$ and a subset $\calu$ of the space of objects $\calc_0$, we can consider the collection of all objects that are isomorphic to objects in the set $\calu$.
Since all topological abelian $2$-groups are groupoids, this amounts to considering the set $t(s^{-1}(\calu))$.
If we are considering a continuous action groupoid $A_1\ltimes A_0$ as in Example \ref{Ex:ActionGroupoid}, then this is the same as considering $A_1 \cdot \calu$.
\end{remark}

\begin{lemma}\label{Lemma:IsoResidual}
Let $\calc$ be a topological abelian $2$-group and let $\calu$ be an open and dense subset of the collection of objects $\calc_0$, and let $s$ and $t$ denote the source and target maps respectively.
Then the collection $\calc_1\cdot \calu:=t(s^{-1}(\calu))\subseteq \calc_0$ of all objects isomorphic to objects in $\calu$ is also open and dense.
\end{lemma}

\begin{proof}
Let $\pi:\calc_0\to\calc_0/\calc_1$ be the quotient map.
Note that:
\begin{equation*}
\calc_1\cdot \calu=t(s^{-1}(\calu))=\pi^{-1}(\pi(\calu)).
\end{equation*}
Thus, this set is open by the continuity and the openness of $\pi$ (Lemma \ref{Lemma:QuotientMapIsOpen}).
To see that it is dense, note that:
\begin{align*}
\overline{\calc_1\cdot \calu}
&= \overline{\pi^{-1}(\pi(\calu))}\\
&=\pi^{-1}\left(\overline{\pi(\calu)}\right)
\qquad \text{ since }\pi\text{ is open}\\
&\supseteq\pi^{-1} \left(\pi\left(\,\overline{\calu}\,\right)\right)
\qquad \text{ since }\pi\text{ is continuous}\\
&=\pi^{-1}(\pi(\calc_0))\\
&=\calc_0.
\end{align*}
Hence, $\calc_1\cdot \calu$ is dense in in $\calc_0$ as claimed.
\end{proof}

\begin{theorem}\label{Thm:GeneralResidualPreservation}
Let $\calc$ and $\cald$ be topological abelian $2$-groups and let $E:\calc\to \cald$  and $P:\cald\to\calc$ be (part of) an equivalence between them (Definition \ref{Def:IsoEquiv2Groups}).
If $\calu$ is an open and dense subset of space of objects $\calc_0$, then $t(s^{-1}(E_0(\calu)))$ is an open and dense subset of the space of objects $\cald_0$.
\end{theorem}

\begin{proof}
First, we observe that, since $PE\cong 1_\calc$ and $EP\cong 1_\cald$, the functors $E$ and $P$ determine inverse homeomorphisms $[E]$ and $[P]$ between the quotient spaces $\calc_0/\calc_1$ and $\cald_0/\cald_1$ as shown in the following diagram:
\begin{equation}\label{Diag:QuotientHomeomorphisms}
\begin{gathered}
\xy
(-36,8)*+{\calc_0}="1";
(-12,8)*+{\cald_0}="2";
(-36,-8)*{\calc_0/\calc_1}="3";
(-12,-8)*+{\cald_0/\cald_1}="4";
{\ar@{->}^{E_0} "1";"2"};
{\ar@{->}_{\pi_\calc} "1";"3"};
{\ar@{->}^{\pi_\cald} "2";"4"};
{\ar@{->}_{[E]} "3";"4"};
%%%%%%%%%%%%%%%%%%%%%%%%
(12,8)*+{\cald_0}="1";
(36,8)*+{\calc_0}="2";
(12,-8)*+{\cald_0/\cald_1}="3";
(36,-8)*{\calc_0/\calc_1}="4";
{\ar@{->}^{P_0} "1";"2"};
{\ar@{->}_{\pi_\cald} "1";"3"};
{\ar@{->}^{\pi_\calc} "2";"4"};
{\ar@{->}_{[P]} "3";"4"};
\endxy
\end{gathered}
\end{equation}
where the vertical maps are the quotient maps.

Now observe that, as in the proof of Lemma \ref{Lemma:QuotientMapIsOpen}, we have the equality:
\begin{equation*}
t(s^{-1}(E_0(U)))=\pi_\cald^{-1}(\pi_\cald(E_0(U))).
\end{equation*}
Since the quotient map is open by Lemma \ref{Lemma:QuotientMapIsOpen} and the map $[E]$ is a homeomorphism, the set $\pi_\cald E_0(U)=[E]\pi_\calc(U)$ is open.
Hence, the set $\pi_\cald^{-1}(\pi_\cald(E_0(U)))$ is open by the continuity of $\pi_\cald$.

Finally, to see that the set $t(s^{-1}(E_0(U)))$ is dense, note that:
\begin{align*}
\overline{t(s^{-1}(E_0(U)))}
&=\overline{\pi_\cald^{-1}(\pi_\cald(E_0(U)))}\\
&=\pi_\cald^{-1}\left(\overline{\pi_\cald(E_0(U))}\right)
\qquad\text{ since $\pi_\cald$ is an open map}\\
&=\pi_\cald^{-1}\left(\overline{[E](\pi_\calc(U))}\right)
\qquad\text{ since }\pi_\cald E_0=[E]\pi_\calc\\
&=\pi_\cald^{-1}\left([E]\left(\overline{\pi_\calc(U)}\right)\right)
\qquad\text{ since $[E]$ is a homeomorphism}\\
&\supseteq\pi_\cald^{-1}\left([E]\left(\pi_\calc\left(\,\overline{U}\,\right)\right)\right)
\qquad\text{ by the continuity of the map $\pi_\calc$}\\
&=\pi_\cald^{-1}\left([E]\left(\pi_\calc\left(\calc_0\right)\right)\right)
\qquad\text{ since $U$ is dense in $\calc_0$}\\
&=\pi_\cald^{-1}\left([E]\left(\calc_0/\calc_1\right)\right)
\qquad\text{ by the surjectivity of $\pi_\calc$}\\
&=\pi_\cald^{-1}\left(B_0/B_1\right)
\qquad\text{ since $[E]$ is a homeomorphism}\\
&=B_0.
\end{align*}
Since the converse inclusion is trivial, the set is dense as claimed.
\end{proof}

\begin{remark}
Let $A_1\ltimes A_0$ and $B_1\ltimes B_0$ be two topological abelian $2$-groups that are action groupoids as described in example \ref{Ex:ActionGroupoid}.
Suppose that $E:A_1\ltimes A_0 \to B_1\ltimes B_0$ is (part of) an equivalence between them (Definition \ref{Def:IsoEquiv2Groups}).
Theorem \ref{Thm:GeneralResidualPreservation} says that if $\calu$ is an open and dense subset of the space $A_0$, then $B_1\cdot E_0(\calu)$ is an open and dense subset of $B_0$.
\end{remark}

\subsection{Open and dense collections of equivariant vector fields}\label{opendense}
In this subsection we prove that the category $\bbx(M)^G$ of equivariant vector fields on a proper $G$-manifold $M$ is a topological abelian $2$-group when equipped with the Whitney $C^\infty$ topology (Theorem \ref{Thm:TopAb2VFs}).
The same is true of the categories of paths of equivariant vector fields.
We then prove the first of the two main theorems of this section: that isomorphisms preserve open and dense subsets of equivariant vector fields (Theorem \ref{Thm:MainResidual1}).
Furthermore, the equivalence in Theorem \ref{Thm:EquivalenceVFs} is an equivalence of topological abelian $2$-groups  (Theorem \ref{Thm:EquivalenceContinuous}).
We use this to prove the second of the main theorems of this section: that the equivalence in Theorem \ref{Thm:EquivalenceVFs} preserves open and dense subsets of equivariant vector fields up to isomorphism (Theorem \ref{Thm:MainResidual2}).

Consider a proper $G$-manifold $M$.
We equip the space of equivariant vector fields $\ffX(M)^G$ and the space of infinitesimal gauge transformations $C^\infty(M,\ffg)^G$ with the Whitney $C^\infty$ topology (Definition \ref{Def:WhitneyTop}).

\begin{remark}
All of the results of this subsection apply just as well to the space of paths of equivariant vector fields $C^\infty(\bbr,\ffX(M)^G)$, the space of paths of infinitesimal gauge transformtiosn $C^\infty(\bbr,C^\infty(M,\ffg)^G)$, and the category of paths of equivariant vector fields $C^\infty(\bbr,\bbx(M)^G)$ of subsection \ref{paths}.
We equip these spaces with Whitney topologies by using the identification:
\begin{equation*}
C^\infty(\bbr,\ffX(M)^G) \cong \Gamma(TM \to \bbr\times M)^G
\end{equation*}
and the identification:
\begin{equation*}
C^\infty(\bbr,C^\infty(M,\ffg)^G) \cong C^\infty(\bbr\times M, \ffg)^G.
\end{equation*}
Then the proofs of all results in this subsection are completely analogous.
\end{remark}

Note that the scalar multiplication in the space $\ffX(M)^G$ need not be continuous (see the discussion after the proof of Proposition~3.5 in \cite[pp.~46-47]{GG73}).
Hence, the space $\ffX(M)^G$ is {\it not} a topological vector space.
The same observation applies to the space $C^\infty(M,\ffg)^G$.
As the following lemma shows, the addition and inversion maps on $\ffX(M)^G$ and $C^\infty(M,\ffg)^G$ are continuous:

\begin{lemma}\label{Lemma:PathSpacesAreTopAbGroups}
Let $M$ be a $G$-manifold.
The space of equivariant vector fields $\ffX(M)^G$ and the space of infinitesimal gauge transformations $C^\infty(M,\ffg)^G$ are topological abelian groups when equipped with the wWhitney $C^\infty$ topology (Definition \ref{Def:WhitneyTop}).
\end{lemma}

\begin{proof}
We consider the case of the space $\ffX(M)^G$.
The case of the space $C^\infty(M,\ffg)^G$ is analogous by thinkng of the Lie algebra $\ffg$ as a vector bundle over a point.
It suffices to prove that the addition and additive inverse maps on the vector space $\ffX(M)^G$ are continuous.
For the addition note that the following diagram commutes:
\begin{equation}\label{Diag:AdditionDiag}
\xy
(-32,10)*+{
\ffX(M)^G\times \ffX(M)^G
}="1";
(32,10)*+{
\ffX(M)^G
}="2";
(-32,-10)*{
C^\infty(M,TM)\fp{C^\infty(M, M)} C^\infty(M,TM)
}="3";
(32,-10)*+{
C^\infty(M,TM)
}="4";
{\ar@{->}^{\phantom{++++++}+} "1";"2"};
{\ar@{->}_{} "1";"3"};
{\ar@{->}^{} "2";"4"};
{\ar@{->}_{\phantom{++++++}+_*} "3";"4"};
\endxy
\end{equation}
where the vertical maps are subspace inclusions, the top map $+$ is the desired addition map, the pushforward $+_*$ is the pushforward of the fiberwise addition $+:TM\times_M TM\to TM$ with the domain of the pushforward identified with the fiber product via the canonical homeomorphism:
\begin{equation}\label{Eq:HomeoFiberProduct}
C^\infty(M,TM)\,\,\fp{C^\infty(M,M)}\,\, C^\infty(M,TM)
\cong
C^\infty\left(M,TM\times_M TM\right)
\end{equation}
of part (3) of Lemma \ref{Lemma:WhitneyProperties}.
The continuity of the addition $+$ now follows by the universal property of the subspace topology, the commutativity of diagram (\ref{Diag:AdditionDiag}), and the fact that the composition along the left and bottom of this diagram is continuous.

For the additive inverse, note that the following diagram commutes:
\begin{equation}\label{Diag:InverseDiag}
\xy
(-26,10)*+{
\ffX(M)^G
}="1";
(26,10)*+{
\ffX(M)^G
}="2";
(-26,-10)*{
C^\infty(M,TM)
}="3";
(26,-10)*+{
C^\infty(M,TM)
}="4";
{\ar@{->}^{-} "1";"2"};
{\ar@{->}_{} "1";"3"};
{\ar@{->}^{} "2";"4"};
{\ar@{->}_{-_*} "3";"4"};
\endxy
\end{equation}
where the vertical maps are inclusions, the top map is the desired additive inverse map, and the bottom map is the pushforward $-_*$ of the fiberwise additive inverse map $-:TM\to TM$.
The continuity of the additive inverse $-$ now follows by the universal property of the subspace topology, the commutativity of diagram (\ref{Diag:InverseDiag}, the continuity off the pushforward $-_*$ by part (1) of Lemma \ref{Lemma:WhitneyProperties}, and the continuity of the inclusion on the left.
\end{proof}

With this we can prove that the category of equivariant vector fields is a topological abelian $2$-group (Definition \ref{Def:TopAb2Group}).

\begin{theorem}\label{Thm:TopAb2VFs}
Let $M$ be a $G$-manifold.
The category $\bbx(M)^G$ of equivariant vector fields on $M$ (Definition \ref{Def:GpoidVFs}) is a topological abelian $2$-group when equipped with the Whitney topologies (Definition \ref{Def:WhitneyTop}).
\end{theorem}

\begin{proof}
Recall that the category of equivariant vector fields $\bbx(M)^G$ is an action groupoid, like the one in Example \ref{Ex:ActionGroupoid}, induced by the map:
\begin{equation}\label{Eq:ThePartialMap}
\partial:C^\infty(M,\ffg)^G \to \ffX(M)^G
\end{equation}
defined by (\ref{Eq:InducedVF}).
By Lemma \ref{Lemma:PathSpacesAreTopAbGroups}, we know that the domain and target of $\partial$ are topological abelian groups.
Thus, to verify that the category $\bbx(M)^G$ is a topological abelian $2$-group, by the observations in Example \ref{Ex:ActionGroupoid}, it suffices to verify that the map (\ref{Eq:ThePartialMap}) is a continuous group homomorphism.

To verify this, we prove that the following diagram commutes:
\begin{equation}\label{Diag:PartialDiag}
\begin{gathered}
\xy
(-30,10)*+{
C^\infty(M,\ffg)^G
}="1";
(30,10)*+{
\ffX(M)^G
}="2";
(-30,-10)*{
C^\infty\left(M,\ffg\right)\times
C^\infty\left(M,M\right)
}="3";
(30,-10)*+{
C^\infty\left(M,TM\right)
}="4";
{\ar@{->}^{\phantom{FFF}\partial} "1";"2"};
{\ar@{->} "1";"3"};
{\ar@{->} "2";"4"};
{\ar@{->}_{\phantom{FFFF}a_*} "3";"4"};
\endxy
\end{gathered}
\end{equation}
where the top map is the boundary map, the left-hand map is the inclusion defined by $\psi \mapsto (\psi,\text{id}_M)$, the right-hand map is the obvious inclusion, and the bottom map is the pushforward of the map:
\begin{equation}\label{Eq:InfinitesimalAction}
a:\ffg\times M \to TM,
\qquad
(\xi,m) \mapsto \frac{\d}{\d\tau}\Big|_0\exp(\tau\xi)\cdot m,
\end{equation}
where we have also used part (3) of Lemma \ref{Lemma:WhitneyProperties} to write the domain as a product.

Note that the map $a$ is smooth since it is obtained by differentiating the action $G\times M\to M$, with respect to the $G$-variables only, at the identity of $G$.
Hence, the pushforward $a_*$ is continuous by part (1) of Lemma \ref{Lemma:WhitneyProperties}.
On the other hand, the inclusion on the left-hand side of diagram (\ref{Diag:PartialDiag}) is continuous since it is the product of the inclusion $C^\infty(M,\ffg)^G \hookrightarrow C^\infty(M, \ffg)$ and the constant map:
\begin{equation*}
C^\infty(M,\ffg)^G \to C^\infty(M,M),
\qquad
\psi\mapsto \text{id}_M.
\end{equation*}
The continuity of the map $\partial$ now follows by the universal property of the subspace topology, the commutativity of diagram (\ref{Diag:PartialDiag}), and the fact that the composition along the left and bottom of this diagram is continuous. 
Thus, the category $\bbx(M)^G$ is a topological abelian $2$-group as claimed.
\end{proof}

Therefore, we obtain:

\begin{theorem}\label{Thm:MainResidual1}
Let $M$ be a proper $G$-manifold, let $\bbx(M)^G$ be the category of euqivariant vector fields on $M$, and let $\calu$ be an open and dense subset of the space of equivariant vector fields $\ffX(M)^G$.
Then the collection:
\begin{equation*}
C^\infty(M,\ffg)^G\cdot \calu \subseteq \ffX(M)^G
\end{equation*}
of all equivariant vector fields isomorphic to vector fields in $\calu$ is also open and dense in $\ffX(M)^G$.
\end{theorem}

\begin{proof}
Since the category $\bbx(M)^G$ is a topological abelian $2$-group by Theorem \ref{Thm:TopAb2VFs}, the result follows immediately by Lemma \ref{Lemma:IsoResidual}.
\end{proof}

We now turn our attention to the equivalence of Theorem \ref{Thm:EquivalenceVFs}.
For this, let $V$ be a representation of a compact Lie subgroup $K$ of a Lie group $G$, and consider the associated bundle $G\times^KV$.
In particular, we want to prove the following theorem:

\begin{theorem}\label{Thm:EquivalenceContinuous}
Let $V$ be a representation of a compact Lie subgroup $K$ of a Lie group $G$.
Then the equivalence of Theorem \ref{Thm:EquivalenceVFs}:
\begin{equation*}
\bbx(G\times^KV)^G \simeq \bbx(V)^K
\end{equation*}
between the categories of equivariant vector fields on the associated bundle $G\times^KV$ and the $K$-representation $V$ is an equivalence of topological abelian $2$-groups with respect to the Whitney $C^\infty$ topology.
\end{theorem}

Recall that the equivalence $\bbx(G\times^KV)^G\simeq \bbx(V)^K$ is given by a functor:
\begin{equation*}
E:\bbx(V)^K\to\bbx(G\times^KV)^G,
\end{equation*}
defined by equivariant extension on both objects and morphisms (Theorem \ref{Thm:InclusionFunctor}), a projection functor:
\begin{equation*}
P:\bbx(G\times^KV)^G\to\bbx(V)^K,
\end{equation*}
corresponding to a choice of equivariant connection on the vector bundle $G\times^KV\to G/K$ (Theorem \ref{Thm:ProjectionFunctor}), and a natural isomorphism $h:EP\cong 1_{\bbx(G\times^KV)^G}$ (Theorem \ref{Thm:NaturalIsomorphism}).
Recall that the composition $P\circ E$ is the identity by Theorem \ref{Thm:ExtendThenProject}.
Therefore, to prove that the equivalence $\bbx(G\times^KV)^G\simeq \bbx(V)^K$ is an equivalence of topological abelian $2$-groups it suffices to prove that the functors $E$ and $P$ are continuous $2$-group homomorphisms (Definition \ref{Def:2GroupHom}) and that the natural isomorphism $h$ is a continuous $2$-group natural isomorphism (Definition \ref{Def:2GroupNat}).
That all the maps involved are homomorphisms follows immediately from the definitions, thus it suffices to check that the maps are continuous.
This is what we do now in Proposition \ref{Prop:ECts}, Proposition \ref{Prop:PCts}, and Proposition \ref{Prop:hCts}.

\begin{proposition}\label{Prop:ECts}
Let $V$ be a representation of a compact Lie subgroup $K$ of a Lie group $G$.
The canonical functor $E:\bbx(V)^K\to \bbx(G\times^KV)^G$ of Theorem \ref{Thm:InclusionFunctor} is continuous with respect to the Whitney $C^\infty$ topology.
\end{proposition}

\begin{proof}
Let $j:V\hookrightarrow G\times^KV$ be the $K$-equivariant embedding defined by $j(v):=[1,v]$.
First, we check that the map on objects:
\begin{equation*}
E_0:\ffX(V)^K\to \ffX(G\times^KV)^G
\end{equation*}
is continuous.
For this, note that the map $E_0$ factors as in the following diagram:
\begin{equation*}
\xy
{(-32,-14)}*+{\ffX(V)^K} = "1";
{(32,-14)}*+{\ffX(G\times^KV)^G} = "2";
{(-32,0)}*+{C^\infty(V,TV)^K} = "3";
{(32,0)}*+{C^\infty\left(G\times^KV,T\left(G\times^KV\right)\right)^G} = "4";
{(0,18)}*+{C^\infty(V,\,\,T\left(G\times^KV\right))^K} = "5";
{\ar@{->}_{E_0} "1";"2"};
{\ar@{->} "1";"3"};
{\ar@{->} "2";"4"};
{\ar@{->}^{\left(T j\right)_*} "3";"5"};
{\ar@{->}^{\epsilon} "5";"4"};
\endxy
\end{equation*}
Here the vertical maps are subspace inclusions, the map $\left(T j\right)_*$ is the pushforward by the tangent map of the slice embedding $j: V\hookrightarrow G\times^KV$, and the map $\epsilon$ is the corresponding equivariant extension (Remark \ref{Rem:EquivariantExtension}).
The continuity of the map $E_0$ follow from the fact that the maps $\left(T j\right)_*$and $\epsilon$ are continuous (see part (1) of Lemma \ref{Lemma:WhitneyProperties} and Corollary \ref{Cor:EquivariantExtensionContinuous}), and the inclusions are subspace inclusions.

Since the categories of equivariant vector fields are action groupoids, to prove $E$ is continuous on morphisms, it suffices to check that the map on infinitesimal gauge transformations:
\begin{equation*}
E_1: C^\infty(V,\ffk)^K \to C^\infty(G\times^KV,\ffg)^G,
\end{equation*}
is continuous (see Theorem \ref{Thm:InclusionFunctor}).
For this, note that $E_1$ factors as in the following diagram:
\begin{equation*}
\xy
{(-30,0)}*+{C^\infty( V, \ffk)^K} = "1";
{(30,0)}*+{C^\infty( G\times^KV,\ffg)^G} = "2";
{(0,14)}*+{C^\infty( V, \ffg)^K} = "3";
{\ar@{->}_{E_1} "1";"2"};
{\ar@{->}^{\iota_*} "1";"3"};
{\ar@{->}^{\epsilon} "3";"2"};
\endxy
\end{equation*}
where $\iota_*$ is the pushforward of the inclusion $\iota:\ffk\hookrightarrow\ffg$, and the map $\epsilon$ is the equivariant extension map (Remark \ref{Rem:EquivariantExtension}).
The maps $\iota_*$ are $\epsilon$ are continuous (see part (1) of Lemma \ref{Lemma:WhitneyProperties} and Corollary \ref{Cor:EquivariantExtensionContinuous}).
Hence, the map $E_1$ is continuous.
Since the functor $E$ is continuous on both morphisms and objects, the functor $E$ is continuous with respect to the Whitney topologies.
\end{proof}

We now verify that any choice of functor $P$ as in Theorem \ref{Thm:ProjectionFunctor} is continuous:

\begin{proposition}\label{Prop:PCts}
Let $V$ be a representation of a compact Lie subgroup $K$ of a Lie group $G$.
Any choice of functor $P:\bbx(G\times^KV)^G\to\bbx(V)^K$ as in Theorem \ref{Thm:ProjectionFunctor} is continuous with respect to the Whitney $C^\infty$ topology.
\end{proposition}

\begin{proof}
First, we check that the map on objects:
\begin{equation*}
P_0: \ffX(G\times^KV)^G \to \ffX(V)^K
\end{equation*}
is continuous.
For this, note that the map $P_0$ factors as in the following diagram:
\begin{equation}\label{Diag:ProjectionMap}
\begin{gathered}
\xy
{(-28,0)}*+{\ffX(G\times^KV)^G} = "1";
{(28,0)}*+{\ffX(V)^K} = "2";
{(0,20)}*+{\Gamma(\calv(\bbr\times G \times^KV))} = "3";
{\ar@{->}_{P_0} "1";"2"};
{\ar@{->}^{\Phi_*} "1";"3"};
{\ar@{->}^{(T j)|^{-1}\circ\,(-)\,\circ\,  j} "3";"2"};
\endxy
\end{gathered}
\end{equation}
where the map $\Phi_*$ is the pushforward of the connection:
\begin{equation*}
\Phi:T\left( G\times^KV\right) \to \calv\left( G\times^KV\right)
\end{equation*}
on the associated bundle $\bbr\times G\times^KV\to G/K$ corresponding to $P$ (see Theorem \ref{Thm:ProjectionFunctor}), the space $\calv(G\times^KV)$ is the vertical bundle of $G\times^KV\to G/K$ corresponding to the connection $\Phi$, and the map $(T j)|^{-1}\circ\,(-)\,\circ\, j$ consists of pullback by the slice embedding $ j: V \hookrightarrow G\times^KV$ and pushforward by the inverse of the restriction of the tangent map $T j$ to a bundle map $T V\to \calv\left( G\times^KV\right)|\, V$ (here $\calv( G\times^KV)|\, V$ denotes the restriction of the vertical bundle of $ G\times^KV\to G/K$ to the slice $V$).

The pushforward $\Phi_*$ is continuous by part (1) of Lemma \ref{Lemma:WhitneyProperties}, the pushforward $\left((T j)|^{-1}\right)_*$ is continuous by part (1) of Lemma \ref{Lemma:WhitneyProperties}, and the pullback $ j^*$ is continuous by Lemma \ref{Lemma:ProperPullbacks} since the map $ j$ is a closed embedding (hence proper).
Consequently, the map $P_0$ is a continuous map.

Since the categories of equivariant vector fields are action groupoids, to prove $P$ is continuous on morphisms, it suffices to check that the map on infinitesimal gauge transformations:
\begin{equation*}
P_1:C^\infty(G\times^KV,\ffg)^G_1 \to C^\infty(V,\ffk)^K
\end{equation*}
is continuous (see Theorem \ref{Thm:ProjectionFunctor} for the definition of $P_1$).
For this, note that $ P_1$ factors as in the following diagram:
\begin{equation}\label{Diag:ProjectionGaugePaths}
\begin{gathered}
\xy
{(-28,0)}*+{C^\infty\left(G\times^KV,\ffg\right)^G} = "1";
{(28,0)}*+{C^\infty(V, \ffk)^K} = "2";
{(-28,15)}*+{C^\infty\left(G\times^KV,\ffg\right)} = "5";
{(28,15)}*+{C^\infty(V, \ffg)^K} = "6";
{(0,30)}*+{C^\infty(G\times V, \ffg)^K} = "7";
{\ar@{->}_{P_1} "1";"2"};
{\ar@{->}^{I} "1";"5"};
{\ar@{->}^{\bbp_*} "6";"2"};
{\ar@{->}^{\pi^*} "5";"7"};
{\ar@{->}^{ j^*} "7";"6"};
\endxy
\end{gathered}
\end{equation}
where the map $I$ is the inclusion, the map $\pi^*$ is the pullback via the principal bundle projection $\pi: G \times V \to G \times^KV$, the map $j^*$ is the pullback via the slice embedding $j : V \hookrightarrow G\times^KV$, and the map $\bbp_*$ is the pushforward of the projection $\bbp:\ffg\to\ffk$.
The map $ P_1$ is continuous since each of the other maps in diagram (\ref{Diag:ProjectionGaugePaths}) is such.
In particular, the pullback $\pi^*$ is continuous by Lemma \ref{Lemma:WhitneyEquivariantProperty}, the pullback $ j^*$ is continuous by Lemma \ref{Lemma:ProperPullbacks} since the map $ j$ is a closed embedding (hence proper), and the pushforward $\bbp_*$ is continuous by part (1) of Lemma \ref{Lemma:WhitneyProperties}.
Since the functor $P$ is continuous on both morphisms and objects, the functor $P$ is continuous with respect to the Whitney topology.
\end{proof}

Next, we deal with the natural isomorphism:

\begin{proposition}\label{Prop:hCts}
Let $V$ be a representation of a compact Lie subgroup $K$ of a Lie group $G$.
Let $E:\bbx(V)^K\to \bbx(G\times^KV)^G$ be the canonical functor in Theorem \ref{Thm:InclusionFunctor}, let $P:\bbx(G\times^KV)^G\to\bbx(V)^K$ be any choice of functor as in Theorem \ref{Thm:ProjectionFunctor}, and let $h:EP \cong 1_{\bbx(G\times^KV)^G}$ be the corresponding natural isomorphism of Theorem \ref{Thm:NaturalIsomorphism}.
Then the corresponding map:
\begin{equation}\label{Eq:NatIsoMapCts}
h:\ffX(G\times^KV)^G\to C^\infty(G\times^KV,\ffg)^G
\end{equation}
is continuous.
\end{proposition}

\begin{proof}
Recall from the proof of Theorem \ref{Thm:NaturalIsomorphism} that the map $h$ factors as in the following diagram:
\begin{equation*}
\begin{gathered}
\xy
{(-20,0)}*+{\ffX(G\times^KV)^G} = "1";
{(20,0)}*+{C^\infty(G\times^KV,\ffg)^G} = "2";
{(-20,15)}*+{\Gamma(\calh)^G} = "3";
{(20,15)}*+{C^\infty(V,\ffg)^K} = "4";
{(0,30)}*+{C^\infty(V,\ffq)^K} = "5";
{\ar@{->}^{h} "1";"2"};
{\ar@{->}^{\alpha} "1";"3"};
{\ar@{->}^{\beta} "3";"5"};
{\ar@{->}^{\iota_*} "5";"4"};
{\ar@{->}^{\epsilon} "4";"2"};
\endxy
\end{gathered}
\end{equation*}
where $\ffg=\ffk\oplus\ffq$ is the $K$-equivariant splitting corresponding to $P$, the bundle $\calh\to G\times^KV$ is the horizontal bundle of $G\times^KV\to G/K$ corresponding to the connection on $G\times^KV\to G/K$ determined by $\ffg=\ffk\oplus\ffq$, and the maps $\alpha,\beta, \iota_*,$ and $\epsilon$ are defined in the proof of Theorem \ref{Thm:NaturalIsomorphism}.
The map $\alpha$ is continuous since the addition in the group $\ffX(G\times^KV)^G$ is continuous and the maps $E_0$ and $P_0$ are continuous.
The map $\iota_*$ is continuous by part (1) of Lemma \ref{Lemma:WhitneyProperties}.
The map $\epsilon$ is continuous by Corollary \ref{Cor:EquivariantExtensionContinuous}
Hence, it remains to show that the map $\beta$ is continuous, and we will have shown that the map $h$ is continuous.

Recall that the map $\beta$ is the inverse of the map in the statement of Lemma \ref{Lemma:HorizontalTriviality}.
Thus, it suffices to check that the map $\beta^{-1}$ is a homeomorphism.
Recall from the proof of Lemma \ref{Lemma:HorizontalTriviality} that the map $\beta^{-1}$ factors as in diagram (\ref{Diag:HorizontalTriviality}), which we reproduce here for convenience:
\begin{equation*}
\begin{gathered}
\xy
{(-20,10)}*+{C^\infty(V, \ffq)^K} = "1";
{(20,10)}*+{\Gamma(\calh)^G} = "2";
{(-20,-10)}*+{\Gamma(V\times \ffq)^K} = "3";
{(20,-10)}*+{\Gamma(\calh \,|\, V)^K} = "4";
{\ar@{->}^{\beta^{-1}} "1";"2"};
{\ar@{->}_{1_{V}\times\, (-)} "1";"3"};
{\ar@{->}_{\eta_*} "3";"4"};
{\ar@{->}_{\epsilon} "4";"2"};
\endxy
\end{gathered}
\end{equation*}
Recall that the trivial bundle $ V\times \ffq \xrightarrow{\pr_V} V$ with the action of the group $K$ on the total space given by:
\begin{equation*}
k\cdot (v,\xi) := (k\cdot v, \Ad(k)\xi), \qquad
k\in K,\, (v, \xi) \in  V\times\ffq,
\end{equation*}
and the action on the base given by:
\begin{equation*}
k\cdot v :=  k\cdot v, \qquad
k\in K,\, v \in V.
\end{equation*}
Thus, the map $1_V\times (-)$ is a homeomorphism since its inverse is given by the pushforward of the bundle projection $\pr_V: V\times\ffq\to V$, which is continuous by part (1) of Lemma \ref{Lemma:WhitneyProperties}.
Recall that the map $\eta$ is a diffeomorphism.
Hence, the pushforward $\eta_*$ of this map is a homeomorphism by part (1) of Lemma \ref{Lemma:WhitneyProperties}.
Recall that the map $\epsilon$ is the restriction of the corresponding equivariant extension map as in Remark \ref{Rem:EquivariantExtension}.
Hence, it is continuous by Corollary \ref{Cor:EquivariantExtensionContinuous}.
The inverse of the map $\epsilon$ is the pullback by the slice embedding $ j: V\hookrightarrow G\times^KV$ of the associated bundle $ G\times^KV\to G/K$.
This pullback is continuous by Lemma \ref{Lemma:ProperPullbacks}.
Hence, the map $\epsilon$ is also a homeomorphism.
Since the top vertical maps are topological identifications, the map $h$ is continuous.
\end{proof}

We can now prove the second of the two main theorems of this section, that the equivalence $\bbx(G\times^KV)^G\simeq \bbx(V)^K$ of Theorem \ref{Thm:EquivalenceVFs} preserves open and dense subsets of equivariant vector fields up to isomorphism:

\begin{theorem}\label{Thm:MainResidual2}
Let $V$ be a representation of a compact Lie subgroup $K$ of a Lie group $G$.
Furthermore, let:
\begin{equation*}
E_0:\ffX(V)^K\hookrightarrow \ffX(G\times^KV)^G
\end{equation*}
be the canonical inclusion by equivariant extension (Theorem \ref{Thm:InclusionFunctor}).
Then if $\calu$ is an open and dense subset of $\ffX(V)^K$, the set:
\begin{equation*}
C^\infty(G,\ffg)^G\cdot E_0(\calu) \subseteq \ffX(G\times^KV)^G
\end{equation*}
of all equivariant vector fields on $G\times^KV$ isomorphic to a vector field in $E_0(\calu)$ is open and dense in $\ffX(G\times^KV)^G$.
Similarly, let:
\begin{equation*}
P_0:\ffX(G\times^KV)^G\to\ffX(V)^K
\end{equation*}
be a choice of projection of equivariant vector fields with respect to some equivariant connection on $G\times^KV\to G/K$ (Theorem \ref{Thm:ProjectionFunctor}).
Then if $\calu$ is an open and dense subset of $\ffX(G\times^KV)^G$, the set:
\begin{equation*}
C^\infty(V,\ffk)^K\cdot P_0(\calu) \subseteq \ffX(V)^K
\end{equation*}
of all equivariant vector fields on $V$ isomorphic to a vector field in $P_0(\calu)$ is open and dense in $\ffX(V)^K$.
\end{theorem}

\begin{proof}
Since the functors $P:\bbx(G\times^KV)^K\simeq \bbx(V)^K:E$ are part of an equivalence of topological abelian $2$-groups by Theorem \ref{Thm:EquivalenceContinuous}, the result follows by Theorem \ref{Thm:GeneralResidualPreservation}.
\end{proof}

%%%%%%%%%%%%%%%%%%%%%%%%%%%%%%%%
% references
%%%%%%%%%%%%%%%%%%%%%%%%%%%%%%%%
%
% Put references in BibTeX format in thesisrefs.bib.
\bibliography{refs}{}
\bibliographystyle{plain}

\end{document}